\documentclass[11pt,letterpaper]{amsart}

\usepackage[margin=1.08in]{geometry}
\usepackage{amsmath,amssymb,amsthm,mathtools}
\usepackage{microtype}
\usepackage[T1]{fontenc}
\usepackage{libertinus}
\usepackage{enumitem}
\usepackage{booktabs}
\usepackage{tikz}
\usepackage{pdflscape}
\usetikzlibrary{arrows.meta,positioning,calc}
\usepackage{xcolor}
\usepackage{array}
\usepackage{longtable}
\usepackage{hyperref}
\usepackage{aliascnt}
\usepackage[nameinlink,capitalize]{cleveref}

\hypersetup{
  colorlinks=true,
  linkcolor=black,
  citecolor=black,
  urlcolor=black,
  pdftitle={Typical Bohnenblust--Hille Ratios},
  pdfauthor={Daniel M. Pellegrino and Eduardo V. Teixeira}
}
\setlist{itemsep=0.25em,topsep=0.45em}
\numberwithin{equation}{section}

\newcommand{\CC}{\mathbb C}
\newcommand{\NN}{\mathbb N}
\newcommand{\RR}{\mathbb R}
\newcommand{\Nzero}{\mathbb N_0}

\newcommand{\TT}{\mathbb T}
\newcommand{\e}{\mathrm e}

\newaliascnt{proposition}{theorem}
\newtheorem{proposition}[proposition]{Proposition}
\aliascntresetthe{proposition}

\newaliascnt{lemma}{theorem}
\newtheorem{lemma}[lemma]{Lemma}
\aliascntresetthe{lemma}

\newaliascnt{corollary}{theorem}
\newtheorem{corollary}[corollary]{Corollary}
\aliascntresetthe{corollary}

\theoremstyle{definition}
\newaliascnt{definition}{theorem}

\aliascntresetthe{definition}

\theoremstyle{remark}
\newaliascnt{remark}{theorem}

\aliascntresetthe{remark}

\crefname{theorem}{theorem}{theorems}
\Crefname{theorem}{Theorem}{Theorems}
\crefname{proposition}{proposition}{propositions}
\Crefname{proposition}{Proposition}{Propositions}
\crefname{lemma}{lemma}{lemmas}
\Crefname{lemma}{Lemma}{Lemmas}
\crefname{corollary}{corollary}{corollaries}
\Crefname{corollary}{Corollary}{Corollaries}
\crefname{definition}{definition}{definitions}
\Crefname{definition}{Definition}{Definitions}
\crefname{remark}{remark}{remarks}
\Crefname{remark}{Remark}{Remarks}

\newcommand{\E}{\mathbb E}
\newcommand{\Prob}{\mathbb P}
\theoremstyle{plain}
\newtheorem{mainthm}{Theorem}

\crefname{mainthm}{theorem}{theorems}
\Crefname{mainthm}{Theorem}{Theorems}

\title[Typical Bohnenblust--Hille ratios]
{Typical Bohnenblust--Hille Ratios}

\author[D. M. Pellegrino]{Daniel M. Pellegrino}
\address{Departamento de Matem\'atica, Universidade Federal da Para\'iba, Jo\~ao Pessoa, PB, Brazil}
\email{daniel.pellegrino@academico.ufpb.br}

\author[E. V. Teixeira]{Eduardo V. Teixeira}
\address{Department of Mathematics, Oklahoma State University, Stillwater, OK 74078, USA}
\email{eduardo.teixeira@okstate.edu}

\begin{document}
\raggedbottom

{\color{black}
\begin{abstract}
The polynomial Bohnenblust--Hille inequality controls the coefficient
$\ell_{2m/(m+1)}$-norm of {\color{black}\hypersetup{linkcolor=black,citecolor=black}a complex} $m$-homogeneous polynomial by its supremum
norm, with a constant independent of the dimension. We study the associated
Bohnenblust--Hille ratio from a probabilistic point of view, by placing
normalized surface measure on the Euclidean coefficient sphere. Our results
reveal a sharp contrast between the extremal behavior governing the classical
Bohnenblust--Hille constants and the typical scale seen in coefficient
directions. For arbitrary prescribed monomial supports, the ratio is eventually
at most $1$ almost surely, and it tends to zero in spherical measure exactly
when the number of monomials tends to infinity. On the full complex polynomial
spaces we determine its typical asymptotic scale uniformly in the dimension;
in the critical regime $n_m/m\to1$ this gives
$2\sqrt2/\sqrt{m\log m}$, in contrast with the nonvanishing extremal scale.
{\color{black}\hypersetup{linkcolor=black,citecolor=black}On the full real polynomial spaces, the corresponding ratio tends to zero
in spherical measure exactly when the number of variables tends to infinity.}
\end{abstract}
}

\maketitle
\tableofcontents
\section{Introduction}\label{sec:introduction}

The Bohnenblust--Hille inequality entered analysis through a problem about
Dirichlet series.  Bohr asked for the largest possible width of a vertical
strip in which an ordinary Dirichlet series may converge uniformly without
converging absolutely.  He proved in 1913 that this width is at most $1/2$.
The problem remained open until 1931, when Bohnenblust and Hille proved that
$1/2$ is sharp \cite{BH}; their argument built on Littlewood's $4/3$
inequality \cite{Littlewood} and on dimension-free estimates for the
coefficients of multilinear forms.  The polynomial form of their estimate
became a dimension-free principle for controlling the coefficients of
homogeneous polynomials.  For the historical connection with Dirichlet series
and the Bohr transform, see also \cite{DFOOS,DGMS}.

In polynomial form, the theorem says that for every degree $m$ the
$\ell_{2m/(m+1)}$-norm of the coefficients of a complex $m$-homogeneous
polynomial is bounded by its supremum norm on the polydisc, with a constant
that is independent of the number of variables.  The exponent
$2m/(m+1)$ is optimal.  A particularly visible instance is the
multidimensional Bohr radius.  The
hypercontractive estimate of Defant, Frerick, Ortega-Cerd\`a, Ouna\"ies, and
Seip \cite{DFOOS} gave the correct order $\sqrt{(\log n)/n}$ up to absolute
constants; the subexponential form of the polynomial Bohnenblust--Hille
inequality obtained in \cite{BPS} led to the asymptotic equivalence at that
scale. {\color{black}More recently, polynomial growth of the complex polynomial
Bohnenblust--Hille constants was established in \cite{PT}.}

{\color{black}
The Bohnenblust--Hille inequality has developed into an active line of
research with connections to complex analysis, Dirichlet series, harmonic
analysis, quantum information, learning theory, Boolean and finite-group
analysis, and noncommutative polynomial inequalities.  Current work also
addresses the growth of the constants, exact and multilinear forms,
coefficient summability, restricted supports, and discrete variants
\cite{ADEDGP,BayartSupports,BPS,CNS,DGM,EskenazisIvanisvili,
IvanisviliHamming,MNP,Montanaro,PT,STVBoolean,SupportSensitive,
SVProductCyclic,SVZ,VZ}.
{\color{black}Defant, Garc\'ia, and Maestre \cite{DGM} studied maximum
moduli of unimodular polynomials, including estimates obtained from random
choices of coefficients.}

For support-restricted and support-sensitive inequalities, several different
notions of sparsity and interaction have been considered
\cite{CNS,MNP,SupportSensitive}.  Here we take a different viewpoint: the
support itself is a prescribed set of monomials, and we ask how the
Bohnenblust--Hille ratio is distributed among coefficient directions. For
$P(z)=\sum_{|\alpha|=m}a_\alpha z^\alpha$ and
$q_m=2m/(m+1)$, write
\[
 R_m(P):=\frac{\|(a_\alpha)\|_{q_m}}{\|P\|_\infty}.
\]
Since this ratio is unchanged by scaling the coefficients, we can
recast the problem on the Euclidean coefficient sphere, where
normalized surface measure gives a precise meaning to typical
behavior. Both the ambient dimension and the support may vary with $m$.

{\color{black}A universal threshold emerges at
$1$: along every prescribed sequence of supports, $R_m\le1$ in all sufficiently
large degrees almost surely, and the threshold is sharp
(Theorem~\ref{thm:tail-main}).}

We also determine exactly when the typical ratio vanishes:
$R_m$ tends to zero in spherical measure if and only if the number
of prescribed monomials tends to infinity
(Theorem~\ref{thm:vanishing-main}). For the full complex polynomial space,
we determine the typical scale uniformly in the dimension. For every fixed
$n\ge2$, $\sqrt{\log m}\,R_m$ converges in spherical measure to
$(n-1)^{-1/2}$. More generally, a single asymptotic scale is valid
uniformly in $n$, and on that scale we determine the exact logarithmic rate
of deviations below the typical value
(Theorem~\ref{thm:full-exact-main}).

There is an interesting tension between this typical behavior and the
extremal problem. Theorems~B and~C show that, as the degree grows, the
Bohnenblust--Hille ratio becomes smaller for most coefficient directions,
and in several natural regimes its typical size actually tends to zero.
Although this might suggest decreasing constants, the available extremal
estimates point in the opposite direction. There is no
contradiction: the optimal constant is determined by the largest ratio and
may therefore be governed by an increasingly exceptional part of the
coefficient sphere. This contrast is one of the motivations for separating
the typical and extremal questions.

The real case obeys a different dimensional criterion. On the full real
polynomial spaces, the ratio tends to zero in spherical measure if and only
if the number of variables tends to infinity
(Theorem~\ref{thm:real-main}).

In the critical complex regime $n_m/m\to1$, the extremal behavior is already
known from \cite[Theorem~B]{PT}. On the same full polynomial spaces, the
typical ratio is asymptotic in spherical measure to
$2\sqrt2/\sqrt{m\log m}$
(Theorem~\ref{thm:critical-main}).

Finally, for multilinear forms with a fixed number $n\ge2$ of coordinates
in each factor, the typical ratio has scale $(m\log m)^{-1/2}$, with the
exact limiting constant identified in Theorem~\ref{thm:multilinear-main}.
}

To state the main results precisely, for $k\in\NN$ write
\[
 [k]:=\{1,\ldots,k\},
\]
and let
\[
 \mathcal M_{m,n}:=\{\alpha\in\mathbb N_0^n:|\alpha|=m\},
 \qquad q_m:=\frac{2m}{m+1}.
\]
For a nonempty set $\Lambda\subseteq\mathcal M_{m,n}$ and
$a=(a_\alpha)_{\alpha\in\Lambda}\in\mathbb C^\Lambda$, put
\[
 P_a:\CC^n\to\CC,\qquad
 P_a(z):=\sum_{\alpha\in\Lambda}a_\alpha z^\alpha,
 \qquad
 R_m(a):=\frac{\|a\|_{q_m}}{\|P_a\|_\infty}\quad(a\ne0),
\]
where $\|P_a\|_\infty$ is the supremum on the unit polydisc, and set
$R_m(0):=0$. When the polynomial itself is the convenient object, we also
write $R_m(P_a):=R_m(a)$.  Let
\[
 \mathbb S_\Lambda:=\{a\in\mathbb C^\Lambda:\|a\|_2=1\}
\]
and let $\mu_\Lambda$ be normalized Euclidean surface measure on this sphere.
{\color{black}When $\Lambda=\mathcal M_{m,n}$, we call $\Lambda$ the full support. The corresponding coefficient space
$\mathbb C^{\mathcal M_{m,n}}$, or equivalently the space of all complex
$m$-homogeneous polynomials in $n$ variables,
\[
 P(z)=\sum_{|\alpha|=m}a_\alpha z^\alpha,
\]
is called the full polynomial space of degree $m$ in $n$ variables. Thus
``full'' is always understood degree by degree. We write
\[
 \mathbb S_{m,n}:=\mathbb S_{\mathcal M_{m,n}},
 \qquad
 \mu_{m,n}:=\mu_{\mathcal M_{m,n}},
\]
so $\mu_{m,n}$ is the normalized spherical measure on the coefficient sphere
of this single degree.}
For a sequence of nonempty supports
{\color{black}$(\Lambda_m)_{m\ge2}$, with
$\Lambda_m\subseteq\mathcal M_{m,n_m}$}, put
\[
 {\color{black}\Omega_{(\Lambda_m)}}:=\prod_{m\ge2}\mathbb S_{\Lambda_m},
 \qquad
 {\color{black}\boldsymbol\mu_{(\Lambda_m)}}:=\bigotimes_{m\ge2}\mu_{\Lambda_m}.
\]
\begin{mainthm}\label{thm:tail-main}
There is an absolute constant $C$ such that
\begin{equation}\label{endpoint:endpoint-uniform-bound}
 \mu_\Lambda\{a\in\mathbb S_\Lambda:R_m(a)>1\}
 \le C\frac{\log m}{m^2}
 \qquad(m\ge2),
\end{equation}
uniformly in $n$ and $\varnothing\ne\Lambda\subseteq\mathcal M_{m,n}$.
{\color{black}Moreover, for every sequence of nonempty supports
$(\Lambda_m)_{m\ge2}$, with $\Lambda_m\subseteq\mathcal M_{m,n_m}$,}
\begin{equation}\label{eq:eventual-contractivity-main}
 {\color{black}\boldsymbol\mu_{(\Lambda_m)}}\bigl\{(a^{(m)})\in{\color{black}\Omega_{(\Lambda_m)}}:
 R_m(a^{(m)})\le1\ \text{for all sufficiently large }m\bigr\}=1.
\end{equation}
The threshold $1$ is optimal.
\end{mainthm}

{\color{black}\hypersetup{linkcolor=black,citecolor=black}For a support consisting of one monomial}, $R_m\equiv1$, so the threshold cannot be lower.

\begin{mainthm}\label{thm:vanishing-main}
{\color{black}Let $(n_m)_{m\ge2}$ be a sequence of positive integers} and let
$\varnothing\ne\Lambda_m\subseteq\mathcal M_{m,n_m}$. Then the following
are equivalent:
\begin{enumerate}[label=\textup{(\roman*)},leftmargin=2.2em]
\item $|\Lambda_m|\longrightarrow\infty$;
\item for every $\varepsilon>0$,
\begin{equation}\label{eq:vanishing-main}
 \mu_{\Lambda_m}\{a\in\mathbb S_{\Lambda_m}:R_m(a)>\varepsilon\}
 \longrightarrow0.
\end{equation}
\end{enumerate}
\end{mainthm}

For the full polynomial space one can go further and identify the exact
first-order scale.

\begin{mainthm}[The full polynomial space]\label{thm:full-exact-main}
Let
\[
 N_{m,n}:=\binom{m+n-1}{m},\qquad
 H_{m,n}:=\frac{n-1}{2}\log\left(m+\frac{m^2}{n}\right).
\]
On the full complex coefficient spheres, the following assertions hold
uniformly for $n\ge2$ as $m\to\infty$.
\begin{enumerate}[label=\textup{(\roman*)},leftmargin=2.2em]
\item
\begin{equation}\label{eq:full-exact-main}
 \frac{\sqrt{H_{m,n}}}{N_{m,n}^{1/(2m)}}\,R_m\longrightarrow1
\end{equation}
in $\mu_{m,n}$-measure.
\item For every fixed $0<y<1$,
\begin{equation}\label{full:full-ratio-rate}
 \frac1{H_{m,n}}\log\mu_{m,n}\left\{
 \frac{\sqrt{H_{m,n}}}{N_{m,n}^{1/(2m)}}R_m<y
 \right\}=-(y^{-2}-1)+\mathrm{o}(1).
\end{equation}
\end{enumerate}
In particular, for every fixed $n\ge2$,
\begin{equation}\label{eq:full-fixed-n-main}
 \sqrt{\log m}\,R_m\longrightarrow\frac1{\sqrt{n-1}}
 \qquad\text{in }\mu_{m,n}\text{-measure}.
\end{equation}
\end{mainthm}

{\color{black}For real scalars, the full real polynomial space carries the analogous coefficient ratio, with the supremum norm taken over
$[-1,1]^n$.} If
$a=(a_\alpha)_{\alpha\in\mathcal M_{m,n}}\in\mathbb R^{\mathcal M_{m,n}}$, set
\[
 P_a^{\mathbb R}:\mathbb R^n\to\mathbb R,\qquad
 P_a^{\mathbb R}(x):=\sum_{|\alpha|=m}a_\alpha x^\alpha,
 \qquad
 \|P_a^{\mathbb R}\|_{\infty,\mathbb R}
 :=\sup_{x\in[-1,1]^n}|P_a^{\mathbb R}(x)|,
\]
and, for $a\ne0$,
\[
 {\color{black}R_m^{\mathbb R}(a):=
 \frac{\|a\|_{q_m}}{\|P_a^{\mathbb R}\|_{\infty,\mathbb R}}},
\]
{\color{black}and set $R_m^{\mathbb R}(0):=0$.}
We also write $R_m^{\mathbb R}(P_a^{\mathbb R}):=R_m^{\mathbb R}(a)$.
Let
\[
 \mathbb S_{m,n}^{\mathbb R}
 :=\{a\in\mathbb R^{\mathcal M_{m,n}}:\|a\|_2=1\},
\]
and let $\mu_{m,n}^{\mathbb R}$ be normalized Euclidean surface measure on
its Borel $\sigma$-algebra.

\begin{mainthm}[The real case]\label{thm:real-main}
{\color{black}
{\color{black}Let $(n_m)$ be any sequence of positive integers. For each $m$, take the full real support $\mathcal M_{m,n_m}$. Then}
\[
 R_m^{\mathbb R}\longrightarrow0
 \quad\text{in }\mu_{m,n_m}^{\mathbb R}\text{-measure}
 \quad\Longleftrightarrow\quad n_m\longrightarrow\infty.
\]
{\color{black}In contrast, in every fixed dimension the ratio remains of order one in
measure. More precisely,} for every fixed $n\ge1$ and every $\varepsilon>0$ there are
$0<c_{n,\varepsilon}<C_{n,\varepsilon}<\infty$ such that
\begin{equation}\label{eq:real-main-window}
 \mu_{m,n}^{\mathbb R}
 \{a:c_{n,\varepsilon}\le R_m^{\mathbb R}(a)\le C_{n,\varepsilon}\}
 \ge1-\varepsilon
\end{equation}
for all sufficiently large $m$.
}
\end{mainthm}

{\color{black}\hypersetup{linkcolor=black,citecolor=black}For complex polynomials, \cite[Theorem~B]{PT} shows that in the regime
$n_m/m\to1$ the extremal constants converge to $2$.
Theorem~\ref{thm:critical-main} describes the typical behavior on the same
coefficient spheres.}

\begin{mainthm}\label{thm:critical-main}
Let $n_m/m\to1$. {\color{black}For each $m$, take the full support $\Lambda_m=\mathcal M_{m,n_m}$.} Then
\begin{equation}\label{eq:critical-typical-main}
 \sqrt{m\log m}\,R_m
 \longrightarrow2\sqrt2
 \qquad\text{in }\mu_{m,n_m}\text{-measure}.
\end{equation}
\end{mainthm}

A parallel asymptotic relation holds for multilinear maps with a fixed number of
coordinates in each factor.
Fix $n\ge2$. For
$a=(a_{i_1,\ldots,i_m})\in\CC^{[n]^m}$, let
\[
 T_a:(\CC^n)^m\to\CC,\qquad
 T_a(x^{(1)},\ldots,x^{(m)})
 :=
 \sum_{i_1,\ldots,i_m=1}^n
 a_{i_1,\ldots,i_m}
 x^{(1)}_{i_1}\cdots x^{(m)}_{i_m},
\]
and, for $a\ne0$, define
\begin{equation}\label{eq:multilinear-ratio-intro}
 {\color{black}R_{m,n}^{\mathrm{ML}}(a)
 :=
 \frac{\left(\sum_{i_1,\ldots,i_m=1}^n
 |a_{i_1,\ldots,i_m}|^{q_m}\right)^{1/q_m}}
 {\|T_a\|}}.
\end{equation}
{\color{black}Set $R_{m,n}^{\mathrm{ML}}(0):=0$.}
{\color{black}Here}
\[
 \|T_a\|
 :=
 \sup_{\|x^{(1)}\|_\infty,\ldots,\|x^{(m)}\|_\infty\le1}
 |T_a(x^{(1)},\ldots,x^{(m)})|.
\]
\begin{mainthm}\label{thm:multilinear-main}
{\color{black}Let $\sigma_{m,n}$ denote normalized Euclidean surface measure on the
unit sphere of the multilinear coefficient space $\CC^{n^m}$.} For every fixed $n\ge2$,
\begin{equation}\label{eq:multilinear-main}
 \sqrt{m\log m}\,R_{m,n}^{\mathrm{ML}}
 \longrightarrow \sqrt{\frac{2n}{n-1}}
 \qquad\text{in }\sigma_{m,n}\text{-measure}.
\end{equation}
\end{mainthm}

{\color{black}
\clearpage
\thispagestyle{empty}
\begin{center}
{\Large\bfseries Proof architecture\par}
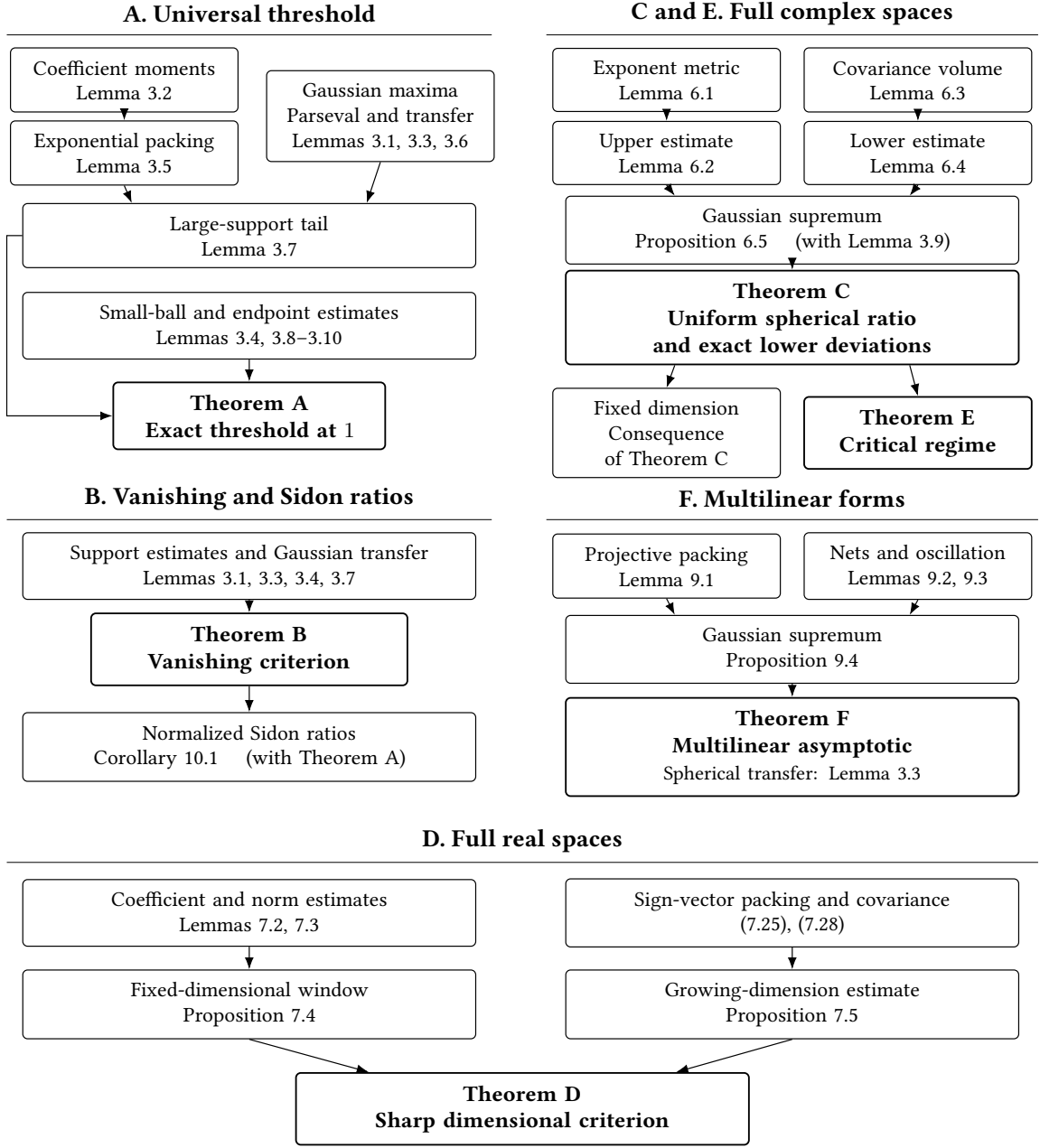
\vspace{0.35cm}
\begin{tikzpicture}[
 x=1cm,y=1cm,
 >=Latex,
 box/.style={draw=black,rounded corners=2pt,line width=.45pt,
   align=center,inner xsep=4pt,inner ysep=5pt,
   font=\fontsize{8.7}{10.5}\selectfont},
 thm/.style={box,line width=.7pt,
   font=\fontsize{9.3}{11.2}\selectfont\bfseries},
 heading/.style={font=\fontsize{10.3}{12.3}\selectfont\bfseries,
   align=center,inner sep=0pt},
 arr/.style={->,line width=.45pt},
 note/.style={font=\fontsize{8.3}{10}\selectfont,align=center,inner sep=0pt}
]
\path[use as bounding box] (0,0) rectangle (15.7,18.55);

\node[heading] at (3.80,18.20) {A. Universal threshold};
\draw[line width=.3pt] (0.25,17.88)--(7.35,17.88);
\node[box,text width=3.05cm] (Amoment) at (1.97,17.22)
 {Coefficient moments\\Lemma~\ref{lem:coefficient-moment}};
\node[box,text width=3.05cm] (Apacking) at (1.97,16.15)
 {Exponential packing\\Lemma~\ref{lem:exponential-packing}};
\node[box,text width=3.05cm] (Agaussian) at (5.72,16.69)
 {Gaussian maxima\\Parseval and transfer\\
  Lemmas~\ref{lem:counting-parseval}, \ref{lem:gaussian-radius-direction},
  \ref{lem:small-correlation-max-sharp}};
\node[box,text width=6.35cm] (Alarge) at (3.80,14.94)
 {Large-support tail\\Lemma~\ref{lem:large-support-tail}};
\node[box,text width=6.35cm] (Aendpoint) at (3.80,13.62)
 {Small-ball and endpoint estimates\\
  Lemmas~\ref{lem:translated-small-ball},
  \ref{endpoint:endpoint-exposed-pair}--\ref{endpoint:endpoint-product}};
\node[thm,text width=3.70cm] (Aresult) at (3.80,12.30)
 {Theorem~\ref{thm:tail-main}\\Exact threshold at $1$};
\draw[arr] (Amoment.south)--(Apacking.north);
\draw[arr] (Apacking.south)--([xshift=-1.72cm]Alarge.north);
\draw[arr] (Agaussian.south)--([xshift=1.72cm]Alarge.north);
\draw[arr] (Alarge.west)--(0.25,14.94)--(0.25,12.30)--(Aresult.west);
\draw[arr] (Aendpoint.south)--(Aresult.north);

\node[heading] at (11.75,18.20) {C and E. Full complex spaces};
\draw[line width=.3pt] (8.15,17.88)--(15.35,17.88);
\node[box,text width=3.05cm] (Cmetric) at (9.91,17.22)
 {Exponent metric\\Lemma~\ref{full:full-metric}};
\node[box,text width=3.05cm] (Cvolume) at (13.59,17.22)
 {Covariance volume\\Lemma~\ref{full:full-covariance-volume}};
\node[box,text width=3.05cm] (Cupper) at (9.91,16.15)
 {Upper estimate\\Lemma~\ref{full:full-upper}};
\node[box,text width=3.05cm] (Clower) at (13.59,16.15)
 {Lower estimate\\Lemma~\ref{full:full-large-lower}};
\node[box,text width=6.35cm] (Cnorm) at (11.75,15.02)
 {Gaussian supremum\\Proposition~\ref{full:full-norm}
  \quad (with Lemma~\ref{endpoint:endpoint-arbitrary-center})};
\node[thm,text width=6.35cm] (Cresult) at (11.75,13.74)
 {Theorem~\ref{thm:full-exact-main}\\Uniform spherical ratio\\and exact lower deviations};
\node[box,text width=3.05cm] (Cfixed) at (9.91,12.05)
 {Fixed dimension\\Consequence of Theorem~\ref{thm:full-exact-main}};
\node[thm,text width=3.05cm] (Eresult) at (13.59,12.05)
 {Theorem~\ref{thm:critical-main}\\Critical regime};
\draw[arr] (Cmetric.south)--(Cupper.north);
\draw[arr] (Cvolume.south)--(Clower.north);
\draw[arr] (Cupper.south)--([xshift=-1.72cm]Cnorm.north);
\draw[arr] (Clower.south)--([xshift=1.72cm]Cnorm.north);
\draw[arr] (Cnorm.south)--(Cresult.north);
\draw[arr] ([xshift=-1.72cm]Cresult.south)--(Cfixed.north);
\draw[arr] ([xshift=1.72cm]Cresult.south)--(Eresult.north);

\node[heading] at (3.80,11.10) {B. Vanishing and Sidon ratios};
\draw[line width=.3pt] (0.25,10.78)--(7.35,10.78);
\node[box,text width=6.35cm] (Binputs) at (3.80,10.10)
 {Support estimates and Gaussian transfer\\
  Lemmas~\ref{lem:counting-parseval}, \ref{lem:gaussian-radius-direction},
  \ref{lem:translated-small-ball}, \ref{lem:large-support-tail}};
\node[thm,text width=4.35cm] (Bresult) at (3.80,8.88)
 {Theorem~\ref{thm:vanishing-main}\\Vanishing criterion};
\node[box,text width=6.35cm] (Bsidon) at (3.80,7.45)
 {Normalized Sidon ratios\\Corollary~\ref{cor:typical-sidon}
  \quad (with Theorem~\ref{thm:tail-main})};
\draw[arr] (Binputs.south)--(Bresult.north);
\draw[arr] (Bresult.south)--(Bsidon.north);

\node[heading] at (11.75,11.10) {F. Multilinear forms};
\draw[line width=.3pt] (8.15,10.78)--(15.35,10.78);
\node[box,text width=3.05cm] (Fpacking) at (9.91,10.10)
 {Projective packing\\Lemma~\ref{lem:projective-packing}};
\node[box,text width=3.05cm] (Fcovering) at (13.59,10.10)
 {Nets and oscillation\\
  Lemmas~\ref{lem:projective-covering}, \ref{lem:dyadic-net-oscillation}};
\node[box,text width=6.35cm] (Fnorm) at (11.75,8.88)
 {Gaussian supremum\\Proposition~\ref{prop:multilinear-gaussian-supremum}};
\node[thm,text width=6.35cm] (Fresult) at (11.75,7.45)
 {Theorem~\ref{thm:multilinear-main}\\Multilinear asymptotic\normalfont
  \\{\fontsize{8.3}{10}\selectfont Spherical transfer:
       Lemma~\ref{lem:gaussian-radius-direction}}};
\draw[arr] (Fpacking.south)--([xshift=-1.72cm]Fnorm.north);
\draw[arr] (Fcovering.south)--([xshift=1.72cm]Fnorm.north);
\draw[arr] (Fnorm.south)--(Fresult.north);

\node[heading] at (7.80,6.08) {D. Full real spaces};
\draw[line width=.3pt] (0.25,5.76)--(15.35,5.76);
\node[box,text width=6.35cm] (Dfixedinputs) at (3.80,5.02)
 {Coefficient and norm estimates\\
  Lemmas~\ref{lem:real-numerator-fixed}, \ref{lem:real-denominator-fixed}};
\node[box,text width=6.35cm] (Dgrowinginputs) at (11.75,5.02)
 {Sign-vector packing and covariance\\
  \eqref{eq:real-hamming-cardinality}, \eqref{eq:real-sign-mixed-bound}};
\node[box,text width=6.35cm] (Dfixed) at (3.80,3.68)
 {Fixed-dimensional window\\Proposition~\ref{prop:real-fixed-tail}};
\node[box,text width=6.35cm] (Dgrowing) at (11.75,3.68)
 {Growing-dimension estimate\\Proposition~\ref{prop:real-growing-dimension}};
\node[thm,text width=6.35cm] (Dresult) at (7.80,2.15)
 {Theorem~\ref{thm:real-main}\\Sharp dimensional criterion};
\draw[arr] (Dfixedinputs.south)--(Dfixed.north);
\draw[arr] (Dgrowinginputs.south)--(Dgrowing.north);
\draw[arr] (Dfixed.south)--([xshift=-2.25cm]Dresult.north);
\draw[arr] (Dgrowing.south)--([xshift=2.25cm]Dresult.north);
\end{tikzpicture}
\end{center}
\clearpage
}

\section{Notation and preliminaries}\label{sec:notation-main}

{\color{black}We record here the conventions and identities used throughout the proofs.} Let
\[
 \TT:=\{z\in\CC:|z|=1\}.
\]
Fix $m,n\in\NN$, recall
\[
 \mathcal M_{m,n}:=\{\alpha\in\Nzero^n:|\alpha|=m\},
 \qquad q_m:=\frac{2m}{m+1},
\]
and let $\varnothing\ne\Lambda\subseteq\mathcal M_{m,n}$.
For $a=(a_\alpha)_{\alpha\in\Lambda}\in {\color{black}\CC^\Lambda}$ and $1\le p<\infty$, set
\begin{equation}\label{eq:supported-coefficient-lp}
 \|a\|_p:=\left(\sum_{\alpha\in\Lambda}|a_\alpha|^p\right)^{1/p}.
\end{equation}
The associated polynomial is
\begin{equation}\label{eq:Pa-map}
 P_a:\CC^n\longrightarrow\CC,
 \qquad P_a(z):=\sum_{\alpha\in\Lambda}a_\alpha z^\alpha.
\end{equation}
We write
\[
 \|P_a\|_\infty:=\sup_{z\in\TT^n}|P_a(z)|,
\]
which equals the supremum on the unit polydisc (apply the one-variable maximum-modulus argument successively in the coordinates). Define
\begin{equation}\label{eq:R-map}
 R_m(a):=
 \begin{cases}
 \|a\|_{q_m}/\|P_a\|_\infty,&a\ne0,\\
 0,&a=0.
 \end{cases}
\end{equation}
Since $R_m(\lambda a)=R_m(a)$ for $\lambda\ne0$, it is enough to work on
\begin{equation}\label{eq:coefficient-sphere}
 \mathbb S_\Lambda:=\{a\in {\color{black}\CC^\Lambda}:\|a\|_2=1\}.
\end{equation}
{\color{black}Identifying $\CC^\Lambda$ with $\RR^{2|\Lambda|}$, let $\sigma_{2|\Lambda|-1}$ be Euclidean
surface measure on $\mathbb S_\Lambda$ and define}
\begin{equation}\label{eq:spherical-measure}
 \mu_\Lambda(A):=
 {\color{black}\frac{\sigma_{2|\Lambda|-1}(A)}{\sigma_{2|\Lambda|-1}(\mathbb S_\Lambda)}}
 \qquad(A\subseteq\mathbb S_\Lambda\text{ Borel}).
\end{equation}

Throughout, $m,n\in\NN$ and $\varnothing\ne\Lambda\subseteq\mathcal M_{m,n}$,
unless a different index set is specified. For a metric space $(X,d)$ and
$r>0$, let $\mathcal N(X,d,r)$ denote the least cardinality of an $r$-net
of $X$.
For vectors $z,w\in\TT^n$, products and conjugation are coordinatewise;
in particular $z\overline w=(z_1\overline w_1,\ldots,z_n\overline w_n)$.
{\color{black}The standard complex Euclidean inner product is taken with the convention
\[
 \langle a,b\rangle:=\sum_{\alpha\in\Lambda}a_\alpha\overline b_\alpha,
\]
so it is linear in the first variable. Adjoints are taken with respect to this
inner product, or to its real part when the spaces are regarded as real
Euclidean spaces.}

Let $m_n$ be normalized Haar measure on $\TT^n$; explicitly,
\[
 \int_{\TT^n}f(z)\,dm_n(z)
 =\frac1{(2\pi)^n}\int_{[0,2\pi]^n}
            f(\e^{it_1},\ldots,\e^{it_n})\,dt_1\cdots dt_n
\]
for every integrable $f:\TT^n\to\CC$.
For a polynomial $P:\CC^n\to\CC$ and $1\le p<\infty$, set
\begin{equation}\label{eq:torus-lp}
 \|P\|_p:=\left(\int_{\TT^n}|P(z)|^p\,dm_n(z)\right)^{1/p}.
\end{equation}

{\color{black}Let $\operatorname{Leb}_{2|\Lambda|}$ denote Lebesgue measure on
${\color{black}\CC^\Lambda}\simeq\RR^{2|\Lambda|}$. {\color{black}A standard complex Gaussian variable is a complex random variable with density $\pi^{-1}\e^{-|z|^2}$ on $\CC$, and a standard complex Gaussian vector in $\CC^\Lambda$ has independent standard complex Gaussian coordinates.} The corresponding Gaussian measure (see, e.g., \cite[Appendix~A]{AubrunSzarek}) is
\begin{equation}\label{eq:gaussian-law}
 \gamma_\Lambda(A):=\pi^{-|\Lambda|}\int_A \e^{-\|a\|_2^2}\,
 d\operatorname{Leb}_{2|\Lambda|}(a)
 \qquad(A\subseteq {\color{black}\CC^\Lambda}\text{ Borel}).
\end{equation}
{\color{black}We write $\Prob$ and $\E$ for probability and expectation with respect to
the relevant standard Gaussian coefficient measure when no ambiguity arises;
integrals with respect to $\gamma_\Lambda$ are written explicitly whenever
the coefficient space needs to be displayed.}}

{\color{black}Gaussian coefficient measure and normalized surface measure on a coefficient sphere are related radially. If a Borel set $E$ is
invariant under multiplication by positive scalars, polar coordinates give
\begin{equation}\label{eq:radial-transfer-principle}
 \gamma_\Lambda(E)=\mu_\Lambda(E\cap\mathbb S_\Lambda).
\end{equation}
Thus every level set of a homogeneous ratio of degree zero has the same
Gaussian and spherical probability. Over $\mathbb R^\Lambda$, let
$\gamma_\Lambda^{\mathbb R}$ denote the standard real Gaussian measure and
let $\mu_\Lambda^{\mathbb R}$ denote normalized Euclidean surface measure on
\[
 \mathbb S_\Lambda^{\mathbb R}
 :=\{a\in\mathbb R^\Lambda:\|a\|_2=1\}.
\]
For every Borel set $E\subseteq\mathbb R^\Lambda$ invariant under
multiplication by positive scalars,
\[
 \gamma_\Lambda^{\mathbb R}(E)
 =
 \mu_\Lambda^{\mathbb R}(E\cap\mathbb S_\Lambda^{\mathbb R}).
\]
}

\section{Coefficient-sphere estimates}\label{sec:auxiliary}

{\color{black}For a fixed support, Parseval and the comparison between $\ell^{q_m}$ and $\ell^2$ give the basic deterministic control of the coefficient ratio.}

\begin{lemma}\label{lem:counting-parseval}
Let $\varnothing\ne\Lambda\subseteq\mathcal M_{m,n}$ and, for
$a=(a_\alpha)_{\alpha\in\Lambda}\in\CC^\Lambda$, write
{\color{black}
\[
 P_a(z):=\sum_{\alpha\in\Lambda}a_\alpha z^\alpha,
 \qquad z\in\CC^n.
\]}
Then
\begin{equation}\label{eq:counting}
 {\color{black}\|a\|_{q_m}\le |\Lambda|^{1/(2m)}\|a\|_2},
 \qquad \|P_a\|_2=\|a\|_2\le\|P_a\|_\infty,
\end{equation}
and, for every $a\ne0$,
\begin{equation}\label{eq:ratio-counting}
 {\color{black}R_m(a)\le |\Lambda|^{1/(2m)}}.
\end{equation}
\end{lemma}
\begin{proof}
Since $q_m<2$ and $1/q_m-1/2=1/(2m)$, H\"older's inequality gives
\[
 \sum_{\alpha\in\Lambda}|a_\alpha|^{q_m}
 \le\left(\sum_{\alpha\in\Lambda}|a_\alpha|^2\right)^{q_m/2}
      |\Lambda|^{1-q_m/2}.
\]
Taking $q_m$th roots proves the first inequality.
{\color{black}For $\alpha,\beta\in\Nzero^n$,
\begin{equation}\label{eq:character-orthogonality}
 \int_{\TT^n}z^\alpha\overline{z^\beta}\,dm_n(z)
 =\prod_{j=1}^n\left(\frac1{2\pi}\int_0^{2\pi}
 \e^{i(\alpha_j-\beta_j)t}\,dt\right)
 =\begin{cases}
 1,&\alpha=\beta,\\
 0,&\alpha\ne\beta.
 \end{cases}
\end{equation}
Indeed, if $\alpha_j\ne\beta_j$, then
\[
 \frac1{2\pi}\int_0^{2\pi}\e^{i(\alpha_j-\beta_j)t}\,dt
 =\frac{\e^{2\pi i(\alpha_j-\beta_j)}-1}
 {2\pi i(\alpha_j-\beta_j)}=0,
\]
whereas the integral equals $1$ when $\alpha_j=\beta_j$. Hence
\begin{align*}
 \|P_a\|_2^2
 &=\int_{\TT^n}
 \left(\sum_{\alpha\in\Lambda}a_\alpha z^\alpha\right)
 \left(\sum_{\beta\in\Lambda}\overline{a_\beta}\,\overline{z^\beta}\right)
 \,dm_n(z)\\
 &=\sum_{\alpha,\beta\in\Lambda}
 a_\alpha\overline{a_\beta}
 \int_{\TT^n}z^\alpha\overline{z^\beta}\,dm_n(z)
 =\sum_{\alpha\in\Lambda}|a_\alpha|^2.
\end{align*}}
The inequality $\|P_a\|_2\le\|P_a\|_\infty$ follows from $m_n(\TT^n)=1$.
For $a\ne0$, combining the two estimates in \eqref{eq:counting} gives
\[
 R_m(a)=\frac{\|a\|_{q_m}}{\|P_a\|_\infty}
 \le |\Lambda|^{1/(2m)}
 \frac{\|a\|_2}{\|P_a\|_\infty}
 \le |\Lambda|^{1/(2m)},
\]
which proves \eqref{eq:ratio-counting}.
\end{proof}

\begin{lemma}\label{lem:coefficient-moment}
{\color{black}For $a=(a_\alpha)_{\alpha\in\Lambda}\in\CC^\Lambda$, write
\[
 P_a(z):=\sum_{\alpha\in\Lambda}a_\alpha z^\alpha,
 \qquad z\in\CC^n.
\]
Then, for every $k\in\NN$ and $a,b\in\CC^\Lambda$,
\begin{equation}\label{eq:coefficient-moment-bound}
 \|P_a\|_{2k}\le k^{m/2}\|a\|_2,
 \qquad
 \bigl|\|P_a\|_{2k}-\|P_b\|_{2k}\bigr|
 \le k^{m/2}\|a-b\|_2.
\end{equation}}
\end{lemma}
\begin{proof}
{\color{black}Expanding the $k$th power gives
\begin{align*}
 P_a(z)^k
 &=\left(\sum_{\alpha\in\Lambda}a_\alpha z^\alpha\right)^k
 =\sum_{\alpha_1,\ldots,\alpha_k\in\Lambda}
 a_{\alpha_1}\cdots a_{\alpha_k}
 z^{\alpha_1+\cdots+\alpha_k}\\
 &=\sum_{\beta\in\mathcal M_{mk,n}}
 \left(
 \sum_{\substack{\alpha_1,\ldots,\alpha_k\in\Lambda\\
                   \alpha_1+\cdots+\alpha_k=\beta}}
 a_{\alpha_1}\cdots a_{\alpha_k}
 \right)z^\beta,
\end{align*}
where the inner sum is zero if $\beta$ has no such representation. Hence
\begin{equation}\label{eq:moment-expanded}
 \|P_a\|_{2k}^{2k}
 =\int_{\TT^n}|P_a(z)|^{2k}\,dm_n(z)
 =\|P_a^k\|_2^2
 =\sum_{\beta\in\mathcal M_{mk,n}}
 \left|
 \sum_{\substack{\alpha_1,\ldots,\alpha_k\in\Lambda\\
                   \alpha_1+\cdots+\alpha_k=\beta}}
 a_{\alpha_1}\cdots a_{\alpha_k}
 \right|^2.
\end{equation}

Fix $\beta\in\mathcal M_{mk,n}$. The number of ordered $k$-tuples
$(\alpha_1,\ldots,\alpha_k)\in\mathcal M_{m,n}^k$ satisfying
$\alpha_1+\cdots+\alpha_k=\beta$ is at most $k^{mk}$. Indeed, regard the
$mk$ units represented by $\beta$ as labeled and assign each of them to one
of $k$ boxes. There are $k^{mk}$ assignments. Every decomposition of $\beta$
into $k$ multiindices of degree $m$ is produced by at least one assignment,
by placing exactly $m$ units in each box.

For this fixed $\beta$, Cauchy--Schwarz yields
{\color{black}
\[
 \left|
 \sum_{\substack{\alpha_1,\ldots,\alpha_k\in\Lambda\\
                  \alpha_1+\cdots+\alpha_k=\beta}}
 a_{\alpha_1}\cdots a_{\alpha_k}
 \right|^2
 \le k^{mk}
 \sum_{\substack{\alpha_1,\ldots,\alpha_k\in\Lambda\\
                  \alpha_1+\cdots+\alpha_k=\beta}}
 |a_{\alpha_1}|^2\cdots|a_{\alpha_k}|^2.
\]
Summing this inequality over $\beta$ in \eqref{eq:moment-expanded} gives
\[
 \|P_a\|_{2k}^{2k}
 \le k^{mk}
 \sum_{\alpha_1,\ldots,\alpha_k\in\Lambda}
 |a_{\alpha_1}|^2\cdots|a_{\alpha_k}|^2
 =k^{mk}\left(\sum_{\alpha\in\Lambda}|a_\alpha|^2\right)^k
 =k^{mk}\|a\|_2^{2k}.
\]
}
{\color{black}Taking the $2k$th root proves the first assertion,
\[
 \|P_a\|_{2k}\le k^{m/2}\|a\|_2.
\]
For the second assertion, apply the estimate first to the coefficient
vector $a-b$:
\[
 \|P_{a-b}\|_{2k}\le k^{m/2}\|a-b\|_2.
\]
Since $P_a-P_b=P_{a-b}$, the triangle inequality in $L^{2k}(\TT^n)$ now gives
\[
 \bigl|\|P_a\|_{2k}-\|P_b\|_{2k}\bigr|
 \le\|P_a-P_b\|_{2k}
 =\|P_{a-b}\|_{2k}
 \le k^{m/2}\|a-b\|_2.
\]}}
\end{proof}

{\color{black}
\begin{lemma}\label{lem:gaussian-radius-direction}
{\color{black}Let $g=(g_\alpha)_{\alpha\in\Lambda}$ be a standard complex Gaussian vector in $\CC^\Lambda$.}
\begin{enumerate}[label=\textup{(\alph*)},leftmargin=2.2em]
\item For every $c>1$,
\begin{equation}\label{eq:gaussian-radius}
 \Prob\{\|g\|_2^2>c|\Lambda|\}
 \le \exp[-|\Lambda|(c-1-\log c)].
\end{equation}
\item For every $t\ge0$,
\[
 \Prob\{R_m(g)>t\}
 =\mu_\Lambda\{a\in\mathbb S_\Lambda:R_m(a)>t\}.
\]
\end{enumerate}
\end{lemma}
\begin{proof}
{\color{black}For each $\alpha\in\Lambda$, the coordinate $g_\alpha$ has density
$\pi^{-1}e^{-|z|^2}$ on $\CC$. Therefore, for $u\ge0$,
\[
 \Prob\{|g_\alpha|^2>u\}
 =\pi^{-1}\int_{|z|>\sqrt u}e^{-|z|^2}\,dz
 =e^{-u},
\]
so $|g_\alpha|^2$ is exponential with mean one.} Hence, for $0<v<1$,
\[
 \int_{\CC^\Lambda}e^{v\|a\|_2^2}\,d\gamma_\Lambda(a)
 =\prod_{\alpha\in\Lambda}\int_0^\infty e^{vu}e^{-u}\,du
 =(1-v)^{-|\Lambda|}.
\]
Markov's inequality \cite{BLM}, applied to
$e^{v\|g\|_2^2}$, gives
\[
 \Prob\{\|g\|_2^2>c|\Lambda|\}
 \le e^{-vc|\Lambda|}(1-v)^{-|\Lambda|}.
\]
Taking $v=1-1/c$ yields \eqref{eq:gaussian-radius}.

{\color{black}For (b), put $d:=2|\Lambda|$ and write
$a=\rho\omega$, where $\rho=\|a\|_2>0$ and
$\omega\in\mathbb S_\Lambda$. Under the identification
$\CC^\Lambda\simeq\RR^d$, Euclidean polar coordinates give
\[
 d\operatorname{Leb}_d(a)
 =\rho^{d-1}\,d\rho\,d\sigma_{d-1}(\omega).
\]
Hence, using
$d\mu_\Lambda(\omega)
=d\sigma_{d-1}(\omega)/\sigma_{d-1}(\mathbb S_\Lambda)$, integration with
respect to $\gamma_\Lambda$ has the polar-coordinate form
\[
 d\gamma_\Lambda(\rho\omega)
 =
 \frac{\sigma_{d-1}(\mathbb S_\Lambda)}{\pi^{|\Lambda|}}
 e^{-\rho^2}\rho^{d-1}\,d\rho\,d\mu_\Lambda(\omega).
\]
The radial factor has total mass one, since
\[
 \frac{\sigma_{d-1}(\mathbb S_\Lambda)}{\pi^{|\Lambda|}}
 \int_0^\infty e^{-\rho^2}\rho^{d-1}\,d\rho
 =
 \frac{2\pi^{|\Lambda|}}{\Gamma(|\Lambda|)\pi^{|\Lambda|}}
 \cdot\frac{\Gamma(|\Lambda|)}2
 =1.
\]
Since $R_m(\rho\omega)=R_m(\omega)$ for every $\rho>0$, it follows directly
that
\begin{align*}
 \Prob\{R_m(g)>t\}
 &=
 \frac{\sigma_{d-1}(\mathbb S_\Lambda)}{\pi^{|\Lambda|}}
 \int_0^\infty e^{-\rho^2}\rho^{d-1}\,d\rho
 \int_{\mathbb S_\Lambda}
 \mathbf 1_{\{R_m(\omega)>t\}}\,d\mu_\Lambda(\omega)\\
 &=\mu_\Lambda\{\omega\in\mathbb S_\Lambda:R_m(\omega)>t\},
\end{align*}
which proves (b).}

\end{proof}

{\color{black}We use the following geometric Brascamp--Lieb inequality in
Lemmas~\ref{lem:translated-small-ball} and~\ref{endpoint:endpoint-arbitrary-center}.
The formulation of Bennett, Carbery, Christ and Tao
\cite[Proposition~2.8]{BCCT} {\color{black}\hypersetup{linkcolor=black,citecolor=black}states that if}}
$B_j:\RR^d\to\RR^{d_j}$ satisfy $B_jB_j^*=I_{\RR^{d_j}}$ and positive
weights $c_j$ satisfy
\[
 \sum_j c_j B_j^*B_j=I_{\RR^d},
\]
then, for nonnegative integrable functions
\[
 f_j:\RR^{d_j}\longrightarrow[0,\infty),
\]
one has
\begin{equation}\label{eq:geometric-BL-form}
 \int_{\RR^d}\prod_j f_j(B_jx)^{c_j}\,dx
 \le \prod_j\left(\int_{\RR^{d_j}}f_j(y)\,dy\right)^{c_j}.
\end{equation}
}

\subsection{Uniform estimates}\label{sec:spherical-estimates}

\begin{lemma}
\label{lem:translated-small-ball}
Let $m,n\in\NN$ and $\varnothing\ne\Lambda\subseteq\mathcal M_{m,n}$. For $a=(a_\alpha)_{\alpha\in\Lambda}\in\CC^\Lambda$ define
{\color{black}
\[
 P_a(z):=\sum_{\alpha\in\Lambda}a_\alpha z^\alpha,
 \qquad z\in\CC^n.
\]}
Let $g=(g_\alpha)_{\alpha\in\Lambda}$ have independent standard complex
Gaussian coordinates and set
{\color{black}
\[
 P_g(z):=\sum_{\alpha\in\Lambda}g_\alpha z^\alpha.
\]}
Then, for every $\sigma>0$ and $s>0$,
\begin{equation}\label{eq:translated-small-ball}
 \Prob\{\|P_a+\sigma P_g\|_\infty\le\sigma s\sqrt{|\Lambda|}\}
 \le(1-e^{-s^2})^{|\Lambda|}\le\exp(-|\Lambda|e^{-s^2}).
\end{equation}
\end{lemma}
\begin{proof}
{\color{black}Put $N:=|\Lambda|$.} {\color{black}Choose an integer $\ell\ge2$ larger than every coordinate of every
multiindex in $\Lambda$, and let
\[
 {\color{black}\xi=e^{2\pi i/\ell}},\qquad
 \mathcal Z=\{({\color{black}\xi}^{j_1},\ldots,{\color{black}\xi}^{j_n}):
 j_1,\ldots,j_n\in\{0,\ldots,\ell-1\}\},\qquad M=\ell^n.
\]
{\color{black}For $z\in\mathcal Z$, define
\[
 B_z:\CC^\Lambda\longrightarrow\CC,\qquad
 B_zx:=\frac1{\sqrt N}\sum_{\alpha\in\Lambda}x_\alpha z^\alpha.
\]}
Its adjoint is
\[
 {\color{black}B_z^*y=\left(\frac{y\overline{z^\alpha}}{\sqrt N}\right)_{\alpha\in\Lambda}},
\]
so $B_zB_z^*=I_\CC$. Moreover, for $\alpha,\beta\in\Lambda$,
\begin{equation}\label{eq:grid-orthogonality}
 \frac1M\sum_{z\in\mathcal Z}z^\beta\overline{z^\alpha}
 =\prod_{j=1}^n\left(\frac1\ell\sum_{u=0}^{\ell-1}
 {\color{black}\xi}^{u(\beta_j-\alpha_j)}\right)
 =\begin{cases}
 1,&\alpha=\beta,\\
 0,&\alpha\ne\beta.
 \end{cases}
\end{equation}
Indeed, $|\beta_j-\alpha_j|<\ell$, so each finite geometric sum is $1$
when $\alpha_j=\beta_j$ and $0$ otherwise. Hence, with $c=N/M$,
\begin{equation}\label{eq:BL-frame}
 c\sum_{z\in\mathcal Z}B_z^*B_z=I_{\CC^\Lambda},
 \qquad \sum_{z\in\mathcal Z}c=N.
\end{equation}
After identifying $\CC^\Lambda$ with $\RR^{2N}$ and $\CC$ with $\RR^2$,
these are precisely the hypotheses of \eqref{eq:geometric-BL-form}.

Set $u=a/\sigma$. For $z\in\mathcal Z$ define
{\color{black}
\[
 f_z(y)=
 \begin{cases}
 e^{-|y-B_zu|^2},&|y|\le s,\\
 0,&|y|>s,
 \end{cases}
 \qquad y\in\CC.
\]}
{\color{black}
For each $z\in\mathcal Z$ and $x\in\CC^\Lambda$,
\[
 f_z(B_zx)^c
 =e^{-c|B_z(x-u)|^2}\,
 \mathbf 1_{\{|B_zx|\le s\}}.
\]
Multiplying these identities over $z\in\mathcal Z$ gives
\begin{equation}\label{eq:BL-product-expanded}
 \prod_{z\in\mathcal Z}f_z(B_zx)^c
 =\exp\!\left(-c\sum_{z\in\mathcal Z}|B_z(x-u)|^2\right)
 \prod_{z\in\mathcal Z}\mathbf 1_{\{|B_zx|\le s\}}.
\end{equation}
The frame identity \eqref{eq:BL-frame}, applied to the vector $x-u$, yields
\[
 c\sum_{z\in\mathcal Z}|B_z(x-u)|^2
 =\left\langle
 c\sum_{z\in\mathcal Z}B_z^*B_z(x-u),x-u
 \right\rangle
 =\|x-u\|_2^2.
\]
Thus the product in \eqref{eq:BL-product-expanded} equals
$e^{-\|x-u\|_2^2}$ when $|B_zx|\le s$ for every $z\in\mathcal Z$.
If $|B_zx|>s$ for at least one $z$, the corresponding indicator is zero,
and hence the whole product is zero.

For $b\in\CC$ and $r\ge0$, set
\[
 D(b,r):=\{y\in\CC:|y-b|\le r\},
\]
and let $\operatorname{Leb}_2$ denote Lebesgue measure on
$\CC\simeq\RR^2$.  {\color{black}For $y\in\CC$ and $0<v\le1$,
\[
 v\le e^{-|y-b|^2}
 \quad\Longleftrightarrow\quad
 |y-b|\le\sqrt{-\log v},
\]
and therefore
\[
 e^{-|y-b|^2}
 =\int_0^1\mathbf 1_{\{v\le e^{-|y-b|^2}\}}\,dv
 =\int_0^1\mathbf 1_{\{|y-b|\le\sqrt{-\log v}\}}\,dv.
\]}
The integrand is nonnegative, so Tonelli's theorem
\cite[Theorem~2.37(a)]{Folland} gives
\begin{align*}
 \int_{D(0,s)}e^{-|y-b|^2}\,dy
 &=\int_{D(0,s)}\int_0^1
 \mathbf 1_{D(b,\sqrt{-\log v})}(y)\,dv\,dy\\
 &=\int_0^1
 \operatorname{Leb}_2\bigl(D(0,s)\cap
 D(b,\sqrt{-\log v})\bigr)\,dv.
\end{align*}
For every $v\in(0,1]$, this intersection is contained in each of the two
discs, and therefore
\[
 \operatorname{Leb}_2\bigl(D(0,s)\cap
 D(b,\sqrt{-\log v})\bigr)
 \le \pi\min\{s^2,-\log v\}.
\]
Consequently,
\begin{align*}
 \int_{D(0,s)}e^{-|y-b|^2}\,dy
 &\le\int_0^1\pi\min\{s^2,-\log v\}\,dv\\
 &=\pi\int_0^{e^{-s^2}} s^2\,dv
   +\pi\int_{e^{-s^2}}^1(-\log v)\,dv\\
 &=\pi s^2e^{-s^2}
   +\pi\Bigl[-v\log v+v\Bigr]_{e^{-s^2}}^1\\
 &=\pi s^2e^{-s^2}
   +\pi\bigl(1-(s^2+1)e^{-s^2}\bigr)\\
 &=\pi(1-e^{-s^2}).
\end{align*}
Thus, for every $b\in\CC$ and $s>0$,
\begin{equation}\label{eq:translated-disc-integral}
 \int_{D(0,s)}e^{-|y-b|^2}\,dy\le\pi(1-e^{-s^2}).
\end{equation}
Taking $b=B_zu$ gives
\[
 \int_\CC f_z(y)\,dy\le\pi(1-e^{-s^2})
 \qquad(z\in\mathcal Z).
\]}
Applying \eqref{eq:geometric-BL-form} gives
\begin{align}
 \Prob\{|B_z(u+g)|\le s\text{ for every }z\in\mathcal Z\}
 &\le\pi^{-N}\prod_{z\in\mathcal Z}
 \left(\int_\CC f_z(y)\,dy\right)^c\notag\\
 &\le\pi^{-N}[\pi(1-e^{-s^2})]^{\sum c}
 =(1-e^{-s^2})^N.\label{eq:BL-measure}
\end{align}
If
\[
 \|P_a+\sigma P_g\|_\infty\le\sigma s\sqrt N,
\]
then, for every $z\in\mathcal Z$,
\[
 |B_z(u+g)|
 =\frac1{\sigma\sqrt N}\,
   |P_a(z)+\sigma P_g(z)|
 \le
 \frac{\|P_a+\sigma P_g\|_\infty}{\sigma\sqrt N}
 \le s.
\]
Hence
\[
 \{\|P_a+\sigma P_g\|_\infty\le\sigma s\sqrt N\}
 \subseteq
 \{|B_z(u+g)|\le s\text{ for every }z\in\mathcal Z\}.
\]
Using \eqref{eq:BL-measure}, we obtain explicitly
\[
 \Prob\{\|P_a+\sigma P_g\|_\infty\le\sigma s\sqrt N\}
 \le (1-e^{-s^2})^N.
\]
Finally, since $1-v\le e^{-v}$ for $0\le v\le1$, with
$v=e^{-s^2}$,
\[
 (1-e^{-s^2})^N
 \le \exp(-Ne^{-s^2}),
\]
which proves both inequalities in
\eqref{eq:translated-small-ball}.}
\end{proof}

{\color{black}Setting $a=0$ and $\sigma=1$ in
Lemma~\ref{lem:translated-small-ball} gives
\begin{equation}\label{eq:double-exp-BL-input}
 \Prob\{\|P_g\|_\infty\le s\sqrt N\}
 \le\exp(-N\e^{-s^2}).
\end{equation}}

\subsection{\textcolor{black}{Packing and Gaussian maxima}}\label{sec:technical-lemmas}

{\color{black}\hypersetup{linkcolor=black,citecolor=black}For large supports, we select many points on the torus where
the Gaussian polynomial has small pairwise correlations. Gaussian
comparison then controls the probability that all these values are small,
giving the estimates needed for Theorems~\ref{thm:tail-main}
and~\ref{thm:vanishing-main}.}

\begin{lemma}
\label{lem:exponential-packing}
{\color{black}Let $\varnothing\ne\Lambda\subseteq\mathcal M_{m,n}$ and assume
$|\Lambda|^{1/m}\ge64$. Put
\begin{equation}\label{eq:packing-parameters}
 \delta:=4^{-m},\qquad
 L:=\frac{(\log4)m|\Lambda|^{1/m}}{256},
 \qquad M:=\lceil \e^L\rceil.
\end{equation}
There are $z_1,\ldots,z_M\in\TT^n$ such that, writing
$z_i=(z_{i,1},\ldots,z_{i,n})$,
\begin{equation}\label{eq:packing-correlations}
 \left|\frac1{|\Lambda|}\sum_{\alpha\in\Lambda}
 \prod_{r=1}^n
 (z_{i,r}\overline{z_{j,r}})^{\alpha_r}\right|\le\delta
 \qquad(1\le i\ne j\le M).
\end{equation}
Equivalently, the product inside the sum is
$(z_i\overline{z_j})^\alpha$, where multiplication is coordinatewise.}
\end{lemma}
\begin{proof}
{\color{black}Put $N:=|\Lambda|$.}
Define
\[
 Q:\CC^n\longrightarrow\CC,\qquad Q(z):=\sum_{\alpha\in\Lambda}z^\alpha,
 \qquad
 \mathcal B:=\{z\in\TT^n:|Q(z)|>\delta N\}.
\]
{\color{black}For $k\in\NN$, apply Lemma~\ref{lem:coefficient-moment} to the
coefficient vector $a=(1)_{\alpha\in\Lambda}$. Since $P_a=Q$ and
$\|a\|_2=\sqrt N$,
\[
 \|Q\|_{2k}\le k^{m/2}\sqrt N.
\]
Therefore
\[
 \int_{\TT^n}|Q(z)|^{2k}\,dm_n(z)
 =\|Q\|_{2k}^{2k}
 \le (k^{m/2}\sqrt N)^{2k}
 =k^{mk}N^k.
\]}
Choose $k:=\lfloor N^{1/m}/64\rfloor$. Since $N^{1/m}/64\ge1$,
$N^{1/m}/128\le k\le N^{1/m}/64$. Markov's inequality \cite{BLM} gives
\begin{align}
 m_n(\mathcal B)
 &\le\frac{k^{mk}N^k}{(\delta N)^{2k}}
 =\delta^{-2k}(k/N^{1/m})^{mk}
 =(16k/N^{1/m})^{mk}\notag\\
 &\le4^{-mk}
 \le\exp\left(-\frac{(\log4)mN^{1/m}}{128}\right)
 =\e^{-2L}.\label{eq:packing-bad-set}
\end{align}
For $w\in\TT^n$, the map
\[
 \tau_w:\TT^n\longrightarrow\TT^n,\qquad \tau_w(z):=z\overline w
\]
is measure-preserving. We now construct $z_1,\ldots,z_M$ inductively so that,
for every $1\le j\le M$,
\begin{equation}\label{eq:packing-induction-invariant}
 |Q(z_r\overline{z_s})|\le\delta N
 \qquad(1\le r\ne s\le j).
\end{equation}
For $j=1$ there is nothing to prove, so take
$z_1=(1,\ldots,1)$.

Assume that $z_1,\ldots,z_j$ have been chosen, with $j<M$, and satisfy
\eqref{eq:packing-induction-invariant}. Since each $\tau_{z_i}$ preserves
$m_n$, \eqref{eq:packing-bad-set} gives
\[
 m_n\left(\bigcup_{i=1}^j\tau_{z_i}^{-1}(\mathcal B)\right)
 \le \sum_{i=1}^j m_n\!\left(\tau_{z_i}^{-1}(\mathcal B)\right)
 =j\,m_n(\mathcal B)
 \le j\e^{-2L}
 <1.
\]
Hence there exists
\[
 z_{j+1}\in
 \TT^n\setminus
 \bigcup_{i=1}^j\tau_{z_i}^{-1}(\mathcal B).
\]
For every $1\le i\le j$ this choice means
$z_{j+1}\overline{z_i}\notin\mathcal B$, and therefore
\[
 |Q(z_{j+1}\overline{z_i})|\le\delta N.
\]
Since
\[
 Q(z_i\overline{z_{j+1}})
 =\overline{Q(z_{j+1}\overline{z_i})},
\]
we also have
\[
 |Q(z_i\overline{z_{j+1}})|\le\delta N.
\]
Thus \eqref{eq:packing-induction-invariant} holds with $j+1$ in place of
$j$. By induction it holds for $j=M$. Finally, for $i\ne j$,
\[
 \frac1N Q(z_i\overline{z_j})
 =
 \frac1{|\Lambda|}
 \sum_{\alpha\in\Lambda}(z_i\overline{z_j})^\alpha,
\]
so dividing \eqref{eq:packing-induction-invariant} by $N=|\Lambda|$ gives
\eqref{eq:packing-correlations}.
\end{proof}

We next record a standard consequence of Slepian's comparison principle:
a large family of weakly correlated Gaussian variables is very unlikely to
have an abnormally small maximum. The constants are stated explicitly for
the parameter range used here.

\begin{lemma}\label{lem:small-correlation-max-sharp}
Let $X=(X_1,\ldots,X_M)$ be a real Gaussian vector satisfying
\[
 \E X_i=0,\qquad \E X_i^2=1\quad(1\le i\le M),
 \qquad \E(X_iX_j)\le\delta\quad(i\ne j).
\]
Assume
\[
 0<\delta\le 2^{-10},\qquad L\ge1000,\qquad M\ge \e^L,
 \qquad h\le\frac{127}{100}\sqrt L.
\]
Then
\begin{equation}\label{eq:small-correlation-maximum-sharp}
 \Prob\left\{\max_{1\le i\le M}X_i\le h\right\}
 \le
 \exp\left(-\frac{L}{3200\delta}\right)
 +\exp\left(-\e^{3L/20}\right).
\end{equation}
\end{lemma}
\begin{proof}
{\color{black}
Define
\[
 \varphi(x):=(2\pi)^{-1/2}e^{-x^2/2},
 \qquad
 \overline\Phi(x):=\int_x^\infty\varphi(u)\,du.
\]
An integration by parts gives, for $x>0$,
\begin{equation}\label{eq:Mills-lower}
 \overline\Phi(x)\ge \frac{x}{1+x^2}\varphi(x).
\end{equation}

Let $Z_0,Z_1,\ldots,Z_M$ be independent real Gaussian variables with mean
zero and variance one, and define
\[
 Y_i:=\sqrt\delta\,Z_0+\sqrt{1-\delta}\,Z_i,
 \qquad 1\le i\le M.
\]
{\color{black}By independence,
\[
 \E Y_i=0,\qquad
 \E Y_i^2=\delta+(1-\delta)=1,\qquad
 \E(Y_iY_j)=\delta\quad(i\ne j).
\]
Thus $X$ and $Y=(Y_1,\ldots,Y_M)$ have the same coordinate variances and
\[
 \E(X_iX_j)\le \E(Y_iY_j)\qquad(i\ne j).
\]
The lower-tail form of Slepian's comparison inequality \cite{Slepian} therefore yields}
\begin{equation}\label{eq:slepian-applied}
 \Prob\{\max_iX_i\le h\}
 \le \Prob\{\max_iY_i\le h\}.
\end{equation}

Set
\[
 E_0:=\left\{Z_0< -\frac{\sqrt L}{40\sqrt\delta}\right\},
 \qquad
 E_1:=\left\{\max_{1\le i\le M}Z_i\le\frac{13}{10}\sqrt L\right\}.
\]
\begin{equation}\label{eq:Y-event-inclusion}
 \{\max_iY_i\le h\}\subseteq E_0\cup E_1.
\end{equation}
Indeed, suppose that neither $E_0$ nor $E_1$ occurs. Then there is an index
$i$ such that $Z_i>13\sqrt L/10$, while
$Z_0\ge-\sqrt L/(40\sqrt\delta)$. For that index,
{\color{black}\[
 Y_i>
 \sqrt\delta\left(-\frac{\sqrt L}{40\sqrt\delta}\right)
 +\sqrt{1-\delta}\,\frac{13}{10}\sqrt L
 =\left(\frac{13}{10}\sqrt{1-\delta}-\frac1{40}\right)\sqrt L.
\]}
Since $\delta\le2^{-10}$,
\[
 \sqrt{1-\delta}>1-\delta\ge\frac{1023}{1024},
\]
and hence
\[
 \frac{13}{10}\sqrt{1-\delta}-\frac1{40}
 >\frac{13}{10}\frac{1023}{1024}-\frac1{40}
 >\frac{127}{100}.
\]
Therefore $Y_i>127\sqrt L/100\ge h$, which proves
\eqref{eq:Y-event-inclusion}.

{\color{black}\hypersetup{linkcolor=black,citecolor=black}To estimate $\Prob(E_0)$, let $Z$ be a real Gaussian variable with mean zero
and variance one. Then}
\[
 \E e^{\lambda Z}=e^{\lambda^2/2}\qquad(\lambda\in\mathbb R).
\]
The exponential Markov inequality, equivalently the Cram\'er--Chernoff
method \cite[Section~2.2]{BLM}, gives, for $t>0$ and $\lambda>0$,
\[
 \Prob\{Z\le-t\}
 =\Prob\{e^{-\lambda Z}\ge e^{\lambda t}\}
 \le e^{-\lambda t}\E e^{-\lambda Z}
 =e^{-\lambda t+\lambda^2/2}.
\]
Choosing $\lambda=t$ yields
\begin{equation}\label{eq:gaussian-one-sided-tail}
 \Prob\{Z\le-t\}\le e^{-t^2/2}\qquad(t>0).
\end{equation}
With $t=\sqrt L/(40\sqrt\delta)$, this gives
\begin{equation}\label{eq:common-gaussian-tail-sharp}
 \Prob(E_0)
 \le \exp\left(-\frac{L}{3200\delta}\right).
\end{equation}

For $E_1$, put $x:=13\sqrt L/10$. Independence of
$Z_1,\ldots,Z_M$ gives
\[
 \Prob(E_1)
 =\bigl(1-\overline\Phi(x)\bigr)^M
 \le \exp\bigl(-M\overline\Phi(x)\bigr).
\]
By \eqref{eq:Mills-lower},
\[
 \overline\Phi(x)
 \ge \frac{x}{1+x^2}\frac1{\sqrt{2\pi}}e^{-x^2/2}.
\]
Since $x^2=169L/100\ge1690>4$,
\[
 \frac{x}{1+x^2}\ge\frac4{5x}.
\]
{\color{black}Using $M\ge e^L$,
\[
 M\overline\Phi(x)
 \ge \frac{e^{31L/200}}{(13/8)\sqrt{2\pi L}}.
\]
To compare the right-hand side with $e^{3L/20}$, it is enough to verify
\[
 \frac{e^{L/200}}{\sqrt L}\ge\frac{13}{8}\sqrt{2\pi}.
\]
For $L\ge1000$ the function on the left is increasing, since
\[
 \frac{d}{dL}\left(\frac{e^{L/200}}{\sqrt L}\right)
 =
 \frac{e^{L/200}}{\sqrt L}
 \left(\frac1{200}-\frac1{2L}\right)>0.
\]
At $L=1000$, {\color{black}\hypersetup{linkcolor=black,citecolor=black}the inequality to be proved} is equivalent to
\[
 e^5\ge\frac{13}{8}\sqrt{2000\pi},
\]
which holds. Therefore
$M\overline\Phi(x)\ge e^{3L/20}$, and hence
\begin{equation}\label{eq:max-independent-gaussians}
 \Prob(E_1)\le\exp\bigl(-e^{3L/20}\bigr).
\end{equation}
{\color{black}\hypersetup{linkcolor=black,citecolor=black}Combining these estimates gives}
\begin{align*}
 \Prob\{\max_iX_i\le h\}
 &\overset{\eqref{eq:slepian-applied}}{\le}
   \Prob\{\max_iY_i\le h\}\\
 &\overset{\eqref{eq:Y-event-inclusion}}{\le}
   \Prob(E_0)+\Prob(E_1)\\
 &\overset{\eqref{eq:common-gaussian-tail-sharp},
            \eqref{eq:max-independent-gaussians}}{\le}
   \exp\left(-\frac{L}{3200\delta}\right)
   +\exp\left(-e^{3L/20}\right),
\end{align*}
which is \eqref{eq:small-correlation-maximum-sharp}.}
}
\end{proof}

\begin{lemma}\label{lem:large-support-tail}
{\color{black}Let $\varnothing\ne\Lambda\subseteq\mathcal M_{m,n}$ and let $g$ be
the standard complex Gaussian coefficient vector on $\CC^\Lambda$.
For every $\eta>0$ there is $m_0(\eta)$ such that, whenever
$m\ge m_0(\eta)$ and $|\Lambda|^{1/m}\ge1710$,}
\begin{equation}\label{eq:double-exp-large-support}
 \Prob\{R_m(g)>\eta\}
 \le4\exp\left(-\frac{|\Lambda|^{1/m}4^m}{1710}\right).
\end{equation}
\end{lemma}
\begin{proof}
{\color{black}Put $N:=|\Lambda|$.}
Let $A:=\{R_m(g)>\eta\}$ and $H:=\{\|g\|_2^2\le2N\}$.
Write
\[
 \rho:=\|g\|_2,\qquad \omega:=\frac{g}{\|g\|_2},
 \qquad g=\rho\omega.
\]
The polar-coordinate formula in the proof of
Lemma~\ref{lem:gaussian-radius-direction} writes integration with respect to
$\gamma_\Lambda$ as a radial integral in $\rho$ followed by integration on
$\mathbb S_\Lambda$ with respect to $\mu_\Lambda$. Consequently, $\rho$ and
$\omega$ are independent. Since $A$ depends only on $\omega$ and $H$ only on
$\rho$, the events $A$ and $H$ are independent.
Since $N^{1/m}\ge1710$, we have $N\ge1710$, and
\[
 \Prob(H^c)\le(2/\e)^N<\tfrac12.
\]
Thus $\Prob(A)\le2\Prob(A\cap H)$. On $A\cap H$, {\color{black}Lemma~\ref{lem:counting-parseval} gives}
\[
 \|P_g\|_\infty
 <\frac{\|g\|_{q_m}}{\eta}
 \le\frac{N^{1/(2m)}\|g\|_2}{\eta}
 \le\frac{\sqrt{2N^{1+1/m}}}{\eta}.
\]
Set
\[
 \delta:=4^{-m},\qquad
 L:=\frac{(\log4)mN^{1/m}}{256},\qquad
 M:=\lceil e^L\rceil,
\]
and choose $z_1,\ldots,z_M\in\TT^n$ from
Lemma~\ref{lem:exponential-packing}. Define
\[
 X_i:=\sqrt{2/N}\,\operatorname{Re} P_g(z_i)\qquad(i\in[M]).
\]
Then
\[
 \E X_i^2=1,\qquad
 \E X_iX_j
 =\operatorname{Re}\left(\frac1N\sum_{\alpha\in\Lambda}
                         (z_i\overline{z_j})^\alpha\right)
 \le\delta\quad(i\ne j).
\]
On $A\cap H$,
\[
 \max_iX_i<\frac{2N^{1/(2m)}}{\eta}.
\]
{\color{black}
The hypothesis on the level $h$ in Lemma~\ref{lem:small-correlation-max-sharp}
is satisfied on $A\cap H$ provided that
\begin{equation}\label{eq:large-support-h-condition}
 \frac{2N^{1/(2m)}}{\eta}
 \le\frac{127}{100}\sqrt L.
\end{equation}
After substituting the value of $L$ and cancelling $N^{1/(2m)}$,
\eqref{eq:large-support-h-condition} is equivalent to
\[
 \frac2\eta\le\frac{127}{1600}\sqrt{m\log4},
\]
or, equivalently,
\[
 m\ge \frac{10\,240\,000}{16\,129\,\eta^2\log4}.
\]
Set
\[
 m_0(\eta):=
 \left\lceil
 \max\left\{460,
 \frac{10\,240\,000}{16\,129\,\eta^2\log4}
 \right\}
 \right\rceil.
\]
For $m\ge m_0(\eta)$, condition
\eqref{eq:large-support-h-condition} holds. In addition, $m\ge460$ and
$N^{1/m}\ge1710$ imply $L>1000$, while $\delta=4^{-m}<2^{-10}$.
Thus all the hypotheses of Lemma~\ref{lem:small-correlation-max-sharp} are
satisfied, and it yields}
\[
 \Prob(A\cap H)
 \le
 \exp\left(-\frac{L}{3200\delta}\right)
 +\exp(-\e^{3L/20}).
\]
{\color{black}The two terms are compared with the target exponent
$N^{1/m}4^m/1710$. For the first one,
\[
 \frac{L}{3200\delta}
 =\frac{m(\log4)N^{1/m}}{819200}\,4^m
 \ge\frac{N^{1/m}4^m}{1710},
\]
because $m\ge460$.

For the second one, it is enough to prove
\[
 \frac{3L}{20}
 \ge m\log4+\log\!\left(\frac{N^{1/m}}{1710}\right),
\]
since exponentiating this inequality gives
\[
 \e^{3L/20}\ge\frac{N^{1/m}4^m}{1710}.
\]
Put $x:=N^{1/m}\ge1710$. After substituting the value of $L$, the estimate
\[
 \e^{3L/20}\ge\frac{N^{1/m}4^m}{1710}
\]
is equivalent to
\[
 \frac{3m(\log4)}{5120}\,x
 \ge m\log4+\log(x/1710).
\]
At $x=1710$ this follows from
\[
 \frac{3\cdot1710}{5120}>1.
\]
Moreover, the derivative of the left-hand side minus the right-hand side is
\[
 \frac{3m\log4}{5120}-\frac1x>0
 \qquad(x\ge1710,\ m\ge460),
\]
so the inequality holds for every $x\ge1710$. Therefore
\[
 \exp\!\left(-\frac{L}{3200\delta}\right)
 +\exp\!\left(-e^{3L/20}\right)
 \le
 2\exp\!\left(-\frac{N^{1/m}4^m}{1710}\right).
\]
Combining this estimate with
$\Prob(A)\le2\Prob(A\cap H)$ proves
\eqref{eq:double-exp-large-support}.}
\end{proof}

\subsection{Endpoint estimate for Theorem A}\label{sec:endpoint-one}

\begin{lemma}\label{endpoint:endpoint-exposed-pair}
{\color{black}Let $\Lambda\subset\mathbb N_0^n$ be a finite set of at least two distinct exponent vectors and,
for $a=(a_\alpha)_{\alpha\in\Lambda}$, let
\[
 P_a(z):=\sum_{\alpha\in\Lambda}a_\alpha z^\alpha.
\]
There are $\alpha_+,\alpha_-\in\Lambda$, depending only on $\Lambda$,
such that}
\begin{equation}\label{endpoint:endpoint-fourier-gap}
 \|P_a\|_\infty^2
 \ge\|a\|_2^2+|a_{\alpha_+}a_{\alpha_-}|.
\end{equation}
\end{lemma}
\begin{proof}
{\color{black}
Choose $\alpha_+$ and $\alpha_-$ explicitly. Since $\Lambda$ is
finite, there is an integer $R\ge2$ such that
\[
 0\le \alpha_j<R
 \qquad(\alpha\in\Lambda,\ 1\le j\le n).
\]
Define the real linear functional
\[
 L:\mathbb R^n\longrightarrow\mathbb R,
 \qquad
 L(x_1,\ldots,x_n)
 :=x_1+Rx_2+R^2x_3+\cdots+R^{n-1}x_n.
\]
The functional $L$ is injective on $\Lambda$. Indeed, let $\alpha,\gamma\in\Lambda$ with $\alpha\ne\gamma$, and let
$j$ be the largest index for which $\alpha_j\ne\gamma_j$. If
$L(\alpha)=L(\gamma)$, then
\[
 (\alpha_j-\gamma_j)R^{j-1}
 =-\sum_{r=1}^{j-1}(\alpha_r-\gamma_r)R^{r-1}.
\]
The absolute value of the left-hand side is at least $R^{j-1}$, because
$\alpha_j-\gamma_j$ is a nonzero integer. On the other hand, since
$|\alpha_r-\gamma_r|\le R-1$,
{\color{black}\[
 \left|\sum_{r=1}^{j-1}(\alpha_r-\gamma_r)R^{r-1}\right|
 \le (R-1)\sum_{r=1}^{j-1}R^{r-1}=R^{j-1}-1.
\]}
{\color{black}This is impossible, so $L$ takes distinct values at distinct points of
$\Lambda$. Since $\Lambda$ is finite, $L$ attains its maximum and minimum
there, and the distinctness just proved makes both extremizers unique.}
Denote them by $\alpha_+$ and $\alpha_-$, respectively, and put
\[
 \beta:=\alpha_+-\alpha_-.
\]
Since $\Lambda$ has at least two points, $\alpha_+\ne\alpha_-$ and hence
$\beta\ne0$.

The representation of $\beta$ as a difference of two
points of $\Lambda$ is unique. Suppose
\[
 \beta=\alpha-\gamma,
 \qquad \alpha,\gamma\in\Lambda.
\]
Applying $L$ gives
\begin{equation}\label{eq:extreme-difference-L}
 L(\alpha_+)-L(\alpha_-)=L(\alpha)-L(\gamma).
\end{equation}
By the definitions of $\alpha_+$ and $\alpha_-$,
\[
 L(\alpha)\le L(\alpha_+),
 \qquad
 L(\gamma)\ge L(\alpha_-).
\]
Therefore
\[
 L(\alpha)-L(\gamma)
 \le L(\alpha_+)-L(\alpha_-).
\]
Equality holds by \eqref{eq:extreme-difference-L}. Thus both preceding
inequalities must be equalities:
\[
 L(\alpha)=L(\alpha_+),
 \qquad
 L(\gamma)=L(\alpha_-).
\]
Since $L$ takes distinct values on $\Lambda$,
\[
 \alpha=\alpha_+,
 \qquad
 \gamma=\alpha_-.
\]
Hence
\begin{equation}\label{eq:unique-extreme-difference}
 \alpha-\gamma=\beta,
 \quad \alpha,\gamma\in\Lambda
 \quad\Longrightarrow\quad
 (\alpha,\gamma)=(\alpha_+,\alpha_-).
\end{equation}

{\color{black}Expanding on the torus gives
\[
 |P_a(z)|^2
 =\left(\sum_{\alpha\in\Lambda}a_\alpha z^\alpha\right)
  \left(\sum_{\gamma\in\Lambda}\overline{a_\gamma}\,z^{-\gamma}\right)
 =\sum_{\alpha,\gamma\in\Lambda}
  a_\alpha\overline{a_\gamma}\,z^{\alpha-\gamma}.
\]}
With the convention
\[
 \widehat f(\eta)
 :=\int_{\TT^n}f(z)\overline{z^\eta}\,dm_n(z),
 \qquad \eta\in\mathbb Z^n,
\]
the coefficient at $\beta$ is
\[
 \widehat{|P_a|^2}(\beta)
 =\sum_{\substack{\alpha,\gamma\in\Lambda\\
                    \alpha-\gamma=\beta}}
   a_\alpha\overline{a_\gamma}.
\]
By \eqref{eq:unique-extreme-difference}, this sum has exactly one term, so
\begin{equation}\label{eq:beta-fourier-coefficient}
 \widehat{|P_a|^2}(\beta)
 =a_{\alpha_+}\overline{a_{\alpha_-}}.
\end{equation}

Set
\[
 Q(z):=\|P_a\|_\infty^2-|P_a(z)|^2.
\]
Then $Q(z)\ge0$ for every $z\in\TT^n$. Since $\beta\ne0$, the constant
function $\|P_a\|_\infty^2$ has Fourier coefficient zero at $\beta$.
Thus, by \eqref{eq:beta-fourier-coefficient},
\[
 \widehat Q(\beta)
 =-a_{\alpha_+}\overline{a_{\alpha_-}},
\]
{\color{black}and therefore
\begin{equation}\label{eq:Q-fourier-bound}
 |a_{\alpha_+}a_{\alpha_-}|
 =|\widehat Q(\beta)|
 =\left|\int_{\TT^n}Q(z)\overline{z^\beta}\,dm_n(z)\right|
 \le\int_{\TT^n}Q(z)\,dm_n(z),
\end{equation}
since $|z^\beta|=1$ on $\TT^n$ and $Q\ge0$.

To evaluate the last integral, we use Parseval's identity for the finite
Fourier series $P_a$. This follows directly from the orthogonality of the
characters of $\TT^n$: for $\eta\in\mathbb Z^n$,
\[
 \int_{\TT^n}z^\eta\,dm_n(z)
 =\begin{cases}
 1,&\eta=0,\\
 0,&\eta\ne0.
 \end{cases}
\]
Thus
\[
 \int_{\TT^n}|P_a(z)|^2\,dm_n(z)
 =\sum_{\alpha,\gamma\in\Lambda}
   a_\alpha\overline{a_\gamma}
   \int_{\TT^n}z^{\alpha-\gamma}\,dm_n(z)
 =\sum_{\alpha\in\Lambda}|a_\alpha|^2
 =\|a\|_2^2.
\]
Consequently,
\[
 \int_{\TT^n}Q(z)\,dm_n(z)
 =\|P_a\|_\infty^2-\|a\|_2^2.
\]
Substitution in \eqref{eq:Q-fourier-bound} gives
\[
 |a_{\alpha_+}a_{\alpha_-}|
 \le\|P_a\|_\infty^2-\|a\|_2^2,
\]
which is exactly \eqref{endpoint:endpoint-fourier-gap}.}
}
\end{proof}

\begin{lemma}\label{endpoint:endpoint-arbitrary-center}
{\color{black}Let $\Gamma\subset\mathbb N_0^n$ be a finite nonempty set of distinct exponent vectors,
and let $g=(g_\alpha)_{\alpha\in\Gamma}$ have independent
standard complex Gaussian coordinates. Write
\[
 P_g(z):=\sum_{\alpha\in\Gamma}g_\alpha z^\alpha.
\]
{\color{black}For every continuous map $h:\TT^n\to\CC$ and every
$\sigma,s>0$,}}
\begin{equation}\label{endpoint:endpoint-translated-bound}
 \Prob\{\|h+\sigma P_g\|_\infty\le\sigma s\sqrt{|\Gamma|}\}
 \le(1-e^{-s^2})^{|\Gamma|}\le\exp(-|\Gamma|e^{-s^2}).
\end{equation}
\end{lemma}
\begin{proof}
{\color{black}Put $K:=|\Gamma|$ and choose an integer
\[
 \ell>\max\bigl\{|\alpha_j-\beta_j|:\alpha,\beta\in\Gamma,\ 1\le j\le n\bigr\}.
\]
Set $\xi:=e^{2\pi i/\ell}$ and let}
\[
 \mathcal Z:=\{(\xi^{j_1},\ldots,\xi^{j_n}):
 j_1,\ldots,j_n\in\{0,\ldots,\ell-1\}\},
 \qquad M:=\ell^n.
\]
For $z\in\mathcal Z$, define
\[
 B_z:\CC^\Gamma\to\CC,\qquad
 B_zx:=K^{-1/2}\sum_{\alpha\in\Gamma}x_\alpha z^\alpha,
 \qquad c:=K/M.
\]
Then $B_zB_z^*=I_{\CC}$. {\color{black}\hypersetup{linkcolor=black,citecolor=black}Moreover, for $\alpha,\beta\in\Gamma$,
writing $z=(\xi^{j_1},\ldots,\xi^{j_n})$ and separating the finite sums gives}
{\color{black}
\begin{align*}
 \frac1M\sum_{z\in\mathcal Z}z^\beta\overline{z^\alpha}
 &=\frac1{\ell^n}\sum_{j_1,\ldots,j_n=0}^{\ell-1}
   \prod_{r=1}^n\xi^{j_r(\beta_r-\alpha_r)}\\
 &=\prod_{r=1}^n\left(\frac1\ell\sum_{u=0}^{\ell-1}
   \xi^{u(\beta_r-\alpha_r)}\right).
\end{align*}}If $\alpha_j=\beta_j$, the $j$th factor is $1$. If
$\alpha_j\ne\beta_j$, then
$0<|\beta_j-\alpha_j|<\ell$, and the finite geometric sum is
\[
 \sum_{u=0}^{\ell-1}\xi^{u(\beta_j-\alpha_j)}
 =\frac{1-\xi^{\ell(\beta_j-\alpha_j)}}
        {1-\xi^{\beta_j-\alpha_j}}
 =0.
\]
Thus
\begin{equation}\label{eq:grid-character-orthogonality}
 \frac1M\sum_{z\in\mathcal Z}z^\beta\overline{z^\alpha}
 =\begin{cases}
 1,&\alpha=\beta,\\
 0,&\alpha\ne\beta.
 \end{cases}
\end{equation}
Consequently, for $x=(x_\beta)_{\beta\in\Gamma}$ and $\alpha\in\Gamma$,
\[
 \left(c\sum_{z\in\mathcal Z}B_z^*B_zx\right)_\alpha
 =\sum_{\beta\in\Gamma}x_\beta
 \left(\frac1M\sum_{z\in\mathcal Z}z^\beta\overline{z^\alpha}\right)
 =x_\alpha.
\]
{\color{black}Since this coordinate identity holds for every
$x\in\CC^\Gamma$, it gives
\begin{equation}\label{eq:endpoint-BL-identity}
 c\sum_{z\in\mathcal Z}B_z^*B_z=I_{\CC^\Gamma}.
\end{equation}
Moreover,
\[
 \sum_{z\in\mathcal Z}c=Mc=K.
\]}

{\color{black}
For $z\in\mathcal Z$, put
\[
 b_z:=\frac{h(z)}{\sigma\sqrt K}
\]
and define
\[
 f_z(w):=
 \begin{cases}
 e^{-|w|^2},&|w+b_z|\le s,\\
 0,&|w+b_z|>s.
 \end{cases}
\]
Define also the subset
\[
 E:=\{x\in\CC^\Gamma:|B_zx+b_z|\le s
       \text{ for every }z\in\mathcal Z\}.
\]
For a fixed $x\in\CC^\Gamma$ and a fixed $z\in\mathcal Z$,
\[
 f_z(B_zx)^c
 =e^{-c|B_zx|^2}\,
  \mathbf 1_{\{|B_zx+b_z|\le s\}}.
\]
Multiplying these identities over $z\in\mathcal Z$ gives
\begin{align*}
 \prod_{z\in\mathcal Z}f_z(B_zx)^c
 &=\exp\left(-c\sum_{z\in\mathcal Z}|B_zx|^2\right)
   \prod_{z\in\mathcal Z}
   \mathbf 1_{\{|B_zx+b_z|\le s\}}.
\end{align*}
Applying \eqref{eq:endpoint-BL-identity} to $x$ and taking the inner
product with $x$ yields
\[
 c\sum_{z\in\mathcal Z}|B_zx|^2
 =\left\langle
   c\sum_{z\in\mathcal Z}B_z^*B_zx,x
  \right\rangle
 =\|x\|_2^2.
\]
Therefore
\begin{equation}\label{eq:endpoint-product-indicator}
 \prod_{z\in\mathcal Z}f_z(B_zx)^c
 =e^{-\|x\|_2^2}\mathbf 1_E(x).
\end{equation}
{\color{black}For every $z\in\mathcal Z$, make the change of variables
$u=w+b_z$. Then $|w+b_z|\le s$ becomes $|u|\le s$ and
$|w|=|u-b_z|$, so
\begin{equation}\label{eq:endpoint-disc-integral}
 \int_{\CC}f_z(w)\,dw
 =\int_{|u|\le s}e^{-|u-b_z|^2}\,du
 \le\pi(1-e^{-s^2}).
\end{equation}}
The estimate
\[
 \int_{|u|\le s}e^{-|u-b_z|^2}\,du\le\pi(1-e^{-s^2})
\]
is \eqref{eq:translated-disc-integral} with $b=b_z$.
Using \eqref{eq:endpoint-product-indicator} and then the geometric
Brascamp--Lieb inequality \eqref{eq:geometric-BL-form}, we obtain
\begin{align*}
 \Prob\{g\in E\}
 &=\pi^{-K}\int_{\CC^\Gamma}
   e^{-\|x\|_2^2}\mathbf 1_E(x)\,dx\\
 &=\pi^{-K}\int_{\CC^\Gamma}
   \prod_{z\in\mathcal Z}f_z(B_zx)^c\,dx\\
 &\le\pi^{-K}\prod_{z\in\mathcal Z}
   \left(\int_{\CC}f_z(w)\,dw\right)^c\\
 &\le\pi^{-K}
   [\pi(1-e^{-s^2})]^{cM}\\
 &= (1-e^{-s^2})^K,
\end{align*}
where the last equality uses $cM=K$.

Finally, if
\[
 \|h+\sigma P_g\|_\infty\le\sigma s\sqrt K,
\]
then, for every $z\in\mathcal Z$,
\[
 |h(z)+\sigma P_g(z)|\le\sigma s\sqrt K.
\]
Since $P_g(z)=\sqrt K\,B_zg$, division by $\sigma\sqrt K$ gives
\[
 |B_zg+b_z|\le s.
\]
Thus $g\in E$, and so
\[
 \Prob\{\|h+\sigma P_g\|_\infty\le\sigma s\sqrt K\}
 \le\Prob\{g\in E\}
 \le(1-e^{-s^2})^K.
\]
The bound
\[
 (1-e^{-s^2})^K\le e^{-Ke^{-s^2}}
\]
follows by applying $1-u\le e^{-u}$ with $u=e^{-s^2}$.
}
\end{proof}

\begin{lemma}\label{endpoint:endpoint-product}
For independent standard complex Gaussian variables $g_1,g_2$,
\begin{equation}\label{endpoint:endpoint-product-bound}
 \Prob\{|g_1g_2|\le x\}
 \le2x^2\log(e/x)\qquad(0<x\le1).
\end{equation}
\end{lemma}
\begin{proof}
{\color{black}
{\color{black}Define
\[
 U,V:\CC^2\longrightarrow[0,\infty),\qquad
 U(g_1,g_2):=|g_1|^2,\quad V(g_1,g_2):=|g_2|^2.
\]
For $t\ge0$, polar coordinates in $\CC\simeq\RR^2$ give
\[
 \Prob\{U\le t\}
 =\frac1\pi\int_{|z|\le\sqrt t}e^{-|z|^2}\,dz
 =2\int_0^{\sqrt t}r e^{-r^2}\,dr
 =1-e^{-t}.
\]}
Thus $U$ has density $e^{-u}\mathbf 1_{(0,\infty)}(u)$. Since $g_1$ and
$g_2$ are independent standard complex Gaussian variables, $V$ has the same
density and $U,V$ are independent. Hence, using Tonelli's theorem for the nonnegative integrand
\cite[Theorem~2.37(a)]{Folland},
{\color{black}\begin{align*}
 \Prob\{|g_1g_2|\le x\}
 &=\Prob\{UV\le x^2\}
 =\int_0^\infty\int_0^{x^2/u}e^{-u-v}\,dv\,du\\
 &=\int_0^\infty e^{-u}\left(1-e^{-x^2/u}\right)du.
\end{align*}}
The last integral is estimated on three ranges.
\begin{itemize}[leftmargin=1.5em,itemsep=0.25em]
\item {\color{black}On $0<u<x^2$, use $0\le1-e^{-x^2/u}\le1$:}
\begin{equation}\label{eq:product-small-u}
 \int_0^{x^2}e^{-u}\left(1-e^{-x^2/u}\right)du
 \le\int_0^{x^2}1\,du=x^2.
\end{equation}
\item {\color{black}On $x^2\le u\le1$, use $1-e^{-t}\le t$ for $t\ge0$:}
\begin{align}
 \int_{x^2}^{1}e^{-u}\left(1-e^{-x^2/u}\right)du
 &\le x^2\int_{x^2}^{1}\frac{e^{-u}}u\,du\notag\\
 &\le x^2\int_{x^2}^{1}\frac{du}{u}
 =2x^2\log(1/x).
 \label{eq:product-middle-u}
\end{align}
\item {\color{black}On $u\ge1$, use again
$1-e^{-v}\le v$ for $v\ge0$:}
\begin{align}
 \int_{1}^{\infty}e^{-u}\left(1-e^{-x^2/u}\right)du
 &\le x^2\int_1^\infty\frac{e^{-u}}u\,du\notag\\
 &\le x^2\int_1^\infty e^{-u}\,du
 =e^{-1}x^2\le x^2.
 \label{eq:product-large-u}
\end{align}
\end{itemize}
Adding \eqref{eq:product-small-u}, \eqref{eq:product-middle-u}, and
\eqref{eq:product-large-u},
\[
 \Prob\{|g_1g_2|\le x\}
 \le2x^2+2x^2\log(1/x)
 =2x^2\log(e/x),
\]
which proves \eqref{endpoint:endpoint-product-bound}.
}
\end{proof}

\section{Proof of Theorem A}\label{sec:proof-A}

\begin{proof}
By Lemma~\ref{lem:gaussian-radius-direction}(b), the spherical probability in Theorem~\ref{thm:tail-main} equals the corresponding Gaussian probability. Let
$g=(g_\alpha)_{\alpha\in\Lambda}$ be a standard complex Gaussian vector in
$\CC^\Lambda$, and write
\[
 {\color{black}N=|\Lambda|},\qquad S=\|g\|_2^2,\qquad A=\{R_m(g)>1\},\qquad H=\{S\le2N\}.
\]
{\color{black}The event $A$ is invariant under multiplication of $g$ by a positive
scalar, so it depends only on the Gaussian direction; $H$ depends only on
the radius $\|g\|_2$. By the polar decomposition in
Lemma~\ref{lem:gaussian-radius-direction}, these two events are independent.
Moreover, since the coordinates are standard complex Gaussians,
\[
 \int_{\CC^\Lambda}\|g\|_2^2\,d\gamma_\Lambda(g)
 =\sum_{\alpha\in\Lambda}
   \int_{\CC}|z|^2\pi^{-1}e^{-|z|^2}\,dz
 =N.
\]
Markov's inequality therefore gives
\[
 \Prob(H^c)=\Prob\{S>2N\}
 \le\frac1{2N}\int_{\CC^\Lambda}S\,d\gamma_\Lambda=\frac12,
\]
and hence $\Prob(H)\ge1/2$. Independence now yields
\[
 \Prob(A\cap H)=\Prob(A)\Prob(H)\ge\frac12\Prob(A),
\]
or equivalently
\begin{equation}\label{endpoint:endpoint-radius-reduction}
 \Prob(A)\le2\Prob(A\cap H).
\end{equation}}
{\color{black}The proof separates according to the size of the support
$N=|\Lambda|$ and, when $N\ge3$, according to $N^{1/m}$. These cases exhaust
all possibilities.

\smallskip
\noindent\textbf{Case 1: $N=1$.} The ratio equals one.}\par
{\color{black}\smallskip
\noindent\textbf{Case 2: $N=2$.} Write
\[
 P_g(z)=g_\alpha z^\alpha+g_\beta z^\beta,\qquad \alpha\ne\beta.
\]
{\color{black}Choose $r$ with $\alpha_r\ne\beta_r$. Fix the other coordinates on
$\TT$. As $z_r$ runs over $\TT$, the factor
$z^{\alpha-\beta}$ runs over the whole unit circle. We may therefore choose
$z_r$ so that
\[
 g_\alpha z^\alpha\quad\text{and}\quad g_\beta z^\beta
\]
have the same argument. At that point their moduli add, and hence
\[
 \|P_g\|_\infty
 =|g_\alpha|+|g_\beta|
 =\|g\|_1
 \ge\|g\|_{q_m}.
\]}}
Hence assume $N\ge3$.

{\color{black}\smallskip
\noindent\textbf{Case 3: $N\ge3$ and $N^{1/m}\le2$.}}
Set $x=2N(N^{1/m}-1)$.
{\color{black}On $A$, Lemma~\ref{lem:counting-parseval} gives
\[
 \|P_g\|_\infty<\|g\|_{q_m}\le N^{1/(2m)}\|g\|_2
 =N^{1/(2m)}\sqrt S,
\]
so $\|P_g\|_\infty^2<N^{1/m}S$. Lemma~\ref{endpoint:endpoint-exposed-pair}
gives
\[
 S+|g_{\alpha_+}g_{\alpha_-}|
 \le\|P_g\|_\infty^2<N^{1/m}S,
\]
and therefore}
\[
 |g_{\alpha_+}g_{\alpha_-}|<(N^{1/m}-1)S.
\]
On $A\cap H$,
\begin{equation}\label{endpoint:endpoint-two-constraints}
 |g_{\alpha_+}g_{\alpha_-}|\le x,
 \qquad \|P_g\|_\infty\le\sqrt{2N^{1+1/m}}.
\end{equation}
{\color{black}
{\color{black}Separate the two distinguished coordinates by writing
\[
 u:=g_{\alpha_+},\qquad v:=g_{\alpha_-},
\]
and let $g'$ collect the remaining $N-2$ coordinates. Conditional on
$(u,v)$, the function
\[
 h_{u,v}:\TT^n\longrightarrow\CC,\qquad
 h_{u,v}(z):=u z^{\alpha_+}+v z^{\alpha_-},
\]
is fixed, whereas $g'$ still consists of independent standard complex
Gaussian coordinates. Write $\Prob_{g'}$ for probability with respect to
these remaining coordinates. Lemma~\ref{endpoint:endpoint-arbitrary-center}
then applies conditionally, with $\sigma=1$ and
\[
 s^2:=\frac{2N^{1+1/m}}{N-2},
 \qquad
 s\sqrt{N-2}=\sqrt{2N^{1+1/m}}.
\]
It gives}
\[
 \Prob_{g'}\left\{
 \|h_{u,v}+P_{g'}\|_\infty\le\sqrt{2N^{1+1/m}}
 \right\}
 \le
 \exp\left[-(N-2)\exp\left(-\frac{2N^{1+1/m}}{N-2}\right)\right]
 \le e^{-cN},
\]
where $c=\tfrac13e^{-12}$.  Indeed,
$(N-2)/N\ge1/3$ and $2N^{1+1/m}/(N-2)\le12$.
{\color{black}The right-hand side is uniform in $(u,v)$. Integrating this conditional estimate with respect to the standard complex
Gaussian measure $\gamma_{\{\alpha_+,\alpha_-\}}$ on the two distinguished
coordinates gives}
{\color{black}\begin{align*}
 &\Prob\left\{
 |g_{\alpha_+}g_{\alpha_-}|\le x,\ 
 \|P_g\|_\infty\le\sqrt{2N^{1+1/m}}
 \right\}\\
 &=\int_{\{|uv|\le x\}}
   \Prob_{g'}\left\{
   \|h_{u,v}+P_{g'}\|_\infty\le\sqrt{2N^{1+1/m}}
   \right\}
   d\gamma_{\{\alpha_+,\alpha_-\}}(u,v)\\
 &\le e^{-cN}
   \int_{\{|uv|\le x\}}
   d\gamma_{\{\alpha_+,\alpha_-\}}(u,v)
 =e^{-cN}\Prob\{|g_1g_2|\le x\}.
\end{align*}}
Using first \eqref{endpoint:endpoint-radius-reduction}, then the two
necessary conditions in \eqref{endpoint:endpoint-two-constraints}, and
finally the conditional Brascamp--Lieb bound just established, we obtain
\begin{align}
 \Prob(A)
 &\le 2\Prob(A\cap H)\notag\\
 &\le 2\Prob\left\{
 |g_{\alpha_+}g_{\alpha_-}|\le x,\ 
 \|P_g\|_\infty\le\sqrt{2N^{1+1/m}}
 \right\}\notag\\
 &\le 2e^{-cN}\Prob\{|g_1g_2|\le x\}.
 \label{endpoint:endpoint-conditional-bound}
\end{align}
}

{\color{black}Put $u=(\log N)/m$, so that $N^{1/m}=e^u$ and
$0\le u\le\log2$. On this interval,
\[
 u\le e^u-1\le2u.
\]
Since $x=2N(e^u-1)$,}
\begin{equation}\label{endpoint:endpoint-x-range}
 \frac{2N\log N}{m}\le x\le\frac{4N\log N}{m}.
\end{equation}
If $x\le1$, its lower bound gives $\log(e/x)\le C\log m$.
{\color{black}By Lemma~\ref{endpoint:endpoint-product} and the upper bound in
\eqref{endpoint:endpoint-x-range},
\[
 \Prob\{|g_1g_2|\le x\}
 \le2x^2\log(e/x)
 \le C\frac{N^2(\log N)^2}{m^2}\log m.
\]
Substituting this estimate into \eqref{endpoint:endpoint-conditional-bound}
gives}
\[
 \Prob(A)\le
 C e^{-cN}N^2(\log N)^2\frac{\log m}{m^2}
 \le C'\frac{\log m}{m^2}.
\]
If $x>1$, then \eqref{endpoint:endpoint-x-range} implies $m<4N\log N$.
In this case \eqref{endpoint:endpoint-conditional-bound} gives
\[
 \Prob(A)\le2e^{-cN}
 \le\frac{32N^2(\log N)^2e^{-cN}}{m^2}
 \le\frac{C}{m^2}.
\]

{\color{black}\smallskip
\noindent\textbf{Case 4: $2<N^{1/m}<1710$.}}
After enlarging the absolute constant to absorb finitely many degrees, assume
$m\ge1024$. The centered estimate \eqref{eq:double-exp-BL-input}, together
with \eqref{endpoint:endpoint-radius-reduction}, gives
\[
 \Prob(A)\le2\exp(-N e^{-2N^{1/m}}).
\]
Put $u:=N^{1/m}\in(2,1710)$ and $f(u):=m\log u-2u$. Then
$N e^{-2N^{1/m}}=e^{f(u)}$. Since $f$ is concave, its minimum on
$[2,1710]$ is attained at an endpoint. Moreover,
\[
 f(1710)-f(2)=m\log855-3416>0,
\]
so $f(u)\ge f(2)=m\log2-4$. Hence
\[
 \Prob(A)\le2\exp(-e^{-4}2^m).
\]

{\color{black}\smallskip
\noindent\textbf{Case 5: $N^{1/m}\ge1710$.}}
For all sufficiently large $m$, Lemma~\ref{lem:large-support-tail}, with
$\eta=1$, gives
\[
 \Prob(A)
 \le4\exp\left(-\frac{N^{1/m}4^m}{1710}\right)
 \le4\exp(-4^m).
\]

{\color{black}The five cases cover every support. Cases 1 and 2 give
$\Prob(A)=0$, Case 3 gives
$\Prob(A)\le C(\log m)/m^2$, and Cases 4 and 5 give stronger
double-exponential estimates. After increasing $C$ to cover the finitely
many excluded degrees,
\[
 \Prob(A)\le C\frac{\log m}{m^2}
\]
for every nonempty support.}
For $M\ge2$, let
\[
 \mathcal E_M:=
 \left\{(a^{(k)})_{k\ge2}\in{\color{black}\Omega_{(\Lambda_m)}}:
 R_m(a^{(m)})>1\ \text{for some }m\ge M\right\}.
\]
By subadditivity and \eqref{endpoint:endpoint-uniform-bound},
\[
 {\color{black}\boldsymbol\mu_{(\Lambda_m)}}(\mathcal E_M)
 \le C\sum_{m\ge M}\frac{\log m}{m^2}\longrightarrow0.
\]
The events $\mathcal E_M$ decrease with $M$, and
\[
 \{R_m(a^{(m)})>1\text{ for infinitely many }m\}
 =\bigcap_{M\ge2}\mathcal E_M.
\]
By continuity of probability from above,
\[
 {\color{black}\boldsymbol\mu_{(\Lambda_m)}}
 \left(\bigcap_{M\ge2}\mathcal E_M\right)
 =\lim_{M\to\infty}
 {\color{black}\boldsymbol\mu_{(\Lambda_m)}}(\mathcal E_M)=0.
\]
Hence \eqref{eq:eventual-contractivity-main} holds. The monomial example gives optimality.
\end{proof}

\section{Proof of Theorem B}\label{sec:vanishing}

\begin{proof}
Assume first that $|\Lambda_m|\to\infty$. Put
\[
 {\color{black}N_m:=|\Lambda_m|}.
\]
Fix $\varepsilon>0$ and let
$g_m=(g_{m,\alpha})_{\alpha\in\Lambda_m}$ be a standard complex Gaussian
vector in $\CC^{\Lambda_m}$. Lemma~\ref{lem:gaussian-radius-direction}(b)
gives the exact identity
\[
 \mu_{\Lambda_m}\{a\in\mathbb S_{\Lambda_m}:R_m(a)>\varepsilon\}
 =\Prob\{R_m(g_m)>\varepsilon\}.
\]
The probability on the right is estimated as follows.

Since $N_m\to\infty$, we have $N_m\ge3$ for all sufficiently large $m$.
Let
\[
 A_m:=\{R_m(g_m)>\varepsilon\},\qquad
 H_m:=\{\|g_m\|_2^2\le2N_m\}.
\]
Write
\[
 \rho_m:=\|g_m\|_2,\qquad
 \omega_m:=\frac{g_m}{\|g_m\|_2}.
\]
The polar-coordinate formula in the proof of
Lemma~\ref{lem:gaussian-radius-direction} shows that $\rho_m$ and
$\omega_m$ are independent. Since $A_m$ depends only on $\omega_m$ and
$H_m$ only on $\rho_m$, the events $A_m$ and $H_m$ are independent.
Moreover, Lemma~\ref{lem:gaussian-radius-direction}(a) gives
\[
 \Prob(H_m^c)\le(2/\e)^{N_m}<\tfrac12
\]
for all sufficiently large $m$. Hence
\begin{equation}\label{eq:vanishing-radial-reduction}
 \Prob(A_m)\le2\Prob(A_m\cap H_m).
\end{equation}

{\color{black}We split according to whether $N_m^{1/m}<1710$ or
$N_m^{1/m}\ge1710$. In the first regime we use the centered
Brascamp--Lieb estimate \eqref{eq:double-exp-BL-input}; in the second we use
Lemma~\ref{lem:large-support-tail}.

\smallskip
\noindent\textbf{Case 1: $N_m^{1/m}<1710$.}}
On $A_m\cap H_m$, Lemma~\ref{lem:counting-parseval} gives
\[
 \|P_{g_m}\|_\infty
 <\frac{\|g_m\|_{q_m}}{\varepsilon}
 \le\frac{N_m^{1/(2m)}\|g_m\|_2}{\varepsilon}
 \le\frac{\sqrt{2N_m^{1+1/m}}}{\varepsilon}.
\]
Apply \eqref{eq:double-exp-BL-input} with
$s^2=2N_m^{1/m}/\varepsilon^2$. Together with
\eqref{eq:vanishing-radial-reduction}, this yields
{\color{black}
\begin{align}
 \Prob(A_m)
 &\le2\exp\left[-N_m\exp\left(-\frac{2N_m^{1/m}}{\varepsilon^2}\right)\right]
 \le2\exp\left[-N_m\exp\left(-\frac{3420}{\varepsilon^2}\right)\right]
 \longrightarrow0.\label{eq:vanishing-small-T}
\end{align}
}

{\color{black}\smallskip
\noindent\textbf{Case 2: $N_m^{1/m}\ge1710$.}}
For all sufficiently large $m$ we have
$m\ge m_0(\varepsilon)$. Lemma~\ref{lem:large-support-tail}
therefore gives
\begin{equation}\label{eq:vanishing-large-T}
 \Prob(A_m)
 \le4\exp\left(-\frac{N_m^{1/m}4^m}{1710}\right)
 \le4\exp(-4^m)\longrightarrow0.
\end{equation}
{\color{black}These two cases exhaust all possibilities. In the first,
\eqref{eq:vanishing-small-T} tends to zero because $N_m\to\infty$; in the
second, \eqref{eq:vanishing-large-T} gives the stronger bound
$4e^{-4^m}$. Hence \textup{(ii)} follows.}

Conversely, suppose that \textup{(ii)} holds and that
$|\Lambda_m|$ does not tend to infinity. Then there are $K\in\NN$ and a
subsequence $(m_j)$ such that $|\Lambda_{m_j}|\le K$ for every $j$.
{\color{black}For $a\in\mathbb S_{\Lambda_{m_j}}$, since $q_{m_j}\le2$,
\[
 \|a\|_{q_{m_j}}\ge \|a\|_2=1,
 \qquad
 \|P_a\|_\infty\le \|a\|_1
 \le \sqrt{|\Lambda_{m_j}|}\,\|a\|_2
 \le \sqrt K.
\]
Recalling the definition of the ratio, these two estimates give directly
\[
 R_{m_j}(a)
 =\frac{\|a\|_{q_{m_j}}}{\|P_a\|_\infty}
 \ge\frac1{\sqrt K}
 \qquad(a\in\mathbb S_{\Lambda_{m_j}}).
\]
Now choose
\[
 \varepsilon:=\frac1{2\sqrt K}<\frac1{\sqrt K}.
\]
Then every $a\in\mathbb S_{\Lambda_{m_j}}$ satisfies
$R_{m_j}(a)>\varepsilon$, and hence
\[
 \mu_{\Lambda_{m_j}}
 \bigl(\{a\in\mathbb S_{\Lambda_{m_j}}:
 R_{m_j}(a)>\varepsilon\}\bigr)=1
 \qquad\text{for every }j.
\]
This contradicts \textup{(ii)}, which requires
\[
 \mu_{\Lambda_{m_j}}
 \{a\in\mathbb S_{\Lambda_{m_j}}:R_{m_j}(a)>\varepsilon\}
 \longrightarrow0
\]
for every fixed $\varepsilon>0$. Therefore
$|\Lambda_m|\to\infty$.}
\end{proof}

\section{The full polynomial space}\label{sec:full-polynomial}
{\color{black}Theorem~\ref{thm:vanishing-main} gives the qualitative criterion for the typical ratio to
vanish. For the full polynomial spaces, the first-order scale of that decay is determined explicitly. For $m,n\ge2$, set
\begin{equation}\label{full:full-entropy}
 N_{m,n}:=|\mathcal M_{m,n}|=\binom{m+n-1}{m},
 \qquad
 H_{m,n}:=\frac{n-1}{2}\log\left(m+\frac{m^2}{n}\right).
\end{equation}
The quantity $H_{m,n}$ will be the scale governing the Gaussian supremum
that enters the typical ratio. For fixed $n$,
$H_{m,n}\sim(n-1)\log m$, whereas when $n$ is comparable with or larger
than $m$,
\[
 H_{m,n}\sim\frac{n-1}{2}\log m.
\]
Throughout this section we write $N:=N_{m,n}$ and
$q_m:=2m/(m+1)$, and
\[
 G:\CC^n\longrightarrow\CC,\qquad
 G(z):=\sum_{|\alpha|=m}g_\alpha z^\alpha,
\]
where the $g_\alpha$ are independent standard complex Gaussian variables.
The symbol $\|G\|_\infty$ denotes the supremum on the unit polydisc. Unless another range is specified, every $\mathrm{o}(1)$ in this section is uniform in $n$.}

\subsection{The covariance of the exponents}

\begin{lemma}\label{full:full-metric}
{\color{black}Let $\rho_{m,n}$ be normalized counting measure on
$\mathcal M_{m,n}$ and put
\[
 \mathbf 1_n:=(1,\ldots,1)\in\mathbb R^n.
\]
Then:
\begin{enumerate}[label=\textup{(\roman*)},leftmargin=2.2em]
\item
\[
 \int_{\mathcal M_{m,n}}\alpha\,d\rho_{m,n}(\alpha)
 =\frac mn\,\mathbf 1_n,
\]
and, writing $^{\mathsf T}$ for transpose,
\[
 \int_{\mathcal M_{m,n}}
 \left(\alpha-\frac mn\mathbf 1_n\right)
 \left(\alpha-\frac mn\mathbf 1_n\right)^{\mathsf T}
 d\rho_{m,n}(\alpha)
 =
 \frac{m(m+n)}{n(n+1)}
 \left(I_n-\frac1n\mathbf 1_n\mathbf 1_n^{\mathsf T}\right).
\]

\item Let $\gamma_{m,n}$ denote the standard complex Gaussian measure on
$\CC^{\mathcal M_{m,n}}$,
\[
 d\gamma_{m,n}(g)
 :=\pi^{-N_{m,n}}e^{-\|g\|_2^2}\,dg.
\]
There is a continuous map
\[
 F:[0,2\pi]^{n-1}\longrightarrow
 L^2\!\left(\CC^{\mathcal M_{m,n}},\gamma_{m,n}\right),
\]
whose values are centered complex Gaussian random variables, such that
\begin{equation}\label{full:full-metric-estimate}
 \sup_{\theta\in[0,2\pi]^{n-1}}|F(\theta)|
 =\frac{\|G\|_\infty}{\sqrt{N_{m,n}}},
 \qquad
 \int_{\CC^{\mathcal M_{m,n}}}|F(\theta)(g)|^2\,d\gamma_{m,n}(g)=1,
\end{equation}
and, for all $\theta,\phi\in[0,2\pi]^{n-1}$,
\[
 \int_{\CC^{\mathcal M_{m,n}}}
 |F(\theta)(g)-F(\phi)(g)|^2\,d\gamma_{m,n}(g)
 \le
 \frac{m(m+n)}{n(n+1)}\,\|\theta-\phi\|_2^2.
\]
\end{enumerate}}
\end{lemma}

\begin{proof}
{\color{black}
For (i), normalized counting measure means
\[
 \int_{\mathcal M_{m,n}} f(\alpha)\,d\rho_{m,n}(\alpha)
 =\frac1{N_{m,n}}\sum_{|\alpha|=m}f(\alpha).
\]
Since $\sum_{j=1}^n\alpha_j=m$ for every $\alpha\in\mathcal M_{m,n}$ and
the coordinates have the same counting distribution,
\[
 n\int_{\mathcal M_{m,n}}\alpha_1\,d\rho_{m,n}(\alpha)
 =
 \int_{\mathcal M_{m,n}}\sum_{j=1}^n\alpha_j\,d\rho_{m,n}(\alpha)
 =m.
\]
Thus
\[
 \int_{\mathcal M_{m,n}}\alpha_j\,d\rho_{m,n}(\alpha)=\frac mn
 \qquad(1\le j\le n),
\]
which proves the first identity.

For the quadratic integrals, let
$e_1,\ldots,e_n$ denote the canonical basis of $\mathbb R^n$ and start from
\[
 \sum_{m\ge0}\left(\sum_{|\alpha|=m}x^\alpha\right)t^m
 =\prod_{j=1}^n(1-x_jt)^{-1}.
\]
Differentiating both sides twice with respect to $x_1$ gives
\[
 \frac{\partial^2}{\partial x_1^2}
 \sum_{m\ge0}\sum_{|\alpha|=m}x^\alpha t^m
 =
 \sum_{m\ge0}\sum_{|\alpha|=m}
 \alpha_1(\alpha_1-1)x^{\alpha-2e_1}t^m
\]
and
\[
 \frac{\partial^2}{\partial x_1^2}
 \prod_{j=1}^n(1-x_jt)^{-1}
 =
 2t^2(1-x_1t)^{-3}\prod_{j=2}^n(1-x_jt)^{-1}.
\]
Thus, after setting $x_1=\cdots=x_n=1$,
\[
 \sum_{m\ge0}\left(\sum_{|\alpha|=m}
 \alpha_1(\alpha_1-1)\right)t^m
 =2t^2(1-t)^{-n-2}.
\]
Comparing the coefficient of $t^m$ yields
\[
 \sum_{|\alpha|=m}\alpha_1(\alpha_1-1)
 =2\binom{m+n-1}{n+1}.
\]
For the mixed derivative,
\[
 \frac{\partial^2}{\partial x_1\partial x_2}
 \sum_{m\ge0}\sum_{|\alpha|=m}x^\alpha t^m
 =
 \sum_{m\ge0}\sum_{|\alpha|=m}
 \alpha_1\alpha_2x^{\alpha-e_1-e_2}t^m,
\]
whereas
\[
 \frac{\partial^2}{\partial x_1\partial x_2}
 \prod_{j=1}^n(1-x_jt)^{-1}
 =
 t^2(1-x_1t)^{-2}(1-x_2t)^{-2}
 \prod_{j=3}^n(1-x_jt)^{-1}.
\]
Setting all $x_j=1$ therefore gives
\[
 \sum_{m\ge0}\left(\sum_{|\alpha|=m}\alpha_1\alpha_2\right)t^m
 =t^2(1-t)^{-n-2},
\]
and comparison of the coefficient of $t^m$ gives
\[
 \sum_{|\alpha|=m}\alpha_1\alpha_2
 =\binom{m+n-1}{n+1}.
\]
Dividing by
$N_{m,n}=\binom{m+n-1}{n-1}$,
\[
 \int\alpha_1(\alpha_1-1)\,d\rho_{m,n}
 =\frac{2m(m-1)}{n(n+1)},
 \qquad
 \int\alpha_1\alpha_2\,d\rho_{m,n}
 =\frac{m(m-1)}{n(n+1)}.
\]
Hence
\[
 \int\alpha_1^2\,d\rho_{m,n}
 =\frac{2m(m-1)}{n(n+1)}+\frac mn.
\]
Subtracting $m^2/n^2$ from the diagonal entries and from the off-diagonal
entries gives, respectively,
\[
 \frac{m(m+n)(n-1)}{n^2(n+1)}
 \quad\text{and}\quad
 -\frac{m(m+n)}{n^2(n+1)}.
\]
These are exactly the entries of
\[
 \frac{m(m+n)}{n(n+1)}
 \left(I_n-\frac1n\mathbf 1_n\mathbf 1_n^{\mathsf T}\right).
\]

For (ii), put $d:=n-1$. If
$z=(e^{i\vartheta_1},\ldots,e^{i\vartheta_n})\in\mathbb T^n$, then
homogeneity of degree $m$ gives
\[
 G(z)
 =e^{im\vartheta_n}
 G(e^{i(\vartheta_1-\vartheta_n)},\ldots,
   e^{i(\vartheta_{n-1}-\vartheta_n)},1).
\]
Therefore taking the supremum over $\mathbb T^n$ is equivalent to taking the
supremum over the first $d$ angular differences with the last coordinate
equal to $1$. Define
\[
 F(\theta):=
 \frac1{\sqrt N}
 \exp\left(-\frac{im}{n}\sum_{j=1}^d\theta_j\right)
 G(e^{i\theta_1},\ldots,e^{i\theta_d},1),
 \qquad \theta\in[0,2\pi]^d.
\]
{\color{black}\hypersetup{linkcolor=black,citecolor=black}The exponential factor} has modulus one, so
\[
 \sup_{\theta\in[0,2\pi]^d}|F(\theta)|
 =\frac{\|G\|_\infty}{\sqrt N}.
\]
Expanding $F$,
\[
 F(\theta)
 =\frac1{\sqrt N}\sum_{|\alpha|=m}g_\alpha
 \exp\left(
 i\sum_{j=1}^d\left(\alpha_j-\frac mn\right)\theta_j
 \right).
\]
Since the $g_\alpha$ are independent standard complex Gaussians,
\[
 \int_{\CC^{\mathcal M_{m,n}}}|F(\theta)(g)|^2\,d\gamma_{m,n}(g)
 =\frac1N\sum_{|\alpha|=m}1=1.
\]
Moreover, using $|e^{iu}-e^{iv}|\le|u-v|$,
\begin{align*}
 \int_{\CC^{\mathcal M_{m,n}}}
 |F(\theta)(g)-F(\phi)(g)|^2\,d\gamma_{m,n}(g)
 &\le
 \frac1N\sum_{|\alpha|=m}
 \left|
 \sum_{j=1}^d\left(\alpha_j-\frac mn\right)
 (\theta_j-\phi_j)
 \right|^2\\
 &=
 \int_{\mathcal M_{m,n}}
 \left|
 \sum_{j=1}^d\left(\alpha_j-\frac mn\right)
 (\theta_j-\phi_j)
 \right|^2\,d\rho_{m,n}(\alpha).
\end{align*}
Applying the covariance identity in (i) to the vector
$(\theta_1-\phi_1,\ldots,\theta_d-\phi_d,0)\in\mathbb R^n$ gives
\[
 \int_{\CC^{\mathcal M_{m,n}}}
 |F(\theta)(g)-F(\phi)(g)|^2\,d\gamma_{m,n}(g)
 \le
 \frac{m(m+n)}{n(n+1)}\|\theta-\phi\|_2^2.
\]
This proves (ii).
}
\end{proof}

\subsection{The upper estimate}

\begin{lemma}\label{full:full-upper}
{\color{black}
Let $N_{m,n}$ and $H_{m,n}$ be given by \eqref{full:full-entropy}. Let
$g=(g_\alpha)_{|\alpha|=m}$ have independent standard complex Gaussian
coordinates, let $\gamma_{m,n}$ be the corresponding standard complex
Gaussian measure on $\CC^{\mathcal M_{m,n}}$, and set
\[
 G(z):=\sum_{|\alpha|=m}g_\alpha z^\alpha.
\]
Uniformly for $n\ge2$,
\[
 \int_{\CC^{\mathcal M_{m,n}}}\|G\|_\infty\,d\gamma_{m,n}
 \le(1+\mathrm{o}(1))\sqrt{N_{m,n}H_{m,n}}.
\]
Moreover, for every fixed $x>1$,
\[
 \Prob\{\|G\|_\infty>x\sqrt{N_{m,n}H_{m,n}}\}
 \le\exp\bigl(-(x^2-1+\mathrm{o}(1))H_{m,n}\bigr).
\]
}
\end{lemma}

\begin{proof}
{\color{black}
We apply a discretized chaining argument to the Gaussian field $F$ from
Lemma~\ref{full:full-metric}(ii); compare \cite[Chapter~13]{BLM}. Put
\[
 d:=n-1,\qquad N:=N_{m,n},\qquad H:=H_{m,n},\qquad
 A^2:=d\,\frac{m(m+n)}{n(n+1)}.
\]
The metric estimate in Lemma~\ref{full:full-metric}(ii) says that
\[
 \int_{\CC^{\mathcal M_{m,n}}}
 |F(\theta)(g)-F(\phi)(g)|^2\,d\gamma_{m,n}(g)
 \le \frac{A^2}{d}\|\theta-\phi\|_2^2.
\]
Since $d=n-1$,
\[
 A^2=\left(m+\frac{m^2}{n}\right)\frac{n-1}{n+1}.
\]
For $n\ge2$ this gives
\[
 \frac m3\le A^2\le m^2.
\]
The parameter $A$ is therefore a convenient uniform bound for the metric
scale of the field.

Set $\delta:=1/\log m$. Partition each interval $[0,2\pi]$ into
subintervals of length at most $\delta/A$, and let
$\Gamma_0\subset[0,2\pi]^d$ be the Cartesian product of the resulting
one-dimensional grids. For each $\theta\in[0,2\pi]^d$, choose a nearest
grid point
\[
 \theta^{(0)}\in\Gamma_0.
\]
By construction,
\[
 \|\theta-\theta^{(0)}\|_2
 \le C\sqrt d\,\frac{\delta}{A}.
\] Moreover,
\[
 |\Gamma_0|\le(CA/\delta)^d.
\]
Since
\[
 H=\frac d2\log\left(m+\frac{m^2}{n}\right)
\]
and $A^2$ differs from $m+m^2/n$ only by the factor
$(n-1)/(n+1)$, we obtain
\begin{equation}\label{full:full-grid-size}
 \log|\Gamma_0|
 \le d\log(CA/\delta)
 =H+O(d\log\log m)
 =(1+\mathrm{o}(1))H.
\end{equation}

The grid controls the values of $F$ at finitely many points; it remains to
control the error between a point and its nearest grid point. Define
\[
 Z:=\sup_{\theta\in[0,2\pi]^d}
 |F(\theta)-F(\theta^{(0)})|.
\]
The oscillation satisfies
\begin{equation}\label{full:full-grid-oscillation}
 \int Z\,d\gamma_{m,n}\le C\delta\sqrt H,
 \qquad
 \gamma_{m,n}\left\{
 Z>\int Z\,d\gamma_{m,n}+t
 \right\}\le e^{-c t^2/\delta^2}.
\end{equation}

For the integral estimate, for every $k\ge1$, let $\Gamma_k$ be a product
grid with coordinate mesh at most $\delta2^{-k}/A$. For each $\theta$ choose
a nearest point $\theta^{(k)}\in\Gamma_k$. Then
\[
 \|\theta^{(k)}-\theta^{(k-1)}\|_2
 \le
 \|\theta^{(k)}-\theta\|_2+\|\theta-\theta^{(k-1)}\|_2
 \le C\sqrt d\,\frac{\delta2^{-k}}A,
\]
after adjusting the absolute constant $C$.
Hence the metric estimate gives, with the constants displayed,
\begin{align*}
 &\int_{\CC^{\mathcal M_{m,n}}}
 \left|F(\theta^{(k)})(g)-F(\theta^{(k-1)})(g)\right|^2
 \,d\gamma_{m,n}(g)\\
 &\qquad\le
 \frac{A^2}{d}\,
 \|\theta^{(k)}-\theta^{(k-1)}\|_2^2\\
 &\qquad\le
 \frac{A^2}{d}
 \left(C\sqrt d\,\frac{\delta2^{-k}}A\right)^2
 =C^2\delta^2\,4^{-k}.
\end{align*}
As $\theta$ varies, at most $|\Gamma_k||\Gamma_{k-1}|$ different pairs
$(\theta^{(k)},\theta^{(k-1)})$ occur. Since
\[
 |\Gamma_k|\le(CA2^k/\delta)^d,
\]
the logarithm of the number of such pairs is bounded by
\[
 C d\,[\log(CA/\delta)+k].
\]

If $Y_1,\ldots,Y_M$ are centered {\color{black}\hypersetup{linkcolor=black,citecolor=black}circularly symmetric} complex Gaussian variables satisfying
\[
 \E|Y_j|^2\le s^2,
\]
then the Gaussian tail bound and the union bound give
\[
 \Prob\{\max_{j\le M}|Y_j|>u\}
 \le M e^{-u^2/s^2}.
\]
{\color{black}\hypersetup{linkcolor=black,citecolor=black}Splitting the tail integral at}
$u_0=s\sqrt{\log M}$ gives
\[
 \E\max_{j\le M}|Y_j|
 \le s\bigl(\sqrt{\log M}+C\bigr).
\]
Applying this estimate to the increments at level $k$ yields
\[
 \int
 \max_\theta|F(\theta^{(k)})-F(\theta^{(k-1)})|
 \,d\gamma_{m,n}
 \le
 C\delta2^{-k}
 \sqrt{d[\log(CA/\delta)+k]}.
\]

For each fixed $\theta$, the mesh tends to zero, so
$\theta^{(k)}\to\theta$. By continuity of $F$,
$F(\theta^{(k)})\to F(\theta)$, and therefore
\[
 F(\theta)-F(\theta^{(0)})
 =
 \sum_{k\ge1}
 \left(F(\theta^{(k)})-F(\theta^{(k-1)})\right).
\]
Taking absolute values, then the supremum over $\theta$, and finally
integrating, we obtain
\begin{align*}
 \int Z\,d\gamma_{m,n}
 &\le
 C\delta\sum_{k\ge1}2^{-k}
 \sqrt{d[\log(CA/\delta)+k]}\\
 &\le
 C\delta\sqrt{d\log(CA/\delta)}
 \le C\delta\sqrt H.
\end{align*}
Hence the first part of \eqref{full:full-grid-oscillation} holds.

For the concentration estimate, consider first the initial increment. For every $\theta$,
\begin{align*}
 &\int_{\CC^{\mathcal M_{m,n}}}
 |F(\theta)(g)-F(\theta^{(0)})(g)|^2\,d\gamma_{m,n}(g)\\
 &\qquad\le
 \frac{A^2}{d}\|\theta-\theta^{(0)}\|_2^2\\
 &\qquad\le
 \frac{A^2}{d}
 \left(C\sqrt d\,\frac{\delta}{A}\right)^2
 =C^2\delta^2.
\end{align*}
To identify the Lipschitz constant, write the increment as a linear
functional of the coefficient vector:
\[
 F(\theta)(g)-F(\theta^{(0)})(g)
 =\sum_{|\alpha|=m}c_\alpha(\theta)\,g_\alpha.
\]
For standard complex Gaussian coefficients, independence and
$\int_{\CC}|z|^2\pi^{-1}e^{-|z|^2}\,dz=1$ give
\[
 \int_{\CC^{\mathcal M_{m,n}}}
 \left|\sum_{|\alpha|=m}c_\alpha(\theta)g_\alpha\right|^2
 \,d\gamma_{m,n}(g)
 =\sum_{|\alpha|=m}|c_\alpha(\theta)|^2.
\]
Hence
\[
 \sum_{|\alpha|=m}|c_\alpha(\theta)|^2
 \le C^2\delta^2.
\]
Hence, for coefficient vectors $g,h$,
\begin{align*}
 &\left|
 \bigl(F(\theta)(g)-F(\theta^{(0)})(g)\bigr)
 -
 \bigl(F(\theta)(h)-F(\theta^{(0)})(h)\bigr)
 \right|\\
 &\qquad\le
 \left(\sum_{|\alpha|=m}|c_\alpha(\theta)|^2\right)^{1/2}
 \|g-h\|_2
 \le C\delta\,\|g-h\|_2.
\end{align*}
Taking absolute values does not increase the Lipschitz constant, since
$||z|-|w||\le|z-w|$. Taking the supremum over $\theta$ also preserves the
common bound. Thus
\[
 |Z(g)-Z(h)|\le C\delta\,\|g-h\|_2.
\]
The Gaussian Lipschitz concentration inequality
\cite[Section~5.4]{BLM} now gives, with an absolute constant $c>0$,
\[
 \gamma_{m,n}\left\{
 Z>\int Z\,d\gamma_{m,n}+t
 \right\}
 \le\exp\left(-\frac{c\,t^2}{\delta^2}\right).
\]
Thus the probability estimate in
\eqref{full:full-grid-oscillation} holds, after changing the absolute
constant $c$ if necessary.

For the supremum, since
\[
 \sup_{\theta}|F(\theta)|
 \le\max_{\theta\in\Gamma_0}|F(\theta)|+Z,
\]
and each $F(\theta)$ is a standard complex Gaussian variable,
\[
 \gamma_{m,n}\left\{
 \max_{\theta\in\Gamma_0}|F(\theta)|>u
 \right\}
 \le|\Gamma_0|e^{-u^2}.
\]
Integrating this bound and using
\eqref{full:full-grid-size} and
\eqref{full:full-grid-oscillation},
\[
 \int\sup_\theta|F(\theta)|\,d\gamma_{m,n}
 \le\sqrt{\log|\Gamma_0|}+C+\int Z\,d\gamma_{m,n}
 \le(1+\mathrm{o}(1))\sqrt H.
\]
Since
$\|G\|_\infty=\sqrt N\,\sup_\theta|F(\theta)|$, we obtain
\[
 \int_{\CC^{\mathcal M_{m,n}}}\|G\|_\infty\,d\gamma_{m,n}
 \le(1+\mathrm{o}(1))\sqrt{NH}.
\]

For fixed $x>1$, apply
$\sup_\theta|F(\theta)|\le\max_{\theta\in\Gamma_0}|F(\theta)|+Z$
with grid threshold $(x-\sqrt\delta)\sqrt H$ and use
\eqref{full:full-grid-oscillation} for the remaining error. Then
\begin{align*}
 \Prob\{\sup_\theta|F(\theta)|>x\sqrt H\}
 &\le
 |\Gamma_0|e^{-(x-\sqrt\delta)^2H}
 +e^{-cH/\delta}\\
 &\le
 \exp\bigl(-(x^2-1+\mathrm{o}(1))H\bigr),
\end{align*}
uniformly in $n$. Multiplying the threshold by $\sqrt N$ gives
\[
 \Prob\{\|G\|_\infty>x\sqrt{NH}\}
 \le\exp\bigl(-(x^2-1+\mathrm{o}(1))H\bigr),
\]
which is the asserted upper-tail estimate.
}
\end{proof}

\subsection{A covariance estimate in large dimension}

\begin{lemma}\label{full:full-covariance-volume}
{\color{black}Define the normalized covariance kernel}
\[
 K_{m,n}(\theta):=\frac1{N_{m,n}}\sum_{|\alpha|=m}
 \e^{i\langle\alpha,\theta\rangle}\qquad(\theta\in\mathbb T^n).
\]
If $n\ge(\log m)^2$, then
\begin{equation}\label{full:full-bad-volume}
 m_n\left(\left\{\theta\in\TT^n:
 |K_{m,n}(\theta)|>\frac1{\log m}\right\}\right)
 \le\exp\bigl(-(1-\mathrm{o}(1))H_{m,n}\bigr),
\end{equation}
and the error is uniform
in this range of dimensions.
\end{lemma}

\begin{proof}
{\color{black}Write
\[
N:=N_{m,n},\qquad H:=H_{m,n},\qquad \rho:=\frac1{\log m}.
\]
For $u=(u_1,\ldots,u_n)\in\CC^n$ and $k\ge0$, let
\[
h_k(u):=\sum_{|\alpha|=k}u^\alpha.
\]
For the coefficient sum defining the kernel, summing over all
degrees and then over all multiindices,
\[
\sum_{k\ge0}h_k(u)\zeta^k
=\sum_{\alpha\in\mathbb N_0^n}(u\zeta)^\alpha
=\prod_{j=1}^n\left(\sum_{\ell\ge0}(u_j\zeta)^\ell\right)
=\prod_{j=1}^n(1-u_j\zeta)^{-1},
\]
for $|\zeta|$ sufficiently small. Thus $h_m(u)$ is the coefficient of
$\zeta^m$ in the last product. Extract this coefficient on the
circle of radius
\[
r:=\frac{m}{m+n},
\qquad
s^2:=\frac{r}{(1-r)^2}=\frac{m+m^2/n}{n}.
\]
The normalization in Cauchy's formula satisfies the following estimates.
Since
\[
\binom{m+n}{m}=\frac{m+n}{n}\,N,
\]
the standard two-sided form of Stirling's formula \cite[Section~5.11]{NIST}, applied to the three
factorials in the binomial coefficient, gives
\begin{equation}\label{full:full-stirling-binomial}
 \frac{m+n}{n}\,N
 =\binom{m+n}{m}
 \ge
 c\sqrt{\frac{m+n}{mn}}\,
 \frac{(m+n)^{m+n}}{m^m n^n}.
\end{equation}
Solving \eqref{full:full-stirling-binomial} for $1/N$ gives
\begin{equation}\label{full:full-stirling-normalization}
 \frac1N
 \le
 C\sqrt{\frac{m(m+n)}n}\,
 \frac{m^m n^n}{(m+n)^{m+n}}.
\end{equation}
On the other hand, the choice $r=m/(m+n)$ gives
\begin{equation}\label{full:full-r-factor}
 r^{-m}(1-r)^{-n}
 =\frac{(m+n)^{m+n}}{m^m n^n}.
\end{equation}
Multiplying \eqref{full:full-stirling-normalization} and
\eqref{full:full-r-factor}, we obtain
\begin{equation}\label{full:full-cauchy-factor}
 \frac{r^{-m}(1-r)^{-n}}{N}
 \le C\sqrt{\frac{m(m+n)}n}
 =C\sqrt{m+\frac{m^2}{n}}.
\end{equation}
}

{\color{black}Cauchy's coefficient formula on $|\zeta|=r$ now gives
\[
 h_m(u)=\frac{r^{-m}}{2\pi}\int_0^{2\pi}
 e^{-imt}\prod_{j=1}^n(1-r u_j e^{it})^{-1}\,dt,
\]
and hence
\[
 |h_m(u)|
 \le \frac{r^{-m}}{2\pi}\int_0^{2\pi}
       \prod_{j=1}^n|1-r u_j e^{it}|^{-1}\,dt
 \le r^{-m}\max_{t\in[0,2\pi]}
       \prod_{j=1}^n|1-r u_j e^{it}|^{-1}.
\]
For $u_j=e^{i\phi_j}$,
\[
 |1-r e^{i\phi_j}|^2
 =(1-r)^2+4r\sin^2\frac{\phi_j}{2}
 =(1-r)^2\left(1+4s^2\sin^2\frac{\phi_j}{2}\right),
\]
where $s^2=r/(1-r)^2$. Define
\[
 \mathbf1:=(1,\ldots,1)\in\mathbb R^n,\qquad
 Q_s(\phi):=\sum_{j=1}^n
 \log\left(1+4s^2\sin^2\frac{\phi_j}{2}\right).
\]
For $\theta\in\mathbb T^n$, the Cauchy estimate gives
\[
 |K_{m,n}(\theta)|
 \le
 \frac{r^{-m}}{N}
 \max_{t\in[0,2\pi]}
 \prod_{j=1}^n|1-r e^{i(\theta_j+t)}|^{-1}.
\]
Applying
\[
 |1-r e^{i\psi}|^{-1}
 =(1-r)^{-1}
 \left(1+4s^2\sin^2\frac{\psi}{2}\right)^{-1/2}
\]
with $\psi=\theta_j+t$ to each factor gives
\[
 \prod_{j=1}^n|1-r e^{i(\theta_j+t)}|^{-1}
 =(1-r)^{-n}
 \exp\left[-\frac12Q_s(\theta+t\mathbf1)\right].
\]
Hence
\[
 \max_{t\in[0,2\pi]}
 \prod_{j=1}^n|1-r e^{i(\theta_j+t)}|^{-1}
 =(1-r)^{-n}
 \exp\left[-\frac12\min_{t\in[0,2\pi]}
 Q_s(\theta+t\mathbf1)\right].
\]
{\color{black}\hypersetup{linkcolor=black,citecolor=black}Using \eqref{full:full-cauchy-factor}, we obtain}
\[
 |K_{m,n}(\theta)|
 \le C\sqrt{m+\frac{m^2}{n}}\,
 \exp\left[-\frac12\min_{t\in[0,2\pi]}
 Q_s(\theta+t\mathbf1)\right].
\]}
Thus $|K_{m,n}(\theta)|>\rho$ implies
\begin{equation}\label{full:full-phase-minimum}
 \min_{t\in[0,2\pi]} Q_s(\theta+t\mathbf1)\le B_0,
 \qquad
 B_0:=\log\!\left(m+\frac{m^2}{n}\right)+2\log(C/\rho).
\end{equation}
{\color{black}For fixed $\theta$, discretize the minimizing phase by writing
\[
 q_j(t):=
 \log\left(1+4s^2\sin^2\frac{\theta_j+t}{2}\right),
 \qquad
 Q_s(\theta+t\mathbf1)=\sum_{j=1}^n q_j(t).
\]
Differentiating,
\[
 q_j'(t)
 =
 \frac{2s^2\sin(\theta_j+t)}
 {1+4s^2\sin^2((\theta_j+t)/2)}.
\]
Put
\[
 u:=2s\sin\frac{\theta_j+t}{2}.
\]
Since
\[
 |\sin(\theta_j+t)|
 \le2\left|\sin\frac{\theta_j+t}{2}\right|,
\]
we obtain
\[
 |q_j'(t)|
 \le\frac{2s|u|}{1+u^2}
 \le s,
\]
because $2|u|\le1+u^2$. Consequently,
\begin{equation}\label{full:full-phase-lipschitz}
 \left|\frac{d}{dt}Q_s(\theta+t\mathbf1)\right|
 \le\sum_{j=1}^n|q_j'(t)|
 \le ns.
\end{equation}
Thus $t\mapsto Q_s(\theta+t\mathbf1)$ is $ns$-Lipschitz.

Choose a grid $\mathcal T\subset[0,2\pi]$ with spacing at most
$1/(ns)$. It may be chosen with
\[
 |\mathcal T|\le Cns+1\le C\sqrt{m(m+n)}.
\]
If $t_*$ minimizes $Q_s(\theta+t\mathbf1)$, choose
$t_\ell\in\mathcal T$ with $|t_\ell-t_*|\le1/(ns)$. Then, by \eqref{full:full-phase-lipschitz},
\[
 Q_s(\theta+t_\ell\mathbf1)
 \le Q_s(\theta+t_*\mathbf1)+ns\,|t_\ell-t_*|
 \le \min_tQ_s(\theta+t\mathbf1)+1.
\]
Combining this with \eqref{full:full-phase-minimum}, whenever
$|K_{m,n}(\theta)|>\rho$ one of the grid phases satisfies
\begin{equation}\label{full:full-phase-grid}
 Q_s(\theta+t_\ell\mathbf1)\le B_0+1.
\end{equation}
}

Put $B=B_0+1$ and $p=n/(2B)$. Since $B=O(\log m)$ and
$n\ge(\log m)^2$, we have $p\ge1$ for all sufficiently large $m$.
Normalized Lebesgue measure on the circle gives
\begin{equation}\label{full:full-one-dimensional-integral}
\begin{aligned}
 \frac1{2\pi}\int_0^{2\pi}
 \left(1+4s^2\sin^2\frac t2\right)^{-p}\,dt
 &\le C\int_0^\infty(1+cs^2t^2)^{-p}\,dt\\
 &\le\frac{C}{s\sqrt p}.
\end{aligned}
\end{equation}
The first inequality uses $\sin(t/2)\ge t/\pi$ on $[0,\pi]$;
the second follows from
\[
 \int_0^\infty(1+t^2)^{-p}\,dt
 =\frac{\sqrt\pi\,\Gamma(p-1/2)}{2\Gamma(p)}
 \le\frac C{\sqrt p},\qquad p\ge1.
\]
Applying Markov's inequality \cite{BLM} with product Haar measure, using
\eqref{full:full-one-dimensional-integral} in each coordinate and then the
phase grid from \eqref{full:full-phase-grid}, gives
\begin{align*}
 m_n\{\theta\in\TT^n:|K_{m,n}(\theta)|>\rho\}
 &\le C\sqrt{m(m+n)}\,
       \e^{pB}\left(\frac C{s\sqrt p}\right)^n\\
 &\le C\sqrt{m(m+n)}
       \left(\frac{C\sqrt B}{\sqrt{m+m^2/n}}\right)^n.
\end{align*}
Taking logarithms, the right side has logarithm at most
\[
 -\frac n2\log\!\left(m+\frac{m^2}{n}\right)
 +\frac n2\log B+Cn+O(\log m+\log n).
\]
Here $\log B=O(\log\log m)$,
$\log(m+m^2/n)\ge\log m$, and $n\ge(\log m)^2$. All positive terms are
\[
 \mathrm{o}\!\left(n\log\!\left(m+\frac{m^2}{n}\right)\right)
\]
uniformly in this range. Since $n/(n-1)\to1$ uniformly as well,
\eqref{full:full-bad-volume} follows.
\end{proof}

\begin{lemma}\label{full:full-large-lower}
{\color{black}Let $N_{m,n}$ and $H_{m,n}$ be given by \eqref{full:full-entropy}. Let
$g=(g_\alpha)_{|\alpha|=m}$ have independent standard complex Gaussian
coordinates, let $\gamma_{m,n}$ be the corresponding standard complex
Gaussian measure on $\CC^{\mathcal M_{m,n}}$, and set
\[
 G(z):=\sum_{|\alpha|=m}g_\alpha z^\alpha.
\]
If $n\ge(\log m)^2$, then
\[
 \int_{\CC^{\mathcal M_{m,n}}}\|G\|_\infty\,d\gamma_{m,n}
 \ge(1-\mathrm{o}(1))\sqrt{N_{m,n}H_{m,n}}.
\]
For every fixed $x>1$,
\[
 \Prob\{\|G\|_\infty>x\sqrt{N_{m,n}H_{m,n}}\}
 \ge\exp\bigl(-(x^2-1+\mathrm{o}(1))H_{m,n}\bigr).
\]
The error is uniform in this range of dimensions.
}
\end{lemma}

\begin{proof}
{\color{black}
By Lemma~\ref{full:full-covariance-volume}, we first extract a large family
of weakly correlated torus points and then apply Slepian's inequality
\cite{Slepian} to the corresponding Gaussian evaluations; compare
\cite[Section~13.2]{BLM}. Write
\[
 N:=N_{m,n},\qquad H:=H_{m,n},\qquad \rho:=\frac1{\log m}.
\]
Consider many points of the torus whose pairwise correlations are
small. Define the high-correlation set
\[
 \mathcal B:=\{\theta\in\TT^n:|K_{m,n}(\theta)|>\rho\},
 \qquad v:=m_n(\mathcal B).
\]
By
Lemma~\ref{full:full-covariance-volume},
\[
 v\le \exp\bigl(-(1-\mathrm{o}(1))H\bigr).
\]

We choose the points recursively. Start with an arbitrary
$\theta^{(1)}\in\mathbb T^n$. Suppose
$\theta^{(1)},\ldots,\theta^{(j)}$ have already been chosen so that every
pairwise difference lies outside $\mathcal B$. At stage $j+1$, a candidate
$\theta\in\TT^n$ is inadmissible precisely when
$\theta-\theta^{(k)}\in\mathcal B$ for at least one $1\le k\le j$.
Thus all inadmissible points lie in
\[
 \bigcup_{k=1}^j\bigl(\theta^{(k)}+\mathcal B\bigr).
\]
Translation invariance of Haar measure gives
\[
 m_n\left(\bigcup_{k=1}^j(\theta^{(k)}+\mathcal B)\right)
 \le\sum_{k=1}^j m_n(\theta^{(k)}+\mathcal B)
 =jv.
\]
If
\[
 j<\left\lfloor\frac1{2v}\right\rfloor,
\]
then $jv\le1/2<1$, so
$\bigcup_{k=1}^j(\theta^{(k)}+\mathcal B)$ cannot fill $\TT^n$. Hence another
admissible point $\theta^{(j+1)}$ can be chosen. Iterating gives at least
$\lfloor1/(2v)\rfloor$ points with
\[
 |K_{m,n}(\theta^{(j)}-\theta^{(k)})|\le\rho
 \qquad(j\ne k).
\]
From Lemma~\ref{full:full-covariance-volume}, there is a nonnegative
sequence $\varepsilon_m\to0$, independent of $n$ in the present range, such
that
\[
 v\le \exp\bigl(-(1-\varepsilon_m)H\bigr).
\]
Set
\[
 \delta_m:=\varepsilon_m+\frac1{\log m}
\]
and choose
\begin{equation}\label{eq:full-large-lower-L}
 L:=\left\lfloor
 \exp\bigl((1-\delta_m)H\bigr)
 \right\rfloor.
\end{equation}
Since $n\ge(\log m)^2$,
\[
 H=\frac{n-1}{2}\log\left(m+\frac{m^2}{n}\right)
 \ge \frac{(\log m)^2-1}{2}\log m,
\]
and hence
\[
 \frac{H}{\log m}\longrightarrow\infty
\]
uniformly in this range. Therefore
\[
 2Lv
 \le
 2\exp\bigl(-(\delta_m-\varepsilon_m)H\bigr)
 =
 2\exp\left(-\frac{H}{\log m}\right)
 \longrightarrow0.
\]
Thus $L\le(2v)^{-1}$ for all sufficiently large $m$, and the recursive
construction yields $\theta^{(1)},\ldots,\theta^{(L)}$ with
\[
 |K_{m,n}(\theta^{(j)}-\theta^{(k)})|\le\rho\qquad(j\ne k).
\]
Since $\delta_m\to0$ and $H\to\infty$ uniformly,
\[
 \log L=(1+\mathrm{o}(1))H.
\]

From these phases define the real Gaussian variables
\[
 X_j:=\sqrt{\frac2N}\operatorname{Re}
 G(e^{i\theta^{(j)}_1},\ldots,e^{i\theta^{(j)}_n}),
 \qquad 1\le j\le L.
\]
Each $X_j$ is centered and
\[
 \int_{\CC^{\mathcal M_{m,n}}}X_j^2\,d\gamma_{m,n}=1.
\]
For $j\ne k$, expanding the evaluations and integrating the Gaussian
coefficients gives
\begin{align*}
 \int_{\CC^{\mathcal M_{m,n}}}X_jX_k\,d\gamma_{m,n}
 &=\frac1N\operatorname{Re}
   \sum_{|\alpha|=m}
   e^{i\langle\alpha,\theta^{(j)}-\theta^{(k)}\rangle}\\
 &=\operatorname{Re}
   K_{m,n}(\theta^{(j)}-\theta^{(k)}).
\end{align*}
Consequently,
\[
 \int_{\CC^{\mathcal M_{m,n}}}X_jX_k\,d\gamma_{m,n}
 \le |K_{m,n}(\theta^{(j)}-\theta^{(k)})|
 \le\rho\qquad(j\ne k).
\]
Thus the construction has produced a large real Gaussian family with
diagonal integrals equal to $1$ and off-diagonal integrals at most $\rho$.
}
{\color{black}Let
\[
 Y_j:=\sqrt\rho\,\xi+\sqrt{1-\rho}\,\xi_j,
\]
where $\xi,\xi_1,\ldots,\xi_L$ are independent standard real Gaussians.
Direct calculation gives
\[
 \E Y_j^2=1,
 \qquad
 \E(Y_jY_k)=\rho\quad(j\ne k).
\]
Thus the diagonal integrals of $(X_j)$ and $(Y_j)$ agree, while the
off-diagonal integrals of $(X_j)$ are no larger. Slepian's comparison
inequality \cite{Slepian} therefore yields
\[
 \int\max_{j\le L}X_j\,d\gamma_{m,n}
 \ge
 \E\max_{j\le L}Y_j.
\]
Since the common variable $\xi$ is added to every coordinate,
\[
 \max_{j\le L}Y_j
 =\sqrt\rho\,\xi+\sqrt{1-\rho}\max_{j\le L}\xi_j.
\]
Taking expectations and using $\E\xi=0$ gives
\[
 \E\max_{j\le L}Y_j
 =
 \sqrt{1-\rho}\,\E\max_{j\le L}\xi_j.
\]
For independent standard real Gaussians,
\[
 \E\max_{j\le L}\xi_j
 =(1+\mathrm{o}(1))\sqrt{2\log L};
\]
see \cite[Section~2.5 and Exercise~2.17]{BLM}.
On the other hand, pointwise in the coefficient vector,
\[
 \|G\|_\infty
 \ge \max_{j\le L}|G(e^{i\theta^{(j)}})|
 \ge \max_{j\le L}\operatorname{Re}G(e^{i\theta^{(j)}})
 =\sqrt{\frac N2}\max_{j\le L}X_j.
\]
Integrating this inequality and using
\[
 \int\max_{j\le L}X_j\,d\gamma_{m,n}
 \ge\sqrt{1-\rho}\,\E\max_{j\le L}\xi_j,
\]
we obtain
\begin{align*}
 \frac1{\sqrt N}
 \int_{\CC^{\mathcal M_{m,n}}}\|G\|_\infty\,d\gamma_{m,n}
 &\ge\frac1{\sqrt2}
   \int\max_{j\le L}X_j\,d\gamma_{m,n}\\
 &\ge\sqrt{\frac{1-\rho}{2}}\,
   \E\max_{j\le L}\xi_j\\
 &\ge(1-\mathrm{o}(1))\sqrt H,
\end{align*}
because $\log L=(1+\mathrm{o}(1))H$ and $\rho\to0$.}

{\color{black}
For the probability estimate, Slepian's inequality gives
\[
 \Prob\{\max_jX_j>t\}\ge\Prob\{\max_jY_j>t\}.
\]
On the event $\{\xi\ge0\}$,
\[
 Y_j=\sqrt\rho\,\xi+\sqrt{1-\rho}\,\xi_j
 \ge\sqrt{1-\rho}\,\xi_j.
\]
Hence, with $t=x\sqrt{2H}$,
\begin{align*}
 \Prob\{\|G\|_\infty>x\sqrt{NH}\}
 &\ge\Prob\{\max_jY_j>x\sqrt{2H}\}\\
 &\ge\Prob\left\{\xi\ge0,
       \max_j\xi_j>x\sqrt{\frac{2H}{1-\rho}}\right\}\\
 &=\frac12\Prob\left\{
       \max_j\xi_j>x\sqrt{\frac{2H}{1-\rho}}\right\}.
\end{align*}
Put
\[
 t_x:=x\sqrt{\frac{2H}{1-\rho}},
 \qquad
 p_x:=\Prob\{\xi_1>t_x\}.
\]
For a standard real Gaussian variable $\xi_1$,
\[
 p_x=\frac1{\sqrt{2\pi}}\int_{t_x}^{\infty}e^{-u^2/2}\,du.
\]
Integration by parts gives, for $t>0$,
\[
 \frac1{\sqrt{2\pi}}\frac{t}{1+t^2}e^{-t^2/2}
 \le \Prob\{\xi_1>t\}
 \le \frac1{\sqrt{2\pi}t}e^{-t^2/2}.
\]
Therefore
\[
 \log p_x
 =-\frac{x^2}{1-\rho}H+O(\log H).
\]
Since $\log L=(1+o(1))H$, $\rho=o(1)$, and $\log H=o(H)$,
\[
 \log(Lp_x)=-(x^2-1+o(1))H,
\]
so
\[
 Lp_x=\exp\bigl(-(x^2-1+o(1))H\bigr)\longrightarrow0.
\]
{\color{black}Finally, independence gives
\[
 \Prob\{\max_{j\le L}\xi_j\le t_x\}=(1-p_x)^L,
\]
and hence
\[
 \Prob\{\max_{j\le L}\xi_j>t_x\}=1-(1-p_x)^L.
\]
Since $Lp_x\to0$, for all sufficiently large $m$ we have $Lp_x\le1/2$.
Using
\[
 (1-p_x)^L\le e^{-Lp_x}
\]
and the elementary inequality $1-e^{-u}\ge u/2$ for $0\le u\le1$,
\[
 \Prob\{\max_{j\le L}\xi_j>t_x\}
 \ge1-e^{-Lp_x}
 \ge\frac{Lp_x}{2}.
\]
Therefore
\begin{align*}
 \Prob\{\|G\|_\infty>x\sqrt{NH}\}
 &\ge\frac12\Prob\{\max_{j\le L}\xi_j>t_x\}\\
 &\ge\frac14\,Lp_x\\
 &=\exp\bigl(-(x^2-1+\mathrm{o}(1))H\bigr),
\end{align*}
which is the asserted lower bound for the upper tail.}
}
\end{proof}

\subsection{Uniform Gaussian supremum asymptotics}

\begin{proposition}\label{full:full-norm}
{\color{black}Let $N_{m,n}$ and $H_{m,n}$ be given by \eqref{full:full-entropy}. Let
$g=(g_\alpha)_{|\alpha|=m}$ have independent standard complex Gaussian
coordinates, let $\gamma_{m,n}$ be the corresponding standard complex
Gaussian measure on $\CC^{\mathcal M_{m,n}}$, and set
\[
 G(z):=\sum_{|\alpha|=m}g_\alpha z^\alpha.
\]
As $m\to\infty$,
\begin{equation}\label{full:full-mean}
 \sup_{n\ge2}\left|
 {\color{black}
 \frac{\displaystyle\int_{\CC^{\mathcal M_{m,n}}}\|G\|_\infty\,d\gamma_{m,n}}
 {\sqrt{N_{m,n}H_{m,n}}}}-1
 \right|\longrightarrow0.
\end{equation}
Moreover,
\[
 \frac{\|G\|_\infty}{\sqrt{N_{m,n}H_{m,n}}}
 \longrightarrow1
\]
in Gaussian measure, uniformly in $n\ge2$. For every fixed $x>1$,
\begin{equation}\label{full:full-upper-rate}
 \frac{1}{H_{m,n}}
 \log\Prob\left\{\|G\|_\infty>x\sqrt{N_{m,n}H_{m,n}}\right\}
 =-(x^2-1)+\mathrm{o}(1).
\end{equation}
}
\end{proposition}

\begin{proof}[Proof of \cref{full:full-norm}]
{\color{black}Write
\[
 N:=N_{m,n},\qquad H:=H_{m,n},\qquad d:=n-1.
\]
Lemma~\ref{full:full-upper} gives, uniformly for all $n\ge2$,
\[
 \int_{\CC^{\mathcal M_{m,n}}}\|G\|_\infty\,d\gamma_{m,n}
 \le(1+\mathrm{o}(1))\sqrt{NH}
\]
and, for every fixed $x>1$,
\[
 \Prob\{\|G\|_\infty>x\sqrt{NH}\}
 \le\exp\bigl(-(x^2-1+\mathrm{o}(1))H\bigr).
\]
When $n\ge(\log m)^2$, Lemma~\ref{full:full-large-lower} gives the matching
lower mean estimate and the matching lower bound for the upper tail. Hence the
mean asymptotic and the upper-tail rate already hold in that range.

For $2\le n<(\log m)^2$,
Lemma~\ref{endpoint:endpoint-arbitrary-center}, applied with zero center,
gives
\begin{equation}\label{full:full-bl-use}
 \Prob\{\|G\|_\infty\le u\sqrt N\}
 \le(1-\e^{-u^2})^N
 \le\exp(-N\e^{-u^2}).
\end{equation}
{\color{black}To compare $N$ and $H$, write
\[
 N=\binom{m+d}{d}
   =\prod_{j=1}^d\frac{m+j}{j}.
\]
Hence
\begin{align*}
 \log N
 &=\sum_{j=1}^d\log(m+j)-\log(d!)\\
 &=d\log m+\sum_{j=1}^d\log\left(1+\frac jm\right)-\log(d!).
\end{align*}
Since $d<(\log m)^2$,
\[
 0\le\sum_{j=1}^d\log\left(1+\frac jm\right)
 \le\frac1m\sum_{j=1}^d j
 =O(d^2/m),
\]
while Stirling's formula \cite[Section~5.11]{NIST} gives $\log(d!)=d\log d-d+O(\log(d+1))$. Therefore
\[
 \log N
 =d\log m\left(
 1+O\left(\frac{\log\log m}{\log m}\right)
 \right).
\]
From the definition of $H$,
\[
 H
 =\frac d2\log\left(m+\frac{m^2}{d+1}\right)
 =d\log m\left(
 1+O\left(\frac{\log\log m}{\log m}\right)
 \right),
\]
uniformly for $2\le n<(\log m)^2$. Consequently
\[
 \frac{\log N}{H}\longrightarrow1
\]
uniformly in this range.}

Fix $0<x<1$ and take $u=x\sqrt H$ in \eqref{full:full-bl-use}. Then
\begin{equation}\label{full:full-small-lower-tail}
 \Prob\{\|G\|_\infty\le x\sqrt{NH}\}
 \le\exp\left[-\exp\bigl((1-x^2+\mathrm{o}(1))H\bigr)\right],
\end{equation}
so the normalized supremum cannot stay below any fixed $x<1$ with
non-negligible probability. In particular,
\[
 \frac{\displaystyle\int\|G\|_\infty\,d\gamma_{m,n}}{\sqrt{NH}}
 \ge x\,\Prob\{\|G\|_\infty>x\sqrt{NH}\}=x-o(1).
\]
Letting $x\uparrow1$ and combining this with the upper mean estimate from
Lemma~\ref{full:full-upper} proves \eqref{full:full-mean}.

For the lower bound on the upper tail, fix $x>1$ and again use
\eqref{full:full-bl-use}, now with $u=x\sqrt H$. Since
$N\e^{-x^2H}=\exp(-(x^2-1+\mathrm{o}(1))H)\to0$, {\color{black}the elementary relation $1-e^{-z}=z(1+\mathrm{o}(1))$ as $z\downarrow0$ gives}
\begin{align*}
 \Prob\{\|G\|_\infty>x\sqrt{NH}\}
 &\ge1-\exp(-N\e^{-x^2H})\\
 &={\color{black}N\e^{-x^2H}(1+\mathrm{o}(1))}\\
 &=\exp\bigl(-(x^2-1+\mathrm{o}(1))H\bigr).
\end{align*}
{\color{black}Together with the upper estimate from
Lemma~\ref{full:full-upper}, this gives
\[
 \frac1H\log\Prob\{\|G\|_\infty>x\sqrt{NH}\}
 =-(x^2-1)+\mathrm{o}(1)
\]
also in the range $2\le n<(\log m)^2$.

{\color{black}\hypersetup{linkcolor=black,citecolor=black}To prove convergence in Gaussian measure, let $g,h$ be two coefficient
vectors and write $G_g,G_h$ for the corresponding polynomials. Then}
\begin{align*}
 \bigl|\|G_g\|_\infty-\|G_h\|_\infty\bigr|
 &\le\|G_g-G_h\|_\infty\\
 &=\|G_{g-h}\|_\infty\\
 &\le\sum_{|\alpha|=m}|g_\alpha-h_\alpha|\\
 &\le\sqrt N\,\|g-h\|_2.
\end{align*}
Thus $g\mapsto\|G_g\|_\infty/\sqrt N$ is $1$-Lipschitz for the complex
Euclidean norm. If the standard complex Gaussian vector is written as
$(X+iY)/\sqrt2$ with $(X,Y)$ a standard real Gaussian vector, the same map
has Lipschitz constant $1/\sqrt2$ in the underlying real coordinates.
Gaussian concentration \cite[Section~5.4]{BLM} therefore gives
\[
 \Prob\left\{\left|
 \frac{\|G\|_\infty}{\sqrt N}
 -\frac1{\sqrt N}\int\|G\|_\infty\,d\gamma_{m,n}
 \right|>t\right\}
 \le2e^{-t^2}.
\]}
{\color{black}Taking $t=\varepsilon\sqrt H$ gives
\[
 \Prob\left\{\left|
 \frac{\|G\|_\infty}{\sqrt{NH}}-
 \frac{\displaystyle\int\|G\|_\infty\,d\gamma_{m,n}}{\sqrt{NH}}
 \right|>\varepsilon\right\}
 \le2e^{-\varepsilon^2H}\longrightarrow0,
\]
uniformly for $n\ge2$, since $H\to\infty$ uniformly. Together with
\eqref{full:full-mean}, this proves the asserted convergence in Gaussian
measure.}}
\end{proof}

\subsection{Proof of Theorem C}\label{sec:proof-C}

\begin{proof}
{\color{black}Set
\[
 N:=N_{m,n},\qquad H:=H_{m,n},
\]
and let $g=(g_\alpha)_{|\alpha|=m}$ be the standard complex Gaussian
coefficient vector, with
\[
 G(z):=\sum_{|\alpha|=m}g_\alpha z^\alpha.
\]
Put
\[
 U:=\frac{\|g\|_{q_m}}{N^{1/q_m}},\qquad
 V:=\frac{\|G\|_\infty}{\sqrt{NH}}.
\]
Then
\[
 \frac{\sqrt H}{N^{1/(2m)}}R_m(g)=\frac UV,
\]
because $1/q_m=1/2+1/(2m)$.

{\color{black}To control $U$, fix $1\le q\le2$ and set
\[
 \mu_q:=
 \int_{\CC}|z|^q\pi^{-1}e^{-|z|^2}\,dz
 =\Gamma(1+q/2),
 \qquad
 Z_\alpha:=|g_\alpha|^q-\mu_q.
\]
Since $|z|^q\le1+|z|^2$ for $1\le q\le2$, the {\color{black}\hypersetup{linkcolor=black,citecolor=black}estimates}
\[
 \int_{\CC}\bigl||z|^q-\mu_q\bigr|^k
 \pi^{-1}e^{-|z|^2}\,dz
 \le C^k k!,
 \qquad k\ge2,
\]
hold with an absolute constant uniformly in $q$.}
Hence there are absolute constants $c,C>0$ such that
\[
 \log\int_{\CC}e^{t(|z|^q-\mu_q)}
 \pi^{-1}e^{-|z|^2}\,dz
 \le Ct^2\qquad(|t|\le c).
\]
For the independent sum $\sum_\alpha Z_\alpha$ this gives
\[
 \log\int_{\CC^{\mathcal M_{m,n}}}
 \exp\left(t\sum_\alpha Z_\alpha\right)\,d\gamma_{m,n}
 \le CNt^2.
\]
Choosing $t$ proportional to $\varepsilon$ in the exponential Markov
inequality, and applying the same estimate to $-Z_\alpha$, yields
\[
 \Prob\left\{
 \left|\frac1N\sum_\alpha |g_\alpha|^q-\mu_q\right|>\varepsilon
 \right\}
 \le2\exp(-c\varepsilon^2N),\qquad 0<\varepsilon<1.
\]
{\color{black}\hypersetup{linkcolor=black,citecolor=black}The values $\mu_q$ remain in a fixed compact subinterval of $(0,\infty)$.
For $1\le q\le2$, the map $s\mapsto s^{1/q}$ is concave.} If $x\ge\mu_q$,
the mean value theorem gives
\[
 x^{1/q}-\mu_q^{1/q}
 \le \frac1q\,\mu_q^{1/q-1}(x-\mu_q).
\]
{\color{black}\hypersetup{linkcolor=black,citecolor=black}If $0\le x\le\mu_q$, concavity and the value at zero give
\[
 x^{1/q}\ge\frac{x}{\mu_q}\mu_q^{1/q},
\]
and therefore
\[
 \mu_q^{1/q}-x^{1/q}
 \le\mu_q^{1/q-1}(\mu_q-x).
\]
Hence}
\[
 |x^{1/q}-\mu_q^{1/q}|
 \le \mu_q^{1/q-1}|x-\mu_q|
 \le C|x-\mu_q|,
 \qquad x\ge0,\quad 1\le q\le2,
\]
with an absolute constant $C$. Applying this estimate to the empirical mean
gives the same exponential bound, up to absolute constants, after taking
the $q$th root. Thus, for $q=q_m$,
\begin{equation}\label{full:full-coefficients}
 \Prob\left\{\left|
 U-\Gamma(1+q_m/2)^{1/q_m}\right|>\varepsilon\right\}
 \le2\exp(-c\varepsilon^2N),\qquad 0<\varepsilon<1.
\end{equation}
Because $q_m\to2$, the center in \eqref{full:full-coefficients} tends to
one; and since $N\ge m+1$, it follows that $U\to1$ in Gaussian measure
uniformly in $n$. Proposition~\ref{full:full-norm} gives $V\to1$ in Gaussian
measure uniformly in $n$. Hence $U/V\to1$, proving
\eqref{eq:full-exact-main} under Gaussian measure.

For the lower-deviation rate, note that
\[
 N\ge\frac{n(m+1)}2,\qquad H\le n\log m,
 \qquad \frac NH\ge\frac{m+1}{2\log m}\longrightarrow\infty.
\]
Here $N/n$ is increasing in $n$ and equals $(m+1)/2$ at $n=2$.
Fix $y\in(0,1)$ and choose $\eta>0$ so small that $(1-\eta)/y>1$.
Then
\begin{align*}
 \Prob\{V>(1+\eta)/y\}-\Prob\{U>1+\eta\}
 &\le\Prob\{U/V<y\}\\
 &\le\Prob\{V>(1-\eta)/y\}+\Prob\{U<1-\eta\}.
\end{align*}
By \eqref{full:full-coefficients}, the two coefficient errors are
$\exp(-cN)$ and are therefore negligible on the scale $H$, since
$N/H\to\infty$. Applying \eqref{full:full-upper-rate} to the two $V$-terms
and then letting $\eta\downarrow0$ gives
\eqref{full:full-ratio-rate} under Gaussian measure.

Both assertions concern homogeneous degree-zero functions of the coefficient
vector. By Lemma~\ref{lem:gaussian-radius-direction}, their Gaussian and
spherical probabilities coincide, proving the two uniform assertions.}

\medskip
To obtain \eqref{eq:full-fixed-n-main}, fix $n\ge2$. Then
\eqref{full:full-entropy} gives
\[
 H_{m,n}
 =\frac{n-1}{2}\log\left(m+\frac{m^2}{n}\right)
 =(n-1)\log m+O_n(1),
\]
and hence $H_{m,n}/\log m\to n-1$. Moreover,
\[
 \frac1{2m}\log N_{m,n}
 =\frac1{2m}\log\binom{m+n-1}{m}\longrightarrow0,
\]
because $N_{m,n}$ has polynomial growth in $m$ when $n$ is fixed. Thus
$N_{m,n}^{1/(2m)}\to1$, and consequently
\[
 \frac{N_{m,n}^{1/(2m)}\sqrt{\log m}}{\sqrt{H_{m,n}}}
 \longrightarrow\frac1{\sqrt{n-1}}.
\]
Multiplying the convergence in Theorem~\ref{thm:full-exact-main}\textup{(i)}
by this deterministic factor gives
\[
 \sqrt{\log m}\,R_m\longrightarrow\frac1{\sqrt{n-1}}
 \qquad\text{in }\mu_{m,n}\text{-measure},
\]
which is \eqref{eq:full-fixed-n-main}.
\end{proof}

\section{The real case}\label{sec:real-comparison}

The real and complex polynomial inequalities have different extremal
asymptotics.  If $D_m^{\mathbb R}$ denotes the optimal real polynomial
Bohnenblust--Hille constant, then
\begin{equation}\label{eq:real-extremal-rate}
 \limsup_{m\to\infty}(D_m^{\mathbb R})^{1/m}=2
\end{equation}
by \cite{CamposRealPolynomial}.  The spherical behavior also depends on the scalar field and, in the real case, on the ambient dimension.

\subsection{The threshold $1$ fails over the reals}

{\color{black}The next proposition gives the simplest obstruction to the complex threshold from Theorem~\ref{thm:tail-main}: already on a two-monomial real support, a positive proportion of coefficient directions have ratio strictly larger than $1$.}

Let $\sigma_1$ denote normalized arclength measure on
$S^1\subset\mathbb R^2$.

\begin{proposition}\label{prop:real-two-monomial}
Let $m$ be even and
\[
 P_{a,b}(x,y)=ax^m+by^m,
 \qquad (a,b)\in S^1.
\]
Then
\begin{equation}\label{eq:real-two-monomial-one}
 \sigma_1\{(a,b):R_m^{\mathbb R}(P_{a,b})>1\}\ge\frac12.
\end{equation}
Moreover, for every $1<\theta<\sqrt2$,
\begin{equation}\label{eq:real-two-monomial-theta}
 \liminf_{\substack{m\to\infty\\ m\ \mathrm{even}}}
 \sigma_1\{(a,b):R_m^{\mathbb R}(P_{a,b})>\theta\}>0.
\end{equation}
\end{proposition}

\begin{proof}
Since $m$ is even, $x^m,y^m\in[0,1]$ on $[-1,1]^2$, and hence
\begin{equation}\label{eq:real-two-monomial-norm}
 \|P_{a,b}\|_{\infty,\mathbb R}
 =\max\{|a|,|b|,|a+b|\}.
\end{equation}
{\color{black}On the two open arcs where $ab<0$ one has
$|a+b|\le\max\{|a|,|b|\}$. Apart from the four points with one coordinate
zero, \eqref{eq:real-two-monomial-norm} therefore gives
\begin{equation}\label{eq:real-two-monomial-ratio}
 R_m^{\mathbb R}(P_{a,b})
 =\frac{(|a|^{q_m}+|b|^{q_m})^{1/q_m}}
        {\max\{|a|,|b|\}}>1
 \qquad(ab<0).
\end{equation}
The two arcs have total $\sigma_1$-measure $1/2$, which proves
\eqref{eq:real-two-monomial-one}.

For the stronger assertion \eqref{eq:real-two-monomial-theta},
$q_m\to2$ and $a^2+b^2=1$ on $S^1$. Hence
\eqref{eq:real-two-monomial-ratio} gives, for every $(a,b)$ with $ab<0$,
\begin{equation}\label{eq:real-two-monomial-limit}
 R_m^{\mathbb R}(P_{a,b})
 \longrightarrow
 \frac{(a^2+b^2)^{1/2}}{\max\{|a|,|b|\}}
 =\frac1{\max\{|a|,|b|\}}.
\end{equation}
At
\[
 p_+:=(2^{-1/2},-2^{-1/2}),
 \qquad
 p_-:=-p_+,
\]
the limit in \eqref{eq:real-two-monomial-limit} equals $\sqrt2$.

Fix $1<\theta<\sqrt2$. Choose $\eta>0$ so that
$\theta+\eta<\sqrt2$. By continuity of
\[
 (a,b)\longmapsto\frac1{\max\{|a|,|b|\}}
\]
on the open set $\{(a,b)\in S^1:ab<0\}$, there are open arcs
$I_+$ and $I_-$ containing $p_+$ and $p_-$, respectively, such that
\[
 \frac1{\max\{|a|,|b|\}}>\theta+\eta
 \qquad\bigl((a,b)\in I_+\cup I_-\bigr).
\]
Shrinking the arcs if necessary, their closures are compact and remain
inside $\{ab<0\}$. The convergence in
\eqref{eq:real-two-monomial-limit} is uniform on these compact arcs, since
the numerator depends continuously on $(a,b,q)$ for $q$ near $2$ and the
denominator is bounded away from zero. Hence, for all sufficiently large
even $m$,
\[
 R_m^{\mathbb R}(P_{a,b})>\theta
 \qquad\bigl((a,b)\in I_+\cup I_-\bigr).
\]
Consequently,
\[
 \sigma_1\{(a,b):R_m^{\mathbb R}(P_{a,b})>\theta\}
 \ge \sigma_1(I_+\cup I_-)>0
\]
for all sufficiently large even $m$, and
\eqref{eq:real-two-monomial-theta} follows.}
\end{proof}

Thus the complex endpoint statement in Theorem~\ref{thm:tail-main} has no
real analogue even on a support of cardinality two.

\subsection{The full real polynomial space in fixed dimension}

{\color{black}Fix $n\ge2$, put $N=N_{m,n}$, and let
$\gamma_{m,n}^{\mathbb R}$ be standard Gaussian measure on
$\mathbb R^{\mathcal M_{m,n}}$.  For $g=(g_\alpha)_{|\alpha|=m}$ set
\[
 P_g^{\mathbb R}(x)=\sum_{|\alpha|=m}g_\alpha x^\alpha,
 \qquad A_m(g)=\|g\|_{q_m},\qquad
 B_m(g)=\|P_g^{\mathbb R}\|_{\infty,\mathbb R}.
\]
{\color{black}Since $A_m/B_m$ is homogeneous of degree zero, its level
sets are radial. By \eqref{eq:radial-transfer-principle}, their
$\gamma_{m,n}^{\mathbb R}$-probabilities are exactly the corresponding
$\mu_{m,n}^{\mathbb R}$-probabilities of $R_m^{\mathbb R}$.}

\begin{lemma}\label{lem:real-numerator-fixed}
{\color{black}Fix $n\ge2$ and let
$g=(g_\alpha)_{|\alpha|=m}$ be a standard real Gaussian vector. Then
\[
 \frac{\|g\|_{q_m}}{\sqrt{N_{m,n}}}\longrightarrow1
 \qquad\text{in Gaussian measure}.
\]}
\end{lemma}
\begin{proof}
{\color{black}Put
\[
 N:=N_{m,n},\qquad
 A_m(g):=\|g\|_{q_m}
 =\left(\sum_{|\alpha|=m}|g_\alpha|^{q_m}\right)^{1/q_m}.
\]
Then
\begin{equation}\label{eq:real-numerator-factorization}
 \frac{A_m(g)}{\sqrt N}
 =
 \left(\frac1N\sum_{|\alpha|=m}|g_\alpha|^{q_m}\right)^{1/q_m}
 N^{1/q_m-1/2}.
\end{equation}
We treat the two factors separately.

For a standard real Gaussian variable with density
$(2\pi)^{-1/2}e^{-x^2/2}$, set
\begin{equation}\label{eq:real-gaussian-q-integral}
 c_m:=
 \frac1{\sqrt{2\pi}}\int_{\mathbb R}|x|^{q_m}e^{-x^2/2}\,dx.
\end{equation}
Since $q_m\to2$ and $1\le q_m<2$, dominated convergence gives
\[
 c_m\longrightarrow
 \frac1{\sqrt{2\pi}}\int_{\mathbb R}x^2e^{-x^2/2}\,dx=1.
\]
Moreover,
\[
 \frac1{\sqrt{2\pi}}\int_{\mathbb R}
 \bigl(|x|^{q_m}-c_m\bigr)^2e^{-x^2/2}\,dx
 \le C
\]
with an absolute constant $C$: indeed
$|x|^{2q_m}\le1+|x|^4$ and the Gaussian integral of $1+|x|^4$ is finite.

Define
\[
 S_m(g):=\frac1N\sum_{|\alpha|=m}|g_\alpha|^{q_m}.
\]
Independence of the coordinates gives
\[
 \int_{\mathbb R^{\mathcal M_{m,n}}}
 \bigl(S_m(g)-c_m\bigr)^2\,d\gamma_{m,n}^{\mathbb R}(g)
 =
 \frac1N\,
 \frac1{\sqrt{2\pi}}\int_{\mathbb R}
 \bigl(|x|^{q_m}-c_m\bigr)^2e^{-x^2/2}\,dx
 \le\frac CN.
\]
Therefore Chebyshev's inequality and the second-moment bound give, for every
$\varepsilon>0$,
\begin{equation}\label{eq:real-numerator-chebyshev}
 \gamma_{m,n}^{\mathbb R}
 \{|S_m-c_m|>\varepsilon\}
 \le\frac{C}{N\varepsilon^2}
 \longrightarrow0,
\end{equation}
because $N=N_{m,n}\to\infty$ for fixed $n\ge2$. Together with
$c_m\to1$, this proves
\[
 S_m\longrightarrow1
 \qquad\text{in Gaussian measure}.
\]
Since $q_m\to2$, continuity of $(x,q)\mapsto x^{1/q}$ in a neighborhood of
$(1,2)$ now gives
\begin{equation}\label{eq:real-numerator-first-factor}
 S_m^{1/q_m}\longrightarrow1
 \qquad\text{in Gaussian measure}.
\end{equation}

For the deterministic second factor in
\eqref{eq:real-numerator-factorization},
\[
 \frac1{q_m}-\frac12=\frac1{2m},
\]
so
\begin{equation}\label{eq:real-numerator-second-factor}
 N^{1/q_m-1/2}
 =N^{1/(2m)}
 =\exp\left(\frac{\log N}{2m}\right)
 \longrightarrow1.
\end{equation}
Indeed, for fixed $n$,
\[
 N=\binom{m+n-1}{n-1}\le(m+n-1)^{n-1},
\]
and hence
\[
 0\le\frac{\log N}{2m}
 \le\frac{(n-1)\log(m+n-1)}{2m}\longrightarrow0.
\]
Finally, combining
\eqref{eq:real-numerator-factorization},
\eqref{eq:real-numerator-first-factor}, and
\eqref{eq:real-numerator-second-factor} proves the lemma.}
\end{proof}

\begin{lemma}\label{lem:real-denominator-fixed}
Put $N_{m,n}:=\binom{m+n-1}{m}$ and let
\[
 P_g^{\mathbb R}(x):=\sum_{|\alpha|=m}g_\alpha x^\alpha,
 \qquad B_m(g):=\|P_g^{\mathbb R}\|_{\infty,\mathbb R},
\]
where $g=(g_\alpha)_{|\alpha|=m}$ is a standard real Gaussian vector.
There is an absolute constant $C>0$ such that
\[
 \mathbb E B_m(g)\le C\sqrt{N_{m,n}n\log(n+1)}.
\]
For each fixed $n\ge2$ there is $C_n<\infty$ such that
\[
 \Prob\{B_m(g)\le u\sqrt{N_{m,n}}\}\le C_nu^2
 \qquad(0<u\le1,\ m\ge2).
\]
\end{lemma}
\begin{proof}
{\color{black}
Write $N:=N_{m,n}$.

\medskip

For $x=(x_1,\ldots,x_n)\in[-1,1]^n$, define
\[
 V(x):=(x^\alpha)_{|\alpha|=m}\in\mathbb R^{\mathcal M_{m,n}}.
\]
Then
\[
 P_g^{\mathbb R}(x)=\langle g,V(x)\rangle,
\]
and the canonical metric is
\begin{equation}\label{eq:real-canonical-metric}
 d(x,y)^2
 :=
 \int_{\mathbb R^{\mathcal M_{m,n}}}
 |P_g^{\mathbb R}(x)-P_g^{\mathbb R}(y)|^2
 \,d\gamma_{m,n}^{\mathbb R}(g)
 =\|V(x)-V(y)\|_2^2.
\end{equation}
{\color{black}\hypersetup{linkcolor=black,citecolor=black}Thus $V$ maps} the parameter cube into Euclidean coefficient
space, and Euclidean lengths of curves under $V$ control distances in
$d$.

Fix a coordinate $j$ and keep all coordinates except $x_j=t$ fixed. Since
the $\alpha$-coordinate of $V$ is $x^\alpha$, {\color{black}\hypersetup{linkcolor=black,citecolor=black}the derivative of $V$} with respect
to $t$ is the vector
\[
 \frac{\partial V}{\partial x_j}(x)
 =
 \bigl(\alpha_j t^{\alpha_j-1}
       \prod_{\ell\ne j}x_\ell^{\alpha_\ell}\bigr)_{|\alpha|=m}.
\]
Its squared Euclidean norm satisfies
\begin{align}
 \left\|\frac{\partial V}{\partial x_j}(x)\right\|_2^2
 &=
 \sum_{|\alpha|=m}
 \alpha_j^2t^{2\alpha_j-2}
 \prod_{\ell\ne j}|x_\ell|^{2\alpha_\ell}\notag\\
 &\le
 \sum_{k=1}^m k^2t^{2k-2}
 \binom{m-k+n-2}{n-2}
 =:S_{m,n}(t).
 \label{eq:real-coordinate-derivative}
\end{align}
Here the binomial coefficient counts the choices of the remaining
$n-1$ entries of $\alpha$ after fixing $\alpha_j=k$.

Put
\[
 B:=\binom{m+n-2}{n-2}.
\]
Since
$\binom{m-k+n-2}{n-2}\le B$,
\begin{equation}\label{eq:real-S-bound}
 S_{m,n}(t)
 \le B\sum_{k=1}^m k^2t^{2k-2}
 \le B\min\{m^3,\,2(1-t)^{-3}\}.
\end{equation}
The last estimate follows respectively from
$\sum_{k=1}^m k^2\le m^3$ and from the differentiated geometric series
$\sum_{k\ge1}k^2t^{k-1}=(1+t)/(1-t)^3$, applied to $t^2\le t$ for
$0\le t<1$.

For a continuously differentiable curve
$\Gamma:[a,b]\to\mathbb R^N$, the fundamental theorem of calculus and
the triangle inequality give
\[
 \|\Gamma(b)-\Gamma(a)\|_2
 \le\int_a^b\|\Gamma'(t)\|_2\,dt.
\]
Applied to the coordinate curve $t\mapsto V(x_1,\ldots,t,\ldots,x_n)$,
\eqref{eq:real-coordinate-derivative} shows that its total length from
$0$ to $1$ is at most
\[
 L:=\int_0^1\sqrt{S_{m,n}(t)}\,dt.
\]
Using \eqref{eq:real-S-bound} and splitting the integral at $1-1/m$,
\begin{align}
 L
 &\le C\sqrt B\left(
 \int_0^{1-1/m}(1-t)^{-3/2}\,dt
 +m^{3/2}\int_{1-1/m}^1dt\right)\notag\\
 &\le C\sqrt{mB}
 =C\sqrt{N\,\frac{m(n-1)}{m+n-1}}
 \le C\sqrt{N(n-1)}.
 \label{eq:real-coordinate-length}
\end{align}

Define the increasing coordinate-length function
\[
 F(t):=\operatorname{sgn}(t)
 \int_0^{|t|}\sqrt{S_{m,n}(s)}\,ds,
 \qquad -1\le t\le1.
\]
To pass from $x$ to $y$, change the coordinates one at a time. Applying the
curve-length estimate to each coordinate and then the triangle inequality
in \eqref{eq:real-canonical-metric} gives
\begin{equation}\label{eq:real-coordinate-metric}
 d(x,y)\le\sum_{j=1}^n|F(x_j)-F(y_j)|.
\end{equation}

The interval $F([-1,1])$ has length at most $2L$. Choose a one-dimensional
grid in $F([-1,1])$ of mesh at most $\eta/n$ and take its $n$-fold Cartesian
product. By \eqref{eq:real-coordinate-metric}, pulling this product grid back through $F$ coordinatewise gives an
$\eta$-net for $[-1,1]^n$ in the metric $d$. Its cardinality is at most
\begin{equation}\label{eq:real-canonical-covering}
 \mathcal N([-1,1]^n,d,\eta)
 \le\left(2+\frac{2nL}{\eta}\right)^n.
\end{equation}
Let $\mathcal A_\eta$ be the corresponding $\eta$-net of $V([-1,1]^n)$. Then
$\mathcal A_\eta\cup(-\mathcal A_\eta)$ is an $\eta$-net of
$V([-1,1]^n)\cup -V([-1,1]^n)$, with cardinality at most twice that in
\eqref{eq:real-canonical-covering}. Moreover,
\[
 B_m(g)=\sup_{v\in V([-1,1]^n)\cup -V([-1,1]^n)}\langle g,v\rangle.
\]
Since $\|V(x)\|_2\le\sqrt N$, the diameter of the signed set is at most
$2\sqrt N$. Dudley's entropy estimate
\cite[Corollary~13.2]{BLM}, together with
\eqref{eq:real-coordinate-length} and
\eqref{eq:real-canonical-covering}, gives
\begin{align*}
 \int B_m(g)\,d\gamma_{m,n}^{\mathbb R}(g)
 &\le C\int_0^{2\sqrt N}
 \sqrt{\log2+n\log\left(2+\frac{2nL}{\eta}\right)}\,d\eta\\
 &\le C\sqrt N\int_0^2
 \sqrt{\log2+n\log\left(2+\frac{Cn^{3/2}}t\right)}\,dt\\
 &\le C\sqrt{Nn\log(n+1)}.
\end{align*}
For the last step, use
\[
 \log\left(2+\frac{Cn^{3/2}}t\right)
 \le C\log(n+1)+\log(1+t^{-1}),
 \qquad0<t\le2,
\]
and the integrability of
$\sqrt{\log(1+t^{-1})}$ at $0$.

\medskip

Two points of the cube suffice. Set
\[
 X:=\frac{P_g^{\mathbb R}(1,\ldots,1)}{\sqrt N},
 \qquad
 Y:=\frac{P_g^{\mathbb R}(-1,1,\ldots,1)}{\sqrt N}.
\]
Both are real linear functionals of the Gaussian coefficient vector.
Directly from the independence and normalization of the coefficients,
\[
 \int_{\mathbb R^{\mathcal M_{m,n}}}X\,d\gamma_{m,n}^{\mathbb R}
 =
 \int_{\mathbb R^{\mathcal M_{m,n}}}Y\,d\gamma_{m,n}^{\mathbb R}=0,
\]
and
\[
 \int_{\mathbb R^{\mathcal M_{m,n}}}X^2\,d\gamma_{m,n}^{\mathbb R}
 =
 \int_{\mathbb R^{\mathcal M_{m,n}}}Y^2\,d\gamma_{m,n}^{\mathbb R}=1.
\]
The mixed integral is
\begin{equation}\label{eq:real-two-point-correlation}
 \rho_m:=
 \int XY\,d\gamma_{m,n}^{\mathbb R}
 =
 \frac1N\sum_{|\alpha|=m}(-1)^{\alpha_1}.
\end{equation}

Group the sum according to $k=\alpha_1$:
\[
 \rho_m
 =
 \frac1N\sum_{k=0}^m(-1)^k
 \binom{m-k+n-2}{n-2}.
\]
For $0\le k\le m$, the identity
\[
 \frac1N\binom{m-k+n-2}{n-2}
 =
 (n-1)\binom mk
 \int_0^1
 u^k(1-u)^{m-k+n-2}\,du
\]
follows from
\[
 \int_0^1u^k(1-u)^{m-k+n-2}\,du
 =
 \frac{k!\,(m-k+n-2)!}{(m+n-1)!}.
\]
Substitution into the finite sum and interchange of sum and integral give
\begin{align*}
 \rho_m
 &=
 (n-1)\int_0^1(1-u)^{n-2}
 \sum_{k=0}^m
 \binom mk(-u)^k(1-u)^{m-k}\,du\\
 &=
 (n-1)\int_0^1(1-2u)^m(1-u)^{n-2}\,du.
\end{align*}
Thus
\begin{equation}\label{eq:real-two-point-correlation-integral}
 \rho_m
 =
 (n-1)\int_0^1(1-2u)^m(1-u)^{n-2}\,du.
\end{equation}

For $m\ge2$,
\begin{align*}
 |\rho_m|
 &\le(n-1)\int_0^1|1-2u|^m(1-u)^{n-2}\,du\\
 &\le(n-1)\int_0^1(1-2u)^2(1-u)^{n-2}\,du\\
 &=1-\frac{4(n-1)}{n(n+1)}
 =:\rho_n<1.
\end{align*}
Thus $|\rho_m|\le\rho_n<1$ uniformly in $m$ for fixed $n$. The two
linear functionals $(X,Y)$ therefore have covariance matrix
\[
 \begin{pmatrix}1&\rho_m\\ \rho_m&1\end{pmatrix},
\]
whose determinant satisfies
$1-\rho_m^2\ge1-\rho_n^2>0$. Hence their joint density is
\[
 f_{X,Y}(x,y)
 =
 \frac1{2\pi\sqrt{1-\rho_m^2}}
 \exp\left(
 -\frac{x^2-2\rho_mxy+y^2}{2(1-\rho_m^2)}
 \right),
\]
and, since the exponential factor is at most $1$,
\[
 f_{X,Y}(x,y)
 \le\frac1{2\pi\sqrt{1-\rho_n^2}}
 =:C_n.
\]
If $B_m(g)\le u\sqrt N$, then in particular $|X|\le u$ and $|Y|\le u$.
Consequently,
\[
 \gamma_{m,n}^{\mathbb R}\{B_m(g)\le u\sqrt N\}
 \le\int_{[-u,u]^2}f_{X,Y}(x,y)\,dx\,dy
 \le4C_nu^2,
\]
which proves the second assertion.
}
\end{proof}

\begin{proposition}\label{prop:real-fixed-tail}
{\color{black}For every fixed $n\ge2$ there is $C_n<\infty$ such that
\[
 \left.
 \begin{aligned}
 \limsup_{m\to\infty}\mu_{m,n}^{\mathbb R}
 \{R_m^{\mathbb R}<u\}&\le C_nu, &&0<u\le1,\\
 \limsup_{m\to\infty}\mu_{m,n}^{\mathbb R}
 \{R_m^{\mathbb R}>v\}&\le C_nv^{-2}, &&v\ge1.
 \end{aligned}
 \right\}.
\]}
\end{proposition}
\begin{proof}
{\color{black}
By \eqref{eq:radial-transfer-principle}, it is enough to work with Gaussian
coefficients. Put
\[
 N:=N_{m,n},\qquad
 A_m(g):=\|g\|_{q_m},\qquad
 B_m(g):=\|P_g^{\mathbb R}\|_{\infty,\mathbb R}.
\]

Fix $0<u\le1$. If $A_m\ge\frac12\sqrt N$ and $A_m/B_m<u$, then
$B_m>\sqrt N/(2u)$. Hence
\begin{equation}\label{eq:real-fixed-lower-split}
 \gamma_{m,n}^{\mathbb R}\{A_m/B_m<u\}
 \le
 \gamma_{m,n}^{\mathbb R}\{A_m<\tfrac12\sqrt N\}
 +\gamma_{m,n}^{\mathbb R}\{B_m>\tfrac{\sqrt N}{2u}\}.
\end{equation}
The first term is $o(1)$ by Lemma~\ref{lem:real-numerator-fixed}. For the
second, Markov's inequality gives
\[
 \gamma_{m,n}^{\mathbb R}\{B_m>s\}
 \le\frac1s\int B_m\,d\gamma_{m,n}^{\mathbb R}.
\]
With $s=\sqrt N/(2u)$, Lemma~\ref{lem:real-denominator-fixed} yields
\[
 \gamma_{m,n}^{\mathbb R}\{B_m>\tfrac{\sqrt N}{2u}\}
 \le\frac{2u}{\sqrt N}\int B_m\,d\gamma_{m,n}^{\mathbb R}
 \le C_nu.
\]
Substitution in \eqref{eq:real-fixed-lower-split} gives
\[
 \limsup_{m\to\infty}
 \mu_{m,n}^{\mathbb R}\{R_m^{\mathbb R}<u\}
 \le C_nu.
\]

\medskip
For the upper tail, let $v\ge2$. If $A_m\le2\sqrt N$ and $A_m/B_m>v$, then
$B_m<2\sqrt N/v$. Therefore
\begin{equation}\label{eq:real-fixed-upper-split}
 \gamma_{m,n}^{\mathbb R}\{A_m/B_m>v\}
 \le
 \gamma_{m,n}^{\mathbb R}\{A_m>2\sqrt N\}
 +\gamma_{m,n}^{\mathbb R}\{B_m<2\sqrt N/v\}.
\end{equation}
The first term is $o(1)$ by Lemma~\ref{lem:real-numerator-fixed}; the second
is at most $C_nv^{-2}$ by Lemma~\ref{lem:real-denominator-fixed}, applied
with $u=2/v$. Hence
\[
 \limsup_{m\to\infty}
 \mu_{m,n}^{\mathbb R}\{R_m^{\mathbb R}>v\}
 \le C_nv^{-2}
\]
for $v\ge2$; increasing $C_n$ covers $1\le v<2$.
}
\end{proof}
}

\subsection{Growing dimension}\label{sec:real-growing-dimension}

A Hamming-separated family of sign vectors produces many evaluations with
uniformly controlled covariances. Slepian's inequality then gives the lower
bound for the real supremum that is needed when the dimension grows.

\begin{proposition}\label{prop:real-growing-dimension}
There are absolute constants $C,c>0$ such that, for $m\ge2$ and $n\ge4$,
\begin{equation}\label{eq:real-growing-tail}
 \mu_{m,n}^{\mathbb R}\left\{a:
 R_m^{\mathbb R}(a)>C\sqrt{\frac1m+\frac1n}\right\}
 \le \exp(-cn)+\exp(-cN_{m,n}).
\end{equation}
\end{proposition}

\begin{proof}
{\color{black}
Put $N:=N_{m,n}$ and
\[
 P_g^{\mathbb R}(x):=\sum_{|\alpha|=m}g_\alpha x^\alpha.
\]
We first work with Gaussian coefficients. The event in
\eqref{eq:real-growing-tail} is invariant under positive scalar multiplication,
so \eqref{eq:radial-transfer-principle} transfers the resulting probability
estimate to the coefficient sphere.

For $\varepsilon,\eta\in\{-1,1\}^n$, define
\[
 d_H(\varepsilon,\eta)
 :=|\{j\in[n]:\varepsilon_j\ne\eta_j\}|.
\]
Choose
$\varepsilon^{(1)},\ldots,\varepsilon^{(M)}$ such that
\begin{equation}\label{eq:real-hamming-separation}
 \frac n4\le d_H(\varepsilon^{(i)},\varepsilon^{(j)})
 \le\frac{3n}{4}\qquad(i\ne j)
\end{equation}
and $M\ge e^{c_1n}$.

Choose the vectors greedily. After selecting $\varepsilon$, exclude the
Hamming balls of radius $\lfloor n/4\rfloor$ centered at $\varepsilon$ and
$-\varepsilon$. If a later vector $\eta$ avoids both balls, then
$d_H(\eta,\varepsilon)\ge n/4$ and
\[
 d_H(\eta,-\varepsilon)
 =n-d_H(\eta,\varepsilon)\ge n/4,
\]
which gives the upper bound in \eqref{eq:real-hamming-separation}.

To estimate the size of one excluded ball, put $t=1/4$ and
$x=t/(1-t)=1/3$. Since $x^k\ge x^{tn}$ for $k\le tn$,
\begin{equation}\label{eq:real-hamming-ball}
 x^{tn}\sum_{k=0}^{\lfloor tn\rfloor}\binom nk
 \le\sum_{k=0}^n\binom nkx^k=(1+x)^n.
\end{equation}
Thus
\[
 \sum_{k=0}^{\lfloor n/4\rfloor}\binom nk
 \le \exp\{n h(1/4)\},
 \qquad
 h(t):=-t\log t-(1-t)\log(1-t).
\]
Since $h(1/4)<\log2$, put
\[
 c_0:=\log2-h(1/4)>0.
\]
Each excluded ball then contains at most $2^ne^{-c_0n}$ sign vectors, so
the greedy construction selects at least
\[
 M\ge\frac12e^{c_0n}.
\]
In particular, for $n\ge n_0:=\lceil2\log2/c_0\rceil$,
\[
 M\ge e^{(c_0/2)n}.
\]
For each of the finitely many integers $4\le n<n_0$, two sign vectors at
Hamming distance $\lfloor n/2\rfloor$ satisfy
\eqref{eq:real-hamming-separation}. Hence, after choosing
\[
 c_1:=
 \min\left\{\frac{c_0}{2},
 \frac{\log2}{n_0}\right\}>0,
\]
one has
\begin{equation}\label{eq:real-hamming-cardinality}
 M\ge e^{c_1n}\qquad(n\ge4).
\end{equation}

Define
\[
 X_j(g):=\frac{P_g^{\mathbb R}(\varepsilon^{(j)})}{\sqrt N}.
\]
Then
\[
 \int X_j^2\,d\gamma_{m,n}^{\mathbb R}=1.
\]
If two chosen vectors differ on $S\subset[n]$, $b:=|S|$, independence of
the coefficients gives
\begin{equation}\label{eq:real-sign-covariance}
 \int X_iX_j\,d\gamma_{m,n}^{\mathbb R}
 =\frac1N\sum_{|\alpha|=m}(-1)^{\sum_{r\in S}\alpha_r}.
\end{equation}
Grouping the multiindices according to
$k=\sum_{r\in S}\alpha_r$ gives
\[
 \int X_iX_j\,d\gamma_{m,n}^{\mathbb R}
 =
 \frac1N\sum_{k=0}^m(-1)^k
 \binom{k+b-1}{b-1}
 \binom{m-k+n-b-1}{n-b-1}.
\]
For $0\le k\le m$,
\begin{align*}
 &\frac1N
 \binom{k+b-1}{b-1}
 \binom{m-k+n-b-1}{n-b-1}\\
 &\qquad=
 \frac{\Gamma(n)}{\Gamma(b)\Gamma(n-b)}
 \binom mk
 \int_0^1
 u^{k+b-1}(1-u)^{m-k+n-b-1}\,du.
\end{align*}
Indeed, the integral equals
\[
 \frac{\Gamma(k+b)\Gamma(m-k+n-b)}
 {\Gamma(m+n)},
\]
and substitution of the factorial formulas for the three binomial
coefficients gives the displayed identity. Summing in $k$ and using
\[
 \sum_{k=0}^m\binom mk(-u)^k(1-u)^{m-k}
 =(1-2u)^m
\]
gives
\begin{equation}\label{eq:real-sign-integral}
 \int X_iX_j\,d\gamma_{m,n}^{\mathbb R}
 =
 \frac{\Gamma(n)}{\Gamma(b)\Gamma(n-b)}
 \int_0^1(1-2u)^m u^{b-1}(1-u)^{n-b-1}\,du.
\end{equation} By
\eqref{eq:real-hamming-separation}, $p:=b/n\in[1/4,3/4]$. Since $m\ge2$,
the absolute value of \eqref{eq:real-sign-integral} is bounded by the same
integral with $(1-2u)^2$. Direct integration gives
\[
 (1-2p)^2+\frac{4p(1-p)}{n+1}
 \le\frac14+\frac1{n+1}<\frac12.
\]
Thus
\begin{equation}\label{eq:real-sign-mixed-bound}
 \int X_iX_j\,d\gamma_{m,n}^{\mathbb R}\le\frac12
 \qquad(i\ne j).
\end{equation}

Let $Z_0,Z_1,\ldots,Z_M$ be independent standard real Gaussian variables
and put $Y_j:=2^{-1/2}(Z_0+Z_j)$. Then
\[
 \E Y_j^2=1,\qquad
 \E(Y_iY_j)=\frac12\quad(i\ne j).
\]
By \eqref{eq:real-sign-mixed-bound}, Slepian's inequality \cite{Slepian}
gives
\[
 \int\max_jX_j\,d\gamma_{m,n}^{\mathbb R}
 \ge\frac1{\sqrt2}\,\E\max_jZ_j
 \ge c\sqrt{\log M}\ge c'\sqrt n,
\]
because \eqref{eq:real-hamming-cardinality} gives
$\log M\ge c_1n$.

Each $X_j$ is a linear functional with Euclidean norm one, so
$g\mapsto\max_jX_j(g)$ is $1$-Lipschitz. Gaussian concentration
{\color{black}\cite[Section~5.4]{BLM}} therefore gives absolute constants $c_2,c_3>0$ such that
\begin{equation}\label{eq:real-sign-lower}
 \gamma_{m,n}^{\mathbb R}
 \{\,\|P_g^{\mathbb R}\|_{\infty,\mathbb R}<c_2\sqrt{Nn}\,\}
 \le e^{-c_3n}.
\end{equation}

The exponential Markov inequality applied to $\|g\|_2^2$ gives
\begin{equation}\label{eq:real-gaussian-norm-tail}
 \gamma_{m,n}^{\mathbb R}\{\|g\|_2>2\sqrt N\}\le e^{-c_4N}
\end{equation}
for an absolute $c_4>0$. Since $q_m<2$,
\begin{equation}\label{eq:real-q2-comparison}
 \|g\|_{q_m}
 \le N^{1/q_m-1/2}\|g\|_2
 =N^{1/(2m)}\|g\|_2.
\end{equation}
Outside the exceptional sets in
\eqref{eq:real-sign-lower} and \eqref{eq:real-gaussian-norm-tail},
\[
 R_m^{\mathbb R}(g)
 \le C\frac{N^{1/(2m)}}{\sqrt n}.
\]
Finally,
\[
 N=\binom{m+n-1}{m}
 \le\left(\frac{e(m+n-1)}m\right)^m,
\]
so
\begin{equation}\label{eq:real-growing-final-scale}
 \frac{N^{1/(2m)}}{\sqrt n}
 \le\sqrt{e\left(\frac1m+\frac1n\right)}.
\end{equation}
Equations \eqref{eq:real-sign-lower},
\eqref{eq:real-gaussian-norm-tail}, and
\eqref{eq:real-growing-final-scale} give the Gaussian estimate.
The event is homogeneous of degree zero, so
\eqref{eq:radial-transfer-principle} gives
\eqref{eq:real-growing-tail}.
}
\end{proof}

\subsection{Proof of Theorem D}\label{sec:proof-D-real}

\begin{proof}
{\color{black}
The fixed- and growing-dimensional regimes are treated separately.

\medskip
\noindent\textbf{Case 1. Fixed dimension.}
Fix $n\ge2$ and $\varepsilon>0$. By
Proposition~\ref{prop:real-fixed-tail}, for $0<u\le1$ and $v\ge1$,
\[
 \limsup_{m\to\infty}\mu_{m,n}^{\mathbb R}\{R_m^{\mathbb R}<u\}
 \le C_nu,
 \qquad
 \limsup_{m\to\infty}\mu_{m,n}^{\mathbb R}\{R_m^{\mathbb R}>v\}
 \le C_nv^{-2}.
\]
Choose $u=u(n,\varepsilon)>0$ and $v=v(n,\varepsilon)>1$ so that
\[
 C_nu<\frac{\varepsilon}{2},
 \qquad
 C_nv^{-2}<\frac{\varepsilon}{2}.
\]
Then
{\color{black}\hypersetup{linkcolor=black,citecolor=black}\[
 \liminf_{m\to\infty}
 \mu_{m,n}^{\mathbb R}
 \{u\le R_m^{\mathbb R}\le v\}
 \ge1-C_nu-C_nv^{-2}>1-\varepsilon.
\]
Consequently, \eqref{eq:real-main-window} holds for all sufficiently
large $m$.} For $n=1$ there is only one monomial,
so $R_m^{\mathbb R}\equiv1$.

\medskip
\noindent\textbf{Case 2. Growing dimension.}
Suppose $n_m\to\infty$. Since also $m\to\infty$,
\[
 \sqrt{\frac1m+\frac1{n_m}}\longrightarrow0.
\]
Proposition~\ref{prop:real-growing-dimension} therefore implies that, for
every $\eta>0$,
\[
 \mu_{m,n_m}^{\mathbb R}\{R_m^{\mathbb R}>\eta\}\longrightarrow0.
\]
Hence $R_m^{\mathbb R}\to0$ in
$\mu_{m,n_m}^{\mathbb R}$-measure.

Conversely, suppose that $n_m$ does not tend to infinity. Then there are
$K\in\mathbb N$ and an infinite subsequence $(m_j)$ such that
$n_{m_j}\le K$ for every $j$. Since
$\{1,\ldots,K\}$ is finite, some value $n_0$ occurs along a further infinite
subsequence. Passing to this further subsequence and relabelling it again as
$(m_j)$, we have $n_{m_j}=n_0$ for every $j$. Apply Case~1 with $n=n_0$
and $\varepsilon=1/2$. There is
$c_{n_0,1/2}>0$ such that, for all sufficiently large indices in that
subsequence,
\[
 \mu_{m_j,n_0}^{\mathbb R}
 \{R_{m_j}^{\mathbb R}\ge c_{n_0,1/2}\}\ge\frac12.
\]
Thus convergence to zero in measure is impossible on that subsequence.
}
\end{proof}

{\color{black}\hypersetup{linkcolor=black,citecolor=black}Theorems~\ref{thm:vanishing-main}
and~\ref{thm:real-main} give different criteria for the typical ratio to
vanish. For complex polynomials, the ratio tends to zero in spherical
measure exactly when the number of prescribed monomials tends to infinity.
For the full real polynomial spaces, this occurs exactly when the ambient
dimension tends to infinity.
Proposition~\ref{prop:real-growing-dimension} gives a quantitative bound
in the latter case.}

\section{Proof of Theorem E}\label{sec:critical}

\begin{proof}
{\color{black}
Theorem~\ref{thm:full-exact-main}\textup{(i)} gives
\begin{equation}\label{eq:critical-uniform-input}
 \frac{\sqrt{H_{m,n_m}}}{N_{m,n_m}^{1/(2m)}}R_m
 \longrightarrow1
 \qquad\text{in }\mu_{m,n_m}\text{-measure}.
\end{equation}
For the deterministic normalization, put
$\lambda_m:=n_m/m$. By assumption, $\lambda_m\to1$. Stirling's formula \cite[Section~5.11]{NIST}
\[
 \log(k!)=k\log k-k+O(\log k)
\]
applied to
\[
 N_{m,n_m}=\binom{m+n_m-1}{m}
\]
gives the following expansion. Since $n_m/m\to1$, replacing $m+n_m-1$
by $m+n_m$ and $n_m-1$ by $n_m$ changes $\log N_{m,n_m}$ by $o(m)$:

\begin{align*}
 \frac1m\log N_{m,n_m}
 &=
 \frac{m+n_m}{m}\log(m+n_m)
 -\log m-\frac{n_m}{m}\log n_m+o(1)\\
 &=
 (1+\lambda_m)\log(1+\lambda_m)
 -\lambda_m\log\lambda_m+o(1).
\end{align*}
Therefore
\begin{equation}\label{eq:critical-N-limit}
 N_{m,n_m}^{1/(2m)}
 \longrightarrow
 \exp\!\left(\frac12\cdot2\log2\right)=2.
\end{equation}

For $H_{m,n_m}$,
\begin{align*}
 \frac{H_{m,n_m}}{m\log m}
 &=
 \frac{n_m-1}{2m}\,
 \frac{\log\!\left(m+m^2/n_m\right)}{\log m}\\
 &=
 \frac{n_m-1}{2m}\,
 \left[
 1+\frac{\log(1+m/n_m)}{\log m}
 \right]
 \longrightarrow\frac12.
\end{align*}
Hence
\begin{equation}\label{eq:critical-H-limit}
 \frac{\sqrt{m\log m}}{\sqrt{H_{m,n_m}}}
 \longrightarrow\sqrt2.
\end{equation}
Multiplying \eqref{eq:critical-uniform-input} by the deterministic factors
in \eqref{eq:critical-N-limit} and \eqref{eq:critical-H-limit} yields
\[
 \sqrt{m\log m}\,R_m
 \longrightarrow2\sqrt2
 \qquad\text{in }\mu_{m,n_m}\text{-measure}.
\]
}
\end{proof}

\section{A multilinear counterpart}\label{sec:multilinear}

Fix $n\ge2$ throughout this section and put $d:=n-1$.
Let
\[
 \mathbb G_n:=\TT^n/\TT,
\]
where $\TT$ acts diagonally. For $z,w\in\TT^n$ define
\begin{equation}\label{eq:projective-distance}
 \delta([z],[w])^2
 :=
 \frac1n\min_{\lambda\in\TT}\|z-\lambda w\|_2^2
 =
 2\left(1-\left|\frac1n\sum_{j=1}^n z_j\overline{w_j}\right|\right).
\end{equation}
The second identity follows by maximizing
$\Re\bigl(\lambda\sum_jz_j\overline{w_j}\bigr)$ over $\lambda\in\TT$.
On $\mathbb G_n^m$ put
\begin{equation}\label{eq:product-projective-distance}
 \Delta_m(z,w)^2:=\sum_{r=1}^m
 \delta(z^{(r)},w^{(r)})^2.
\end{equation}

\begin{lemma}\label{lem:projective-packing}
Fix $n\ge2$, put $d:=n-1$, and equip $\mathbb G_n^m$ with the metric $\Delta_m$
from \eqref{eq:product-projective-distance}. For all sufficiently large $m$,
there are points $z_1,\ldots,z_{M_m}\in\mathbb G_n^m$ such that
\begin{equation}\label{eq:projective-packing-correlation}
 \Delta_m(z_i,z_j)\ge2\sqrt{\log m}
 \qquad(i\ne j),
\end{equation}
and
\begin{equation}\label{eq:projective-packing-size}
 \log M_m
 \ge
 \frac d2\,m\log m-\frac d2\,m\log\log m-O_n(m).
\end{equation}
\end{lemma}

\begin{proof}
{\color{black}
Let $\nu_n$ be normalized Haar measure on $\mathbb G_n$.

\medskip

A class near the identity has a representative
\[
 u=(e^{it_1},\ldots,e^{it_{n-1}},1),
 \qquad |t_j|<\pi,
\]
and we put $t_n:=0$. From \eqref{eq:projective-distance},
\[
 \delta(u,1)^2
 =\frac1n\min_{\theta\in\mathbb R}
 \sum_{j=1}^n|e^{it_j}-e^{i\theta}|^2.
\]
For $t$ and $\theta$ in a fixed neighbourhood of zero, the quantities
$|e^{it_j}-e^{i\theta}|$ and $|t_j-\theta|$ are comparable, with constants
depending only on $n$.
Moreover,
\[
 \min_{\theta\in\mathbb R}\sum_{j=1}^n|t_j-\theta|^2
 =
 \sum_{j=1}^n|t_j-\bar t|^2,
 \qquad
 \bar t:=\frac1n\sum_{j=1}^nt_j.
\]
On the section $t_n=0$ this quadratic form is positive definite. Hence
there are constants $0<c_n<C_n<\infty$ and $r_n\in(0,1)$ such that
\[
 c_n\sum_{j=1}^{n-1}t_j^2
 \le\delta(u,1)^2
 \le C_n\sum_{j=1}^{n-1}t_j^2
\]
whenever $\delta(u,1)\le r_n$. Haar measure in these coordinates is a
constant multiple of Lebesgue measure, and therefore
\[
 \nu_n\{u:\delta(u,1)\le r\}\le C_nr^d,
 \qquad 0<r\le r_n.
\]
For $r_n<r\le1$, the trivial bound
\[
 \nu_n\{u:\delta(u,1)\le r\}\le1
 \le r_n^{-d}r^d
\]
has the same form after increasing the constant. Thus
\begin{equation}\label{eq:projective-small-ball}
 \nu_n\{u:\delta(u,1)\le r\}\le C_nr^d,
 \qquad 0<r\le1.
\end{equation}

\medskip

Set
\[
 F(r):=\nu_n\{u:\delta(u,1)\le r\},
 \qquad0\le r\le2.
\]
For $\lambda\ge1$, integration by parts for the nondecreasing function $F$
gives
\begin{align}
 \int_{\mathbb G_n}e^{-\lambda\delta(u,1)^2}\,d\nu_n(u)
 &=\int_{[0,2]}e^{-\lambda r^2}\,dF(r)\notag\\
 &=e^{-4\lambda}
 +2\lambda\int_0^2re^{-\lambda r^2}F(r)\,dr.\label{eq:projective-ibp}
\end{align}
Using \eqref{eq:projective-small-ball} on $[0,1]$ and $F(r)\le1$ on
$[1,2]$,
\begin{align*}
 \int_{\mathbb G_n}e^{-\lambda\delta(u,1)^2}\,d\nu_n(u)
 &\le e^{-4\lambda}
 +2C_n\lambda\int_0^1r^{d+1}e^{-\lambda r^2}\,dr
 +2\lambda\int_1^2re^{-\lambda r^2}\,dr.
\end{align*}
In the middle integral put $s=\sqrt\lambda\,r$. The resulting integral is
bounded independently of $\lambda$, while the two remaining terms are
exponentially small. Thus
\begin{equation}\label{eq:projective-laplace}
 \int_{\mathbb G_n}e^{-\lambda\delta(u,1)^2}\,d\nu_n(u)
 \le C_n'\lambda^{-d/2}.
\end{equation}

\medskip

Let $\nu_n^{\otimes m}$ be product Haar measure on $\mathbb G_n^m$. For
$\lambda\ge1$, on the event
\[
 \sum_{r=1}^m\delta(u_r,1)^2\le4\log m
\]
one has
\[
 e^{-\lambda\sum_r\delta(u_r,1)^2}\ge e^{-4\lambda\log m}.
\]
Markov's inequality therefore gives
\begin{align}
 &\nu_n^{\otimes m}
 \left\{u:\sum_{r=1}^m\delta(u_r,1)^2\le4\log m\right\}\notag\\
 &\qquad\le
 e^{4\lambda\log m}
 \prod_{r=1}^m
 \int_{\mathbb G_n}e^{-\lambda\delta(u_r,1)^2}\,d\nu_n(u_r)\notag\\
 &\qquad\overset{\eqref{eq:projective-laplace}}{\le}
 e^{4\lambda\log m}
 \left(C_n'\lambda^{-d/2}\right)^m.
 \label{eq:projective-ball-measure}
\end{align}
Choose
\[
 \lambda:=\frac{dm}{8\log m}.
\]
For large $m$ this is at least $1$. Taking logarithms in
\eqref{eq:projective-ball-measure} gives
\begin{equation}\label{eq:projective-ball-volume}
 \log\nu_n^{\otimes m}
 \{u:\Delta_m(u,1)\le2\sqrt{\log m}\}
 \le
 -\frac d2m\log m+\frac d2m\log\log m+O_n(m).
\end{equation}

\medskip

Choose a maximal
$2\sqrt{\log m}$-separated set
$\{z_1,\ldots,z_{M_m}\}\subset\mathbb G_n^m$. By maximality, the balls
\[
 B_{\Delta_m}(z_j,2\sqrt{\log m}),\qquad1\le j\le M_m,
\]
cover $\mathbb G_n^m$. Translation invariance gives all these balls the
same measure, hence
\[
 1\le
 M_m\,\nu_n^{\otimes m}
 \{u:\Delta_m(u,1)\le2\sqrt{\log m}\}.
\]
Using \eqref{eq:projective-ball-volume} and taking logarithms yields
\eqref{eq:projective-packing-size}; the separation property is precisely
\eqref{eq:projective-packing-correlation}.
}
\end{proof}

\begin{lemma}\label{lem:projective-covering}
Fix $n\ge2$, put $d:=n-1$, and let $\Delta_m$ be the metric in
\eqref{eq:product-projective-distance}. For $0<r\le1$,
\begin{equation}\label{eq:projective-covering-number}
 \mathcal N(\mathbb G_n^m,\Delta_m,r)
 \le
 \left(\frac{C_n\sqrt m}{r}\right)^{dm}.
\end{equation}
\end{lemma}

\begin{proof}
{\color{black}
For a class near the identity of $\mathbb G_n$, choose the representative
\[
 u=(e^{it_1},\ldots,e^{it_{n-1}},1),
 \qquad |t_j|<\pi,
\]
and put $t_n:=0$. {\color{black}\hypersetup{linkcolor=black,citecolor=black}For $t$ sufficiently small, a phase minimizing
$\sum_{j=1}^n|e^{it_j}-e^{i\theta}|^2$ can be chosen near zero.
In this range, $|e^{it_j}-e^{i\theta}|$ is comparable to $|t_j-\theta|$,
uniformly in $j$, with constants depending only on $n$.}
Therefore $\delta(u,1)^2$ is comparable, with constants depending only on
$n$, to
\[
 \min_{\theta\in\mathbb R}\sum_{j=1}^n|t_j-\theta|^2
 =\sum_{j=1}^n|t_j-\bar t|^2,
 \qquad \bar t:=\frac1n\sum_{j=1}^nt_j.
\]
On the section $t_n=0$ this quadratic form is positive definite, so there
are constants $0<c_n<C_n<\infty$ such that
\[
 c_n\sum_{j=1}^{n-1}t_j^2
 \le
 \delta(u,1)^2
 \le
 C_n\sum_{j=1}^{n-1}t_j^2.
\]
By translation, the same comparison holds in a neighbourhood of every
point of $\mathbb G_n$, with constants depending only on $n$.
Since $\mathbb G_n$ is compact, finitely many such neighbourhoods cover it.
Consequently, for $0<\rho\le\rho_n$, each chart admits a $\rho$-net with at
most $C_n\rho^{-d}$ points. Taking the union over the finite covering gives
\[
 |\mathcal N_\rho|
 \le C_n\rho^{-d}.
\]
For $\rho_n<\rho\le1$, one fixed finite $\rho_n$-net is automatically a
$\rho$-net, and
\[
 |\mathcal N_{\rho_n}|
 \le C_n\rho_n^{-d}
 \le C_n'\rho^{-d}.
\]
After increasing the constant,
\begin{equation}\label{eq:projective-one-block-net}
 |\mathcal N_\rho|
 \le\left(\frac{C_n}{\rho}\right)^d,
 \qquad 0<\rho\le1.
\end{equation}

Set
\[
 \rho:=\frac r{\sqrt m}.
\]
For
$z=(z^{(1)},\ldots,z^{(m)})\in\mathbb G_n^m$, choose
$w^{(j)}\in\mathcal N_\rho$ with
$\delta(z^{(j)},w^{(j)})\le\rho$. Then
\[
 \Delta_m(z,w)^2
 =\sum_{j=1}^m\delta(z^{(j)},w^{(j)})^2
 \le m\rho^2=r^2.
\]
Thus $\mathcal N_\rho^m$ is an $r$-net of $\mathbb G_n^m$, and
\[
 \mathcal N(\mathbb G_n^m,\Delta_m,r)
 \le|\mathcal N_\rho|^m
 \overset{\eqref{eq:projective-one-block-net}}{\le}
 \left(\frac{C_n\sqrt m}{r}\right)^{dm}.
\]
}
\end{proof}

Let $g=(g_{i_1,\ldots,i_m})$ have independent standard complex Gaussian
coordinates, and let $\gamma_{m,n}^{\mathrm{ML}}$ denote the corresponding
standard complex Gaussian measure on $\CC^{[n]^m}$. Set
\begin{equation}\label{eq:multilinear-gaussian-process}
 X_m(z)
 :=
 n^{-m/2}
 \sum_{i_1,\ldots,i_m=1}^n
 g_{i_1,\ldots,i_m}
 z^{(1)}_{i_1}\cdots z^{(m)}_{i_m},
 \qquad z\in(\TT^n)^m.
\end{equation}
Then
\[
 \int |X_m(z)|^2\,d\gamma_{m,n}^{\mathrm{ML}}=1,
\]
and
\begin{equation}\label{eq:multilinear-covariance}
 K_m(z,w)
 :=
 \int X_m(z)\overline{X_m(w)}\,d\gamma_{m,n}^{\mathrm{ML}}
 =
 \prod_{r=1}^m
 \left(\frac1n\sum_{j=1}^n
 z_j^{(r)}\overline{w_j^{(r)}}\right).
\end{equation}
If the $r$th block is multiplied by $\lambda_r\in\TT$, then
\[
 X_m(\lambda_1 z^{(1)},\ldots,\lambda_m z^{(m)})
 =(\lambda_1\cdots\lambda_m)X_m(z),
\]
so its modulus is unchanged. Thus $|X_m|$ descends to
$\mathbb G_n^m$. {\color{black}\hypersetup{linkcolor=black,citecolor=black}In particular, when representatives of two points in $\mathbb G_n^m$
are compared, their block phases may be chosen independently without
changing either modulus.} This freedom allows each factor
in \eqref{eq:multilinear-covariance} to be chosen real and nonnegative.

\begin{lemma}\label{lem:dyadic-net-oscillation}
Fix $n\ge2$ and let $X_m$ be the Gaussian process defined in
\eqref{eq:multilinear-gaussian-process}. For $m\ge3$, let $r_0=1/\log m$, and let
$\mathcal N_0$ be an $r_0$-net supplied by
Lemma~\ref{lem:projective-covering}. Then
\begin{equation}\label{eq:dyadic-net-oscillation}
 \Prob\left\{
 \sup_{z\in\mathbb G_n^m}|X_m(z)|>
 \max_{w\in\mathcal N_0}|X_m(w)|
 +C_n r_0\sqrt{m\log m}\right\}\longrightarrow0.
\end{equation}
\end{lemma}

\begin{proof}
{\color{black}
Put
\[
 r_k:=2^{-k}r_0,\qquad k\ge0.
\]

\medskip

Suppose $\Delta_m(z,w)\le r$. Choose the block phases of a representative
$w'$ so that every factor in \eqref{eq:multilinear-covariance} is real and
nonnegative. Then $|X_m(w')|=|X_m(w)|$, and
\begin{align}
 \int|X_m(z)-X_m(w')|^2\,d\gamma_{m,n}^{\mathrm{ML}}
 &=2\bigl(1-|K_m(z,w)|\bigr)\notag\\
 &\le \Delta_m(z,w)^2.
 \label{eq:multilinear-increment-variance}
\end{align}
Indeed, since $0\le\delta(z^{(j)},w^{(j)})^2/2\le1$,
\[
 |K_m(z,w)|
 =\prod_{j=1}^m\left(1-\frac{\delta(z^{(j)},w^{(j)})^2}{2}\right)
 \ge
 1-\frac12\sum_{j=1}^m\delta(z^{(j)},w^{(j)})^2
 =
 1-\frac12\Delta_m(z,w)^2,
\]
which gives \eqref{eq:multilinear-increment-variance}.

\medskip

For each $k\ge0$, choose an $r_k$-net $\mathcal N_k$ satisfying
\eqref{eq:projective-covering-number}. For every
$v\in\mathcal N_{k+1}$ choose $p_k(v)\in\mathcal N_k$ with
\[
 \Delta_m(v,p_k(v))
 \le r_k+r_{k+1}\le2r_k.
\]
Let
\[
 \mathcal E_k:=\{(v,p_k(v)):v\in\mathcal N_{k+1}\}.
\]
Then
\begin{equation}\label{eq:dyadic-edge-count}
 |\mathcal E_k|
 \le
 \left(\frac{C_n\sqrt m}{r_{k+1}}\right)^{dm}.
\end{equation}
For $(v,w)\in\mathcal E_k$,
\eqref{eq:multilinear-increment-variance} gives
\[
 \int|X_m(v)-X_m(w')|^2\,d\gamma_{m,n}^{\mathrm{ML}}\le4r_k^2.
\]
A centered {\color{black}\hypersetup{linkcolor=black,citecolor=black}circularly symmetric} complex Gaussian variable $Z$ with
$\E|Z|^2\le4r_k^2$ satisfies
\[
 \Prob\{|Z|>2r_ku\}\le e^{-u^2}.
\]
Since
$||X_m(v)|-|X_m(w)||\le|X_m(v)-X_m(w')|$,
\begin{equation}\label{eq:dyadic-edge-tail}
 \Prob\left\{
 \big||X_m(v)|-|X_m(w)|\big|>2r_ku
 \right\}
 \le e^{-u^2}.
\end{equation}

Choose
\[
 u_k^2:=
 2dm\log\!\left(\frac{C_n\sqrt m}{r_{k+1}}\right)
 +2(k+1)\log2+2\log m.
\]
Using \eqref{eq:dyadic-edge-count} and
\eqref{eq:dyadic-edge-tail}, the union bound gives
\begin{align*}
 &\Prob\left\{
 \exists k\ge0,\ (v,w)\in\mathcal E_k:
 \big||X_m(v)|-|X_m(w)|\big|>2r_ku_k
 \right\}\\
 &\qquad\le
 \sum_{k\ge0}|\mathcal E_k|e^{-u_k^2}
 \le
 m^{-2}\sum_{k\ge0}2^{-2(k+1)}
 \longrightarrow0.
\end{align*}

\medskip

On the event on which all inequalities
$\big||X_m(v)|-|X_m(w)|\big|\le2r_ku_k$ hold simultaneously for every
$k\ge0$ and every $(v,w)\in\mathcal E_k$, fix
$z\in\mathbb G_n^m$. For $K\ge1$, choose
$v_K\in\mathcal N_K$ with
$\Delta_m(z,v_K)\le r_K$, and define recursively
\[
 v_k:=p_k(v_{k+1}),\qquad k=K-1,\ldots,0.
\]
Then
\[
 \big||X_m(v_K)|-|X_m(v_0)|\big|
 \le2\sum_{k=0}^{K-1}r_ku_k.
\]
Because $r_K\to0$ and $|X_m|$ is continuous, letting $K\to\infty$ gives
\begin{equation}\label{eq:dyadic-telescope}
 |X_m(z)|
 \le
 \max_{w\in\mathcal N_0}|X_m(w)|
 +2\sum_{k\ge0}r_ku_k.
\end{equation}
From the definition of $u_k$ and
$r_{k+1}=2^{-k-1}/\log m$,
\[
 u_k\le
 C_n\sqrt{m(\log m+\log\log m+k+1)}.
\]
Therefore
\begin{align*}
 \sum_{k\ge0}r_ku_k
 &\le
 C_nr_0\sqrt{m\log m}
 \sum_{k\ge0}2^{-k}
 \sqrt{1+\frac{\log\log m+k+1}{\log m}}\\
 &\le C_n'r_0\sqrt{m\log m}.
\end{align*}
Substitution in \eqref{eq:dyadic-telescope} proves
\eqref{eq:dyadic-net-oscillation}.
}
\end{proof}

Projective packing gives many weakly correlated multilinear evaluations,
which provides the lower bound through Slepian's inequality. A multiscale net
and Gaussian increment estimates give the matching upper bound.

\begin{proposition}\label{prop:multilinear-gaussian-supremum}
Let $d=n-1$. For every $\varepsilon>0$,
\begin{equation}\label{eq:multilinear-gaussian-supremum-limit}
 \Prob\left\{
 \left|\frac{\sup_z|X_m(z)|}{\sqrt{\frac d2\,m\log m}}-1\right|>
 \varepsilon\right\}\longrightarrow0.
\end{equation}
\end{proposition}

\begin{proof}
{\color{black}
From \eqref{eq:projective-distance} and
\eqref{eq:multilinear-covariance},
\begin{equation}\label{eq:covariance-decay}
 |K_m(z,w)|
 =
 \prod_{r=1}^m
 \left(1-\frac{\delta(z^{(r)},w^{(r)})^2}{2}\right)
 \le
 \exp\left(-\frac{\Delta_m(z,w)^2}{2}\right).
\end{equation}

\medskip

{\color{black}\hypersetup{linkcolor=black,citecolor=black}Choose the points $z_1,\ldots,z_{M_m}$ from
Lemma~\ref{lem:projective-packing}, fix arbitrary representatives
in $(\mathbb T^n)^m$, and set}
\[
 Y_j:=\sqrt2\,\operatorname{Re}X_m(z_j).
\]
For the standard complex Gaussian coefficients,
\[
 \int X_m(z)X_m(w)\,d\gamma_{m,n}^{\mathrm{ML}}=0,
 \qquad
 \int X_m(z)\overline{X_m(w)}\,d\gamma_{m,n}^{\mathrm{ML}}=K_m(z,w).
\]
Hence
\[
 \int Y_j^2\,d\gamma_{m,n}^{\mathrm{ML}}=1,
 \qquad
 \int Y_iY_j\,d\gamma_{m,n}^{\mathrm{ML}}
 =\operatorname{Re}K_m(z_i,z_j).
\]
By \eqref{eq:projective-packing-correlation} and
\eqref{eq:covariance-decay},
\begin{equation}\label{eq:multilinear-packed-correlation}
 \int Y_iY_j\,d\gamma_{m,n}^{\mathrm{ML}}
 \le |K_m(z_i,z_j)|
 \le m^{-2}
 \qquad(i\ne j).
\end{equation}

Let $Z_0,Z_1,\ldots,Z_{M_m}$ be independent standard real Gaussian
variables and define
\[
 W_j:=m^{-1}Z_0+\sqrt{1-m^{-2}}\,Z_j.
\]
Then
\[
 \E W_j^2=1,
 \qquad
 \E(W_iW_j)=m^{-2}\quad(i\ne j).
\]
Slepian's inequality \cite{Slepian}, together with
\eqref{eq:multilinear-packed-correlation}, gives
\[
 \int\max_jY_j\,d\gamma_{m,n}^{\mathrm{ML}}
 \ge
 \sqrt{1-m^{-2}}\,\E\max_jZ_j.
\]
For independent standard real Gaussians,
\[
 \E\max_{1\le j\le M_m}Z_j
 =(1+o(1))\sqrt{2\log M_m};
\]
see \cite[Section~2.5 and Exercise~2.17]{BLM}. Using
\eqref{eq:projective-packing-size},
\begin{equation}\label{eq:multilinear-packed-mean}
 \int\max_jY_j\,d\gamma_{m,n}^{\mathrm{ML}}
 \ge(1-o(1))\sqrt{dm\log m}.
\end{equation}

{\color{black}\hypersetup{linkcolor=black,citecolor=black}Write $g=(\xi+i\eta)/\sqrt2$, where $\xi$ and $\eta$ are independent
standard real Gaussian vectors. As functions of $(\xi,\eta)$, the
variables $Y_j$ are real linear functionals of Euclidean norm $1$;
hence their maximum is $1$-Lipschitz. Fix $\varepsilon>0$.}
For large $m$, \eqref{eq:multilinear-packed-mean} implies
\[
 \int\max_jY_j\,d\gamma_{m,n}^{\mathrm{ML}}
 \ge(1-\varepsilon/2)\sqrt{dm\log m}.
\]
Gaussian concentration \cite[Section~5.4]{BLM} then gives
\[
 \Prob\left\{
 \max_jY_j<(1-\varepsilon)\sqrt{dm\log m}
 \right\}
 \le e^{-c_\varepsilon m\log m}.
\]
Since
$|X_m(z_j)|\ge Y_j/\sqrt2$,
\begin{equation}\label{eq:multilinear-supremum-lower}
 \Prob\left\{
 \sup_z|X_m(z)|
 <
 (1-\varepsilon)\sqrt{\frac d2\,m\log m}
 \right\}
 \longrightarrow0.
\end{equation}

\medskip

Put
\[
 r_0:=\frac1{\log m}
\]
and choose an $r_0$-net $\mathcal N_0$ from
Lemma~\ref{lem:projective-covering}. Then
\begin{align}
 \log|\mathcal N_0|
 &\le
 dm\log\left(C_n\sqrt m\,\log m\right)\notag\\
 &=
 \frac d2m\log m+dm\log\log m+O_n(m).
 \label{eq:base-net-entropy}
\end{align}
At every fixed $z$,
\[
 \Prob\{|X_m(z)|>u\}=e^{-u^2}.
\]
Hence the union bound gives
\begin{equation}\label{eq:base-net-union}
 \Prob\left\{
 \max_{z\in\mathcal N_0}|X_m(z)|>u
 \right\}
 \le|\mathcal N_0|e^{-u^2}.
\end{equation}
Take
\[
 u=(1+\varepsilon/2)
 \sqrt{\frac d2\,m\log m}.
\]
Using \eqref{eq:base-net-entropy} in
\eqref{eq:base-net-union}, the exponent is
\[
 -\left(\varepsilon+O(\varepsilon^2)\right)
 \frac d2\,m\log m
 +dm\log\log m+O_n(m),
\]
which tends to $-\infty$. Therefore
\begin{equation}\label{eq:base-net-maximum}
 \Prob\left\{
 \max_{z\in\mathcal N_0}|X_m(z)|
 >
 (1+\varepsilon/2)
 \sqrt{\frac d2\,m\log m}
 \right\}
 \longrightarrow0.
\end{equation}
Moreover,
\[
 C_nr_0\sqrt{m\log m}
 =C_n\sqrt{\frac m{\log m}}
 =o(\sqrt{m\log m}).
\]
Thus, for large $m$, the oscillation term in
Lemma~\ref{lem:dyadic-net-oscillation} is at most
\[
 \frac{\varepsilon}{2}
 \sqrt{\frac d2\,m\log m}.
\]
Combining this with \eqref{eq:base-net-maximum} gives
\[
 \Prob\left\{
 \sup_z|X_m(z)|
 >
 (1+\varepsilon)
 \sqrt{\frac d2\,m\log m}
 \right\}
 \longrightarrow0.
\]
Together with \eqref{eq:multilinear-supremum-lower}, this proves
\eqref{eq:multilinear-gaussian-supremum-limit}.
}
\end{proof}

\subsection{Proof of Theorem F}\label{sec:proof-F}

\begin{proof}
{\color{black}
Put
\[
 d:=n-1,\qquad N:=n^m
\]
and let
$g=(g_{i_1,\ldots,i_m})$ have independent standard complex Gaussian
coordinates. Write
\[
 \rho:=\|g\|_2,\qquad \omega:=\frac{g}{\|g\|_2}.
\]
Polar coordinates in $\CC^{[n]^m}$ show that $\omega$ is distributed
according to the normalized surface measure $\sigma_{m,n}$ and is independent
of $\rho$. Since $R_{m,n}^{\mathrm{ML}}(\rho\omega)
=R_{m,n}^{\mathrm{ML}}(\omega)$ for $\rho>0$, every level set of
$R_{m,n}^{\mathrm{ML}}$ has the same Gaussian and spherical probability.
It is therefore enough to work with Gaussian coefficients.

\medskip
For the numerator, set
\[
 S_m(g):=
 \frac1N\sum_{i_1,\ldots,i_m=1}^n
 |g_{i_1,\ldots,i_m}|^{q_m},
\]
and
\[
 c_m:=
 \frac1\pi\int_{\mathbb C}|z|^{q_m}e^{-|z|^2}\,dz
 =
 \Gamma\!\left(1+\frac{q_m}{2}\right).
\]
Since $q_m\to2$, dominated convergence gives $c_m\to1$. Also,
$|z|^{2q_m}\le1+|z|^4$, so
\[
 \frac1\pi\int_{\mathbb C}
 \bigl(|z|^{q_m}-c_m\bigr)^2e^{-|z|^2}\,dz
 \le C
\]
uniformly in $m$. Independence yields
\[
 \int|S_m-c_m|^2\,d\gamma_{m,n}^{\mathrm{ML}}
 \le\frac CN.
\]
Therefore, for every $\varepsilon>0$,
\[
 \Prob\{|S_m-c_m|>\varepsilon\}
 \le\frac{C}{N\varepsilon^2}
 \longrightarrow0.
\]
Thus
\[
 S_m^{1/q_m}\longrightarrow1
 \qquad\text{in Gaussian measure}.
\]
Since
\[
 \frac1{q_m}-\frac12=\frac1{2m}
\]
and $N=n^m$,
\[
 N^{1/q_m-1/2}=n^{1/2}.
\]
Consequently,
\begin{equation}\label{eq:multilinear-numerator-limit}
 \frac{\|g\|_{q_m}}{\sqrt N}
 =
 N^{1/q_m-1/2}S_m^{1/q_m}
 \longrightarrow\sqrt n
 \qquad\text{in Gaussian measure}.
\end{equation}

\medskip
For the denominator, multilinearity allows the supremum to be taken over the torus in each block:
\[
 \|T_g\|
 =
 \max_{z^{(1)},\ldots,z^{(m)}\in\mathbb T^n}
 |T_g(z^{(1)},\ldots,z^{(m)})|.
\]
From \eqref{eq:multilinear-gaussian-process},
\[
 T_g(z^{(1)},\ldots,z^{(m)})
 =\sqrt N\,X_m(z),
\]
and therefore
\begin{equation}\label{eq:multilinear-denominator-identity}
 \frac{\|T_g\|}{\sqrt N}
 =
 \sup_z|X_m(z)|.
\end{equation}
Proposition~\ref{prop:multilinear-gaussian-supremum} gives
\[
 \frac{\|T_g\|}
 {\sqrt N\sqrt{\frac d2\,m\log m}}
 \longrightarrow1
 \qquad\text{in Gaussian measure}.
\]
Combining this with
\eqref{eq:multilinear-numerator-limit} and
$d=n-1$,
\[
 \sqrt{m\log m}\,R_{m,n}^{\mathrm{ML}}(g)
 \longrightarrow
 \frac{\sqrt n}{\sqrt{d/2}}
 =
 \sqrt{\frac{2n}{n-1}}.
\]
The Gaussian convergence therefore transfers to
$\sigma_{m,n}$-measure, proving \eqref{eq:multilinear-main}.
}
\end{proof}

\section{A Sidon consequence}\label{sec:sidon-application}

For background on Sidon sets and quantitative Sidon constants, see for instance \cite{DGMS,HY}.  For $\varnothing\ne\Lambda\subseteq\mathcal M_{m,n}$, put
$N:=|\Lambda|$.  The individual Sidon ratio of
$P_a(z)=\sum_{\alpha\in\Lambda}a_\alpha z^\alpha$ is
\begin{equation}\label{eq:sidon-ratio}
 \operatorname{Sid}_\Lambda(a):=
 \frac{\|a\|_1}{\|P_a\|_\infty}
 \qquad(a\in {\color{black}\CC^\Lambda}\setminus\{0\});
\end{equation}
the Sidon constant of $\Lambda$ is the supremum of
\eqref{eq:sidon-ratio} over $a\ne0$.  The passage from the
Bohnenblust--Hille coefficient norm to $\ell_1$ uses the H\"older factor
$N^{(m-1)/(2m)}$.  Define
\begin{equation}\label{eq:normalized-sidon-ratio}
 \mathsf S_{m,\Lambda}(a):=
 N^{-(m-1)/(2m)}\operatorname{Sid}_\Lambda(a).
\end{equation}

\begin{corollary}\label{cor:typical-sidon}
For every $a\in {\color{black}\CC^\Lambda}\setminus\{0\}$,
\begin{equation}\label{eq:sidon-sandwich}
 N^{-(m-1)/(2m)}
 \le \mathsf S_{m,\Lambda}(a)
 \le R_m(a).
\end{equation}
Consequently, for every fixed $t\ge1$,
\begin{equation}\label{eq:sidon-uniform-tail}
 \lim_{m\to\infty}
 \sup_{n\in\NN}
 \sup_{\varnothing\ne\Lambda\subseteq\mathcal M_{m,n}}
 \mu_\Lambda\bigl(
 \{a\in\mathbb S_\Lambda:\mathsf S_{m,\Lambda}(a)>t\}
 \bigr)=0.
\end{equation}
Moreover, for every sequence
$\varnothing\ne\Lambda_m\subseteq\mathcal M_{m,n_m}$,
\begin{equation}\label{eq:sidon-vanishing-iff}
 \mathsf S_{m,\Lambda_m}\longrightarrow0
 \ \text{in }\mu_{\Lambda_m}\text{-measure}
 \qquad\Longleftrightarrow\qquad
 |\Lambda_m|\longrightarrow\infty.
\end{equation}
\end{corollary}

\begin{proof}
{\color{black}
Put $N:=|\Lambda|$.

Since
\[
 \frac1{q_m}=\frac{m+1}{2m},
\]
H\"older's inequality gives
\[
 \|a\|_1
 \le
 N^{1-1/q_m}\|a\|_{q_m}
 =
 N^{(m-1)/(2m)}\|a\|_{q_m}.
\]
After division by
$N^{(m-1)/(2m)}\|P_a\|_\infty$,
\[
 \mathsf S_{m,\Lambda}(a)\le R_m(a).
\]
On the other hand,
\[
 \|P_a\|_\infty
 \le\sum_{\alpha\in\Lambda}|a_\alpha|
 =\|a\|_1,
\]
and hence
\[
 \mathsf S_{m,\Lambda}(a)
 \ge N^{-(m-1)/(2m)}.
\]
Hence \eqref{eq:sidon-sandwich} holds.

For $t\ge1$,
\[
 \{\mathsf S_{m,\Lambda}>t\}
 \subseteq
 \{R_m>t\}
 \subseteq
 \{R_m>1\},
\]
so Theorem~\ref{thm:tail-main} gives
\eqref{eq:sidon-uniform-tail}.

For \eqref{eq:sidon-vanishing-iff}, first suppose
$|\Lambda_m|\to\infty$. Then
Theorem~\ref{thm:vanishing-main} gives
$R_m\to0$ in $\mu_{\Lambda_m}$-measure. Since
$0\le\mathsf S_{m,\Lambda_m}\le R_m$,
the same is true for $\mathsf S_{m,\Lambda_m}$.

Conversely, suppose $|\Lambda_m|$ does not tend to infinity. Then there are
$K\in\mathbb N$ and a subsequence $(m_j)$ such that
$|\Lambda_{m_j}|\le K$ for every $j$. From \eqref{eq:sidon-sandwich},
\[
 \mathsf S_{m_j,\Lambda_{m_j}}(a)
 \ge
 |\Lambda_{m_j}|^{-(m_j-1)/(2m_j)}
 \ge K^{-1/2}
 \qquad(a\in\mathbb S_{\Lambda_{m_j}}),
\]
because $(m_j-1)/(2m_j)\le1/2$. Hence the normalized Sidon ratio cannot
converge to zero in measure along this subsequence.
}
\end{proof}

\section*{Acknowledgments}

\subsection*{Funding} 

D. Pellegrino is supported by Grants No.~406457/2023-9
(CNPq/MCTI N\textsuperscript{o}~10/2023) and No.~403964/2024-5
(MCTI/CNPq N\textsuperscript{o}~16/2024). He is also supported by Grant
No.~305807/2025-0 from the Conselho Nacional de Desenvolvimento
Cient\'ifico e Tecnol\'ogico (CNPq, Brazil).
E. Teixeira gratefully acknowledges support from the
Grayce B. Kerr Chair at Oklahoma State University.

This work was conducted in part within the DARPA ExpMath project
\emph{A Human-Centered Framework for AI-Mathematician Collaboration
in Research-Level Mathematics}
(Agreement No.~HR0011262E029).
The views and conclusions expressed here are those of the authors
and should not be interpreted as representing the official policies
of the Department of Defense or the U.S.\ Government.

\subsection*{AI assistance disclosure}

The authors conceived the research program, formulated its central ideas, and directed the successive development of its statements and arguments. ChatGPT 5.6 Sol (OpenAI) was used as a collaborative research tool to stress-test ideas and support exploratory analysis, including the examination of parameter ranges and the refinement of estimates. It also assisted with consistency checks, the organization of arguments, literature searches, and routine \LaTeX{} typesetting.

The authors retain intellectual authorship of the work. They wrote and revised the manuscript with this assistance and independently verified all mathematical statements and proofs. They reviewed the references and take full responsibility for the correctness, originality, and integrity of the work.

\subsection*{Data availability}
No datasets were generated or analyzed during the current study.

\subsection*{Competing interests}
The authors declare no competing interests.



\begin{thebibliography}{99}

\bibitem{ADEDGP}
S.~Arunachalam, A.~Dutt, F.~Escudero Guti\'errez, and C.~Palazuelos,
\emph{A cb-Bohnenblust--Hille inequality with constant one and its applications in learning theory},
Math. Ann. \textbf{392} (2025), 3367--3396.
\href{https://doi.org/10.1007/s00208-025-03142-5}{doi:10.1007/s00208-025-03142-5}.

\bibitem{AubrunSzarek}
G.~Aubrun and S.~J.~Szarek,
\emph{Alice and Bob Meet Banach: The Interface of Asymptotic Geometric Analysis and Quantum Information Theory},
Mathematical Surveys and Monographs, vol.~223,
American Mathematical Society, Providence, RI, 2017.

\bibitem{BayartSupports}
F.~Bayart,
\emph{Summability of the coefficients of a multilinear form},
J. Eur. Math. Soc. \textbf{24} (2022), no.~4, 1161--1188.
\href{https://doi.org/10.4171/JEMS/1109}{doi:10.4171/JEMS/1109}.

\bibitem{BPS}
F.~Bayart, D.~Pellegrino, and J.~B.~Seoane-Sep\'ulveda,
\emph{The Bohr radius of the $n$-dimensional polydisk is equivalent
to $\sqrt{(\log n)/n}$},
Adv. Math. \textbf{264} (2014), 726--746.
\href{https://doi.org/10.1016/j.aim.2014.07.029}
{doi:10.1016/j.aim.2014.07.029}.

\bibitem{BCCT}
J.~Bennett, A.~Carbery, M.~Christ, and T.~Tao,
\emph{The Brascamp--Lieb inequalities: finiteness, structure and extremals},
Geom. Funct. Anal. \textbf{17} (2008), no.~5, 1343--1415.
\href{https://doi.org/10.1007/s00039-007-0619-6}
{doi:10.1007/s00039-007-0619-6}.

\bibitem{BH}
H.~F.~Bohnenblust and E.~Hille,
\emph{On the absolute convergence of Dirichlet series},
Ann.\ of Math. (2) \textbf{32} (1931), no.~3, 600--622.
\href{https://doi.org/10.2307/1968255}{doi:10.2307/1968255}.

\bibitem{BLM}
S.~Boucheron, G.~Lugosi, and P.~Massart,
\emph{Concentration Inequalities: A Nonasymptotic Theory of Independence},
Oxford University Press, Oxford, 2013.
\href{https://doi.org/10.1093/acprof:oso/9780199535255.001.0001}
{doi:10.1093/acprof:oso/9780199535255.001.0001}.

\bibitem{CamposRealPolynomial}
J.~R.~Campos, P.~Jim\'enez-Rodr\'iguez, G.~A.~Mu\~noz-Fern\'andez,
D.~Pellegrino, and J.~B.~Seoane-Sep\'ulveda,
\emph{On the real polynomial Bohnenblust--Hille inequality},
Linear Algebra Appl. \textbf{465} (2015), 391--400.
\href{https://doi.org/10.1016/j.laa.2014.09.040}
{doi:10.1016/j.laa.2014.09.040}.

\bibitem{CNS}
N.~Caro-Montoya, D.~N\'u\~nez-Alarc\'on, and D.~Serrano-Rodr\'iguez,
\emph{Asymptotic contractivity of Bohnenblust--Hille constants with bounded monomial support},
Bull. Braz. Math. Soc. New Series \textbf{57} (2026), Art.~40.
\href{https://doi.org/10.1007/s00574-026-00527-1}
{doi:10.1007/s00574-026-00527-1}.

\bibitem{DFOOS}
A.~Defant, L.~Frerick, J.~Ortega-Cerd\`a, M.~Ouna\"ies, and K.~Seip,
\emph{The Bohnenblust--Hille inequality for homogeneous polynomials
is hypercontractive},
Ann.\ of Math. (2) \textbf{174} (2011), no.~1, 485--497.
\href{https://doi.org/10.4007/annals.2011.174.1.13}
{doi:10.4007/annals.2011.174.1.13}.

\bibitem{SupportSensitive}
A.~Defant, D.~Galicer, M.~Mansilla, M.~Masty\l o, and S.~Muro,
\emph{Support-sensitive Bohnenblust--Hille inequalities and local invariants
on Hamming schemes}, preprint, 2026.
\href{https://arxiv.org/abs/2607.05594}{arXiv:2607.05594}.

\bibitem{DGM}
A.~Defant, D.~Garc\'ia, and M.~Maestre,
\emph{Maximum moduli of unimodular polynomials},
J. Korean Math. Soc. \textbf{41} (2004), no.~1, 209--229.
\href{https://doi.org/10.4134/JKMS.2004.41.1.209}
{doi:10.4134/JKMS.2004.41.1.209}.

\bibitem{DGMS}
A.~Defant, D.~Garc\'ia, M.~Maestre, and P.~Sevilla-Peris,
\emph{Dirichlet Series and Holomorphic Functions in High Dimensions},
New Mathematical Monographs, vol.~37,
Cambridge University Press, Cambridge, 2019.
\href{https://doi.org/10.1017/9781108691611}
{doi:10.1017/9781108691611}.

\bibitem{EskenazisIvanisvili}
A.~Eskenazis and P.~Ivanisvili,
\emph{Learning low-degree functions from a logarithmic number of random queries},
Proceedings of the 54th Annual ACM SIGACT Symposium on Theory of Computing
(STOC 2022), 203--207.
\href{https://arxiv.org/abs/2109.10162}{arXiv:2109.10162}.

\bibitem{Folland}
G.~B.~Folland,
\emph{Real Analysis: Modern Techniques and Their Applications}, second edition,
Pure and Applied Mathematics (New York),
John Wiley \& Sons, New York, 1999.

\bibitem{HY}
K.~E.~Hare and R.~(Xu)~Yang,
\emph{Sidon sets are proportionally Sidon with small Sidon constants},
Canad. Math. Bull. \textbf{62} (2019), no.~4, 798--809.
\href{https://doi.org/10.4153/S0008439518000620}
{doi:10.4153/S0008439518000620}.

\bibitem{IvanisviliHamming}
P.~Ivanisvili,
\emph{Polynomial growth of Bohnenblust--Hille constants on the Hamming cube},
preprint, 2026.
\href{https://arxiv.org/abs/2609.12427}{arXiv:2609.12427}.

\bibitem{Littlewood}
J.~E.~Littlewood,
\emph{On bounded bilinear forms in an infinite number of variables},
Quart. J. Math. Oxford Ser. \textbf{1} (1930), 164--174.
\href{https://doi.org/10.1093/qmath/os-1.1.164}{doi:10.1093/qmath/os-1.1.164}.

\bibitem{MNP}
M.~Maia, T.~Nogueira, and D.~Pellegrino,
\emph{The Bohnenblust--Hille inequality for polynomials whose monomials
have a uniformly bounded number of variables},
Integral Equations Operator Theory \textbf{88} (2017), no.~1, 143--149.
\href{https://doi.org/10.1007/s00020-017-2372-z}{doi:10.1007/s00020-017-2372-z}.

\bibitem{Montanaro}
A.~Montanaro,
\emph{Some applications of hypercontractive inequalities in quantum information theory},
J. Math. Phys. \textbf{53} (2012), 122206.
\href{https://doi.org/10.1063/1.4769269}{doi:10.1063/1.4769269}.

\bibitem{NIST}
F.~W.~J.~Olver, D.~W.~Lozier, R.~F.~Boisvert, and C.~W.~Clark (eds.),
\emph{NIST Handbook of Mathematical Functions},
Cambridge University Press, New York, 2010.
\href{https://dlmf.nist.gov/5.11}{DLMF, Section~5.11}.

\bibitem{PT}
D.~M.~Pellegrino and E.~V.~Teixeira,
\emph{Polynomial growth of complex polynomial Bohnenblust--Hille constants},
preprint, 2026.
\href{https://arxiv.org/abs/2608.16584}{arXiv:2608.16584}.

\bibitem{Slepian}
D.~Slepian,
\emph{The one-sided barrier problem for Gaussian noise},
Bell System Tech. J. \textbf{41} (1962), no.~2, 463--501.
\href{https://doi.org/10.1002/j.1538-7305.1962.tb02419.x}
{doi:10.1002/j.1538-7305.1962.tb02419.x}.

\bibitem{STVBoolean}
J.~Slote, C.-K.~Tseng, and A.~Volberg,
\emph{An $m^{2.943}$ Bohnenblust--Hille bound on the Boolean cube},
preprint, 2026.
\href{https://arxiv.org/abs/2609.21144}{arXiv:2609.21144}.

\bibitem{SVProductCyclic}
J.~Slote and A.~Volberg,
\emph{Polynomial Bohnenblust--Hille bounds for product of cyclic groups},
preprint, 2026.
\href{https://arxiv.org/abs/2609.07758}{arXiv:2609.07758}.

\bibitem{SVZ}
J.~Slote, A.~Volberg, and H.~Zhang,
\emph{Bohnenblust--Hille inequality for cyclic groups},
Adv. Math. \textbf{452} (2024), Paper No.~109824.
\href{https://doi.org/10.1016/j.aim.2024.109824}{doi:10.1016/j.aim.2024.109824}.

\bibitem{VZ}
A.~Volberg and H.~Zhang,
\emph{Noncommutative Bohnenblust--Hille inequalities},
Math. Ann. \textbf{389} (2024), 1657--1676.
\href{https://doi.org/10.1007/s00208-023-02680-0}{doi:10.1007/s00208-023-02680-0}.

\end{thebibliography}
\end{document}